\documentclass[11pt]{amsart}
\usepackage{geometry}                
\usepackage{graphicx}
\usepackage{subcaption}
\usepackage{amssymb}
\usepackage{epstopdf}
\usepackage{geometry} 
\usepackage{graphicx}
\usepackage{amsmath} 
\numberwithin{equation}{section}
\usepackage{commath}
\usepackage{changepage}
\usepackage{xfrac}
\usepackage{cleveref}
\usepackage[normalem]{ulem}
\usepackage{bbm}
\usepackage{float}
\usepackage{url}
\usepackage{mathtools}
\usepackage{amsthm} 
\usepackage[disable]{todonotes} 
\DeclareMathOperator*{\esssup}{ess\,sup}

\newcommand{\inisp}{\mathcal{S}^\epsilon}

\newcommand{\inise}{\mathcal{S}_d^\epsilon}

\theoremstyle{plain}
\newtheorem{theorem}{Theorem}[section]
\newtheorem*{theorem*}{Theorem}
\newtheorem{lemma}[theorem]{Lemma}

\newenvironment{claim}[1]{\par\noindent\underline{Claim.}\space#1}{}
\newenvironment{claimproof}[1]{\par\noindent\emph{Proof of Claim.}\space#1}{\leavevmode\unskip\penalty9999 \hbox{}\nobreak\hfill\quad\hbox{$\blacksquare$}}
\theoremstyle{remark}
\newtheorem{remark}{Remark}
\theoremstyle{definition}
\newtheorem{definition}[theorem]{Definition}
\title[$L^2$-type a posteriori error analysis]{Theory of shifts, shocks, and the intimate connections to $L^2$-type a posteriori error analysis of numerical schemes for hyperbolic problems}
\author[Giesselmann]{Jan Giesselmann}
\address[Jan Giesselmann]{\newline Department of Mathematics\\ Technische Universität Darmstadt \newline Dolivostraße 15 \newline 64293 Darmstadt, Germany}
\email{giesselmann@mathematik.tu-darmstadt.de}
\author[Krupa]{Sam G.  Krupa}
\address[Sam G. Krupa]{\newline Max Planck Institute for Mathematics in the Sciences \newline Inselstraße 22
\newline 04103 Leipzig, Germany}
\email{Sam.Krupa@mis.mpg.de}

\begin{document}
\date{\today}

\thanks{\textbf{Acknowledgment.} 
This material is based upon work supported by the National Science Foundation under Grant No. DMS-1926686.
The authors are grateful for financial support by the German Science Foundation (DFG) via grant TRR~154 (\emph{Mathematical modelling, simulation and optimization using the example of gas networks}), project C05.
}
\begin{abstract}
In this paper, we develop reliable a posteriori error estimates for numerical approximations of scalar hyperbolic conservation laws in one space dimension. 

  Our methods have no inherent small-data limitations and
   are a step towards error control of numerical schemes for systems.  We are careful not to appeal to the Kruzhkov theory for scalar conservation laws.  Instead,
   we derive novel quantitative stability estimates that extend the theory of shifts, and in particular, the framework for proving stability first developed by the second author and Vasseur. This is the first time this methodology has been used for quantitative estimates. 
   
   We work entirely within the context of the theory of shifts and $a$-contraction, techniques which adapt well to systems.  In fact, the stability framework by the second author and Vasseur  has itself recently been pushed to systems [Chen-Krupa-Vasseur. Uniqueness and weak-{BV} stability for {$2\times 2$} conservation
  laws. {\em Arch. Ration. Mech. Anal.}, 246(1):299--332, 2022].

  Our theoretical findings are complemented by a numerical implementation in MATLAB and numerical experiments.

\end{abstract}
\maketitle
\tableofcontents

\section{Introduction}
In this paper, we investigate error estimates for numerical approximations of the scalar conservation law in one space dimension: 

\begin{align}
\label{system}
\begin{cases}
\partial_t u + \partial_x A(u)=0,\mbox{ for } x\in\mathbb{R},\mbox{ } t>0,\\
u(x,0)={u}^0(x)
\end{cases}
\end{align}

where $u\colon\mathbb{R}\times[0,T)\to \mathbb{R}$ is the unknown (for $T>0$ fixed), $u^0\in L^\infty(\mathbb{R})$ is the given initial data, and $A\colon\mathbb{R}\to\mathbb{R}$ is the given \emph{flux function}. In this paper, we only consider $A\in C^2(\mathbb{R})$ which are strictly convex. A \emph{classical} solution to \eqref{system} is a locally Lipschitz function $u\colon\mathbb{R}\times[0,T)\to \mathbb{R}$ which satisfies  $\partial_t u + \partial_x A(u)=0$ almost everywhere and verifies $u(x,0)={u}^0(x)$ for all $x\in\mathbb{R}.$ We are also concerned with \emph{weak solutions} to \eqref{system}, i.e.  locally bounded and measurable functions $u\colon\mathbb{R}\times[0,T)\to \mathbb{R}$ which verify
\begin{align}\label{def_weak}
\int\limits_0^\infty\int\limits_{-\infty}^\infty \partial_t \phi u+\partial_x \phi A(u)\,dxdt +\int\limits_{-\infty}^\infty \phi(x,0)u^0(x)\,dx=0
\end{align}
for every Lipschitz-continuous function $\phi\colon\mathbb{R}\times[0,T)\to \mathbb{R}$ with compact support. In particular, every weak solution is a distributional solution to \eqref{system}. Remark also that every classical solution is also a weak solution.

\vspace{.15in}

A pair of functions $\eta,q\colon\mathbb{R}\to\mathbb{R}$ are called an \emph{entropy} and \emph{entropy-flux}, respectively, for the system \eqref{system} if
\begin{align}\label{compat_cond}
    q'(u)=\eta'(u) A'(u).
\end{align}

We say a weak solution to \eqref{system} is \emph{entropic} for the entropy $\eta$ if it satisfies the \emph{entropy inequality}
\begin{align}
    \partial_t \eta(u)+\partial_x q(u)\leq 0
\end{align}
in the sense of distributions, where $q$ is any corresponding entropy-flux. More precisely, 
\begin{align}\label{entropy_integral}
\int\limits_0^\infty\int\limits_{-\infty}^\infty \partial_t \phi \eta(u)+\partial_x \phi q(u)\,dxdt +\int\limits_{-\infty}^\infty \phi(x,0)\eta(u^0(x))\,dx\geq 0
\end{align}
for every nonnegative Lipschitz-continuous function $\phi\colon\mathbb{R}\times[0,T)\to \mathbb{R}$ with compact support

We have in mind future applications to hyperbolic systems with multiple conserved quantities.  Hyperbolic conservation laws are an important class of models in many physical and engineering applications. One example is the description of non-viscous compressible flows by the Euler equations.  One particular feature of hyperbolic conservation laws that makes their analysis and the construction of reliable numerical schemes challenging is that they, in general, only have smooth solutions up to some finite time. 
For an overview on numerical  schemes for hyperbolic conservation laws and the available convergence analysis we refer to \cite{GR96,Kroner97,LeVeque02,Hesthaven18} and references therein.

In his 2023 high-level survey of the state of conservation laws, Dafermos writes ``Great progress should also be expected in the art and science of scientific computation of solutions [...] On the theoretical side, the supporting numerical analysis of the algorithms employed in the solution of systems is based to a great extent on the theory of the scalar conservation law. The expectation is that progress will be made by bringing closer the numerical and the theoretical analysis of systems'' \cite[p.~485]{MR4627977}. The present work is a step in the program of Dafermos.

The goal of the  work at hand is to provide a first step towards  reliable a posteriori error estimates for Runge-Kutta discontinuous Galerkin schemes for systems of hyperbolic conservation laws in one space dimension. To this end, we provide such estimates for scalar problems without using the Kruzhkov stability theory for scalar conservation laws. 
Instead, we develop an extension, which allows for quantitative estimates, of the theories of shifts and $a$-contraction (see \Cref{sec:stable}, below). Due to being non-perturbative, these techniques can handle large data. Moreover, they require only a single entropy condition, and hence apply to both scalar equations as well as systems. Thus, there is hope that the techniques developed in this paper will extend to systems. This is in contrast to the results in \cite{KO00,DMO07} that use techniques that are systematically limited to scalar problems.

\subsection{Theory of stability}\label{sec:stable}
For nonlinear hyperbolic conservation laws, analysis of systems is much more involved than the analysis of scalar problems. The main difference is that for scalar problems there exist infinitely many entropies that gives rise to a rather restrictive class of entropy (weak) solutions for which existence and uniqueness of solutions as well as an $L^1$ contraction property can be shown. In contrast, systems are often only endowed with one entropy/entropy-flux pair due to the compatibility condition \eqref{compat_cond} being overdetermined in the systems case (for integrability conditions on the entropy, see \cite[p.~13-14]{dafermos_big_book} and \cite[p.~54-55]{dafermos_big_book}).  For many physical systems only one entropy exists; therefore existence and uniqueness of entropy solutions are delicate issues. 

Indeed,  for systems in multiple space dimensions it can be shown via the technique of convex integration that for some set of initial data entropy solutions are not unique \cite{MR3352460,MR2564474}. On the other hand,  in one space dimension Bressan and coworkers \cite{Bressan2000,MR1686652,MR1337114} have shown that for small-$BV$ initial data global entropy solutions exist and satisfy $L^1$ stability estimates in the class of small-$BV$ solutions.

Given a classical solution $\bar{u}$ to a conservation law, and a weak solution $u$ which is entropic for \emph{at least one strictly convex entropy}, the weak-strong stability theory of Dafermos \cite{MR546634,doi:10.1080/01495737908962394} and DiPerna \cite{MR523630} gives $L^2$-type stability estimates between $u$ and $\bar{u}$. Expanding on the theory of relative entropy, the theories of shifts \cite{VASSEUR2008323} and $a$-contraction \cite{MR3519973} have allowed for the method of relative entropy to consider the stability of solutions from a larger class. Recently, the theories of shifts and $a$-contraction have become sufficiently developed to allow for the stability of the well-known small-$BV$ solutions \cite{Bressan2000,MR1686652} to be studied among solutions from the much larger class of solutions taking possibly very large values in $L^2$ \cite{chen2020uniqueness}. However, quantitative stability estimates have up to now been out of the reach of the theory. The theories of shifts and $a$-contraction will be explained in more detail in \Cref{describe_theory}.

\subsection{State of the art in numerical analysis}
In the convergence analysis of numerical schemes, two main classes of error estimates need to be distinguished: \textit{a priori} and \textit{a posteriori} estimates. A priori estimates aim at providing bounds for errors of numerical schemes that only depend on mesh width and features of the exact solution. In contrast, a posteriori estimates require computing the numerical solution but are independent of (usually unknown) quantities such as norms of exact solutions. In a posteriori estimates, it is desirable that the error estimate is \textit{reliable} which means that the estimator is an upper bound of the error (i.e. of some norm of the difference of exact and numerical solution) and \textit{efficient} which means that some (fixed) multiple of the error estimator is also a lower bound for the error. In the context of numerical schemes for hyperbolic conservation laws efficiency of estimators seems out of reach and we will settle for a weaker property -- see the discussion in Section \ref{sec:op-ed}.

There are also some ``intermediate'' classes of estimates where \textit{ a priori} convergence is guaranteed if some criterion that can only be verified \textit{a posteriori}, i.e. once the numerical solution has been computed, is satisfied. Indeed, interesting recent work by Bressan-Chiri-Shen \cite{Bressan2020} shows that if numerical solutions from a rather general class of numerical schemes satisfy an a posteriori stability condition and have  small BV then they converge in $L^\infty(0,T; L^1(\mathbb{R}))$ to the exact solution. Indeed, the results in \cite{Bressan2020} provide error bounds that are related to the consistency error in the $L^1(0,T;W^{-1,1}(\mathbb{R})) $ norm. The Bressan-Chiri-Shen work \cite{Bressan2020} is limited to small BV due to its reliance on the classical theory, e.g. \cite{MR1337114}. Instead of providing a priori bounds for these consistency errors for specific schemes (as was done in \cite{Bressan2020}) one can also compute them a posteriori (for a wider variety of numerical schemes) -- which is being done in the recent preprint \cite{GiesselmannSikstel2023}. For a comparison of convergence rates of error estimators see \Cref{sec:op-ed}.

The first systematic a posteriori analysis for numerical approximations of scalar conservation laws, accompanied with corresponding adaptive algorithms, can be traced back to \cite{KO00,GM00} (see also \cite{Coc03,DMO07} and references therein).  These estimates were derived by employing Kruzhkov’s doubling of variables technique using an infinite family of entropies \cite{MR0267257}.  A posteriori results for systems were derived in \cite{Laforest04} for front tracking schemes. However, these results are limited to solutions with small BV due to limitations in the $L^1$ front tracking theory \cite{MR1723032,Bressan2000}. For Glimm's scheme, a posteriori results are given in \cite{Laforest08}, albeit restricted to Burgers equation with non-increasing initial data. Goal oriented estimates, based on dual weighted residuals, have been investigated in \cite{HartmannHouston2003}. However, the goal oriented estimates require the solution of dual problems that consist of advection equations with discontinuous velocity fields. Their well-posedness theory is only understood for scalar problems in one space dimension \cite{Ulbrich2002}. Indeed,  for discontinuous solutions to systems of conservation laws the well-posedness of dual problems is unclear and, therefore, dual weighted residual-type estimates cannot be proven to be reliable.

We would like to mention recent a priori  convergence results on  numerical schemes for systems of hyperbolic conservation laws: In \cite{FLM20} convergence of finite volume schemes to so-called dissipative solutions is shown while in \cite{FLMW20} convergence to statistical solutions is proven. Both results are based on compactness arguments so that no quantitative information on the size of the error can be extracted. Indeed, quantitative information on the size of the error seems too much to ask for in this context since there is no uniqueness of solutions.

\subsection{Our stratagem}

Our approach follows the general ideas of reconstruction-based a posteriori error estimates as outlined in \cite{Mak07}.
The fundamental idea is to compute a reconstruction of the numerical solution that satisfies all conditions necessary so that its difference to the exact solution can be bounded in terms of the residual using some appropriate stability theory (that is the theory of shifts and $a$-contraction in our case). By ``residual'' we refer to the non-zero right-hand side of \eqref{system} obtained as ``error'' when the reconstructed solution is inserted into the PDE.
A similar approach, but using the  original relative entropy estimates of Dafermos and DiPerna as stability theory, was used for hyperbolic systems in \cite{DednerGiesselmann,GiesselmannMakridakisPryer}. \emph{The present paper improves upon the results in \cite{DednerGiesselmann,GiesselmannMakridakisPryer} as it provides an estimator that is convergent post shock formation.} This is intimately linked with the ability of the theory of shifts to consider solutions from a broader class than the original weak-strong theory. In particular, the theory of shifts is not restricted to studying the stability of a smooth solution $\bar{u}$ among the class of weak solutions $u$, but can in fact allow for arbitrarily many shocks to exist in the otherwise smooth solution $\bar{u}$.

As discussed above, the theory of shifts and $a$-contraction are well-known to have no small-data limitations, and in particular are able to study large shocks (see e.g \cite{Leger2011,MR3519973}). Thus, the goal of future work is to push forward the ideas and results in the present paper to systems. In this way, we hope to study a posteriori stability of large-data solutions to hyperbolic systems.


Using the theory of shifts comes with some cost though.
As mentioned above, the theory of shifts allows us to generalize the relative entropy theory to allow for shocks in an otherwise smooth solution $\bar{u}$. Thus, the core of our proof involves piecewise-smooth approximate solutions. However, the theory of shifts must allow for the discontinuities in $\bar{u}$ to move in largely unpredictable ways, not necessarily following the Rankine-Hugoniot jump condition (see \Cref{describe_theory} for more on this). The cost is: due to the uncertainty in shock position, we are required to numerically simulate extensions for each of the continuous pieces of $\bar{u}$. These extensions are global in space and time.

As a consequence, we need to solve not one initial value problem numerically but as many as there are discontinuities in the initial data and we need to track approximate positions and position error bounds for each discontinuity so that we know which numerical solution might be revealed where and when, in comparison with the exact solution. 
This might seem like a huge computational burden severely increasing computational cost but since all the individual solutions will remain smooth we may use high order numerical schemes without stabilization mechanisms which mitigates the increase in computational costs. 
For each of the numerical solutions we will use reconstructions as described in \cite{DednerGiesselmann,GiesselmannMakridakisPryer} since (for certain classes of numerical fluxes) they lead to residuals that scale optimally with mesh width for smooth solutions.

Indeed, for initial data that consists of smooth, increasing pieces with down jumps in between we obtain a posteriori error estimates where the estimators are $\mathcal{O}(h)$ for first order finite volume schemes on meshes with mesh width $h$. Then, let us for a moment consider numerical schemes on meshes that move such that any nondifferentiable points in our solution are always aligned with some cell boundary of the mesh.
Then, for initial data that also contain (smooth) decreasing pieces our error estimators are $\mathcal{O}(h^{1/3})$. For more details, see \Cref{sec:op-ed}.
Known a posteriori error estimators (for first order finite volume schemes ) are of order $\mathcal{O}(h^{1/2})$ for scalar problems \cite{KO00}. Remark that for systems in one space dimension, Bressan-Chiri-Shen gives an a posteriori error estimator of order $\mathcal{O}(h^{1/3})$ \cite{Bressan2020}. 

\subsection{Plan for the paper}

The remainder of this paper is organized as follows: in \Cref{sec:statement} we introduce basic concepts and present our main results. In \Cref{describe_theory}, we give a brief introduction to the key ideas behind the theories of shifts and $a$-contraction, and we sketch how we will use them to derive quantitative estimates. In \Cref{sec:technical}, we give some technical lemmas we will need. \Cref{construction_section} is devoted to the description of our numerical scheme and the construction of the key ancillary function $\hat\psi$. In \Cref{algo_section} we describe our error estimator. In \Cref{Revelation_section}, we collect the key calculations used in our error estimator (including the proof of our essential estimate \eqref{control_l21206}, below). In \Cref{sec:proof_conv_thm}, we prove \Cref{conv_thm}. \Cref{sec:op-ed} is devoted to heuristics of the expected convergence behavior of our error estimator when using dG schemes with polynomials of degree $0$ and an upwind type numerical flux. Finally, in \Cref{sec:numerical_experiments} we introduce the MATLAB implementation of our numerical scheme and associated error estimator and give the results of numerical experiments.

\section{Statement of result and basic concepts}\label{sec:statement}
We are concerned with approximate solutions $\hat{u}$ to our conservation law, i.e.  $\hat{u}\colon\mathbb{R}\times[0,T)\to \mathbb{R}$ for some $T>0$ that are obtained by gluing together approximate solutions that are in turn obtained from Lipschitz-continuous reconstructions of numerical solutions (analogously to what was done in \cite{GiesselmannMakridakisPryer,DednerGiesselmann}). The gluing process allows $\hat{u}$ to have downward jumps and this is why  we are able to obtain much stronger results than those that were obtained in \cite{GiesselmannMakridakisPryer,DednerGiesselmann}.

Throughout this paper, we will consider initial data in the following two forms. The first form is the one required for the exact initial data
\begin{equation}\label{PL}
\begin{aligned}
\inisp:= 
\{ v\in &L^\infty (\mathbb{R}) \, : \, \exists N\in \mathbb{N},  -\infty=x_0<x_1<\cdots<x_N<x_{N+1}=+\infty  \, : \,
\text{ for } i=0,\ldots,N, \\
&\text{$v'(x)$ exists and is uniformly continuous for $x\in (x_i,x_{i+1})$, $v(x_i -)\geq v(x_i +)$} \\
&\text{ (when $i\neq 0$), and exactly one of two cases hold: either for all } x\in (x_i,x_{i+1}),
\\
&\text{we have } \partial_x v|_{(x_i,x_{i+1})}>-\epsilon \text{ or for all } x\in (x_i,x_{i+1}) \text{, we have } \partial_x v|_{(x_i,x_{i+1})}\leq -\epsilon. 
\\
&\text{Lastly, if $v|_{(x_i,x_{i+1})}>-\epsilon$ and $v|_{(x_{i+1},x_{i+2})}>-\epsilon$,}
\\
&\text{then  $v(x_{i+1} -) > v(x_{i+1} +)$ (for $i\neq N$)}  \},
\end{aligned}
\end{equation}
where $\epsilon >0$ is a number that is fixed from now on and needs to be chosen such that up to a (fixed) time $T>0$ no slopes form which are less than $-1$, from any initial data with slope $> - \epsilon$. Notice there are no smallness assumptions for the class of functions $\inisp$.

For a function $v\in\inisp$, we will call it \emph{rapidly decreasing} on intervals $(x_i,x_{i+1})$ where $\partial_x v|_{(x_i,x_{i+1})}\leq -\epsilon$ and \emph{nearly non-decreasing} on intervals where $\partial_x v|_{(x_i,x_{i+1})}>-\epsilon$. See \Cref{shocks_fig}. 

Remark that because $\inisp \subset L^\infty(\mathbb{R})$, we must have that for $v\in\inisp$, $v$ is nearly non-decreasing on the intervals $(x_0,x_1)$ and $(x_N,x_{N+1})$ (where the $x_0<x_1<\cdots<x_{N+1}$ are in the context of \eqref{PL}).

The second form is what we require for the initial data of the reconstructed numerical solution.
\begin{equation}
\begin{aligned}\label{LIWASe}
\inise:= 
\{ v\in L^\infty &(\mathbb{R}) \, : \, \exists N\in \mathbb{N},  -\infty=x_0<x_1<\cdots<x_N<x_{N+1}=+\infty  \, : \,
\text{ for } i=0,\ldots,N, \\ 
&\text{$v'(x)$ exists and is uniformly continuous for $x\in (x_i,x_{i+1})$,}
\\
&\partial_x v|_{(x_i,x_{i+1})}  > -\epsilon, \text{ and }
v(x_i -)>v(x_i +) \}
\end{aligned}
\end{equation}

This is closely related to the function space that is considered in \cite{2017arXiv170905610K} for sequences of approximate solutions.

Note that $\inisp$ is a larger space than $\inise$ since some of the pieces are allowed to decrease strongly. Note also that functions in $\inise$ are dense in $\inisp$ (in the $L^p_{\text{loc}}$ topology, $1\leq p<\infty$). This can be seen by approximating the rapidly decreasing parts by piecewise-constant functions.

Following the general theory of shifts and $a$-contraction, we study solutions to \eqref{system} among the class of functions verifying a Strong Trace Property (first introduced in \cite{Leger2011}):

\begin{definition}[Strong Trace Property]\label{strong_trace_definition}
Fix $T>0$. Let $u\colon\mathbb{R}\times[0,T)\to\mathbb{R}$ verify $u\in L^\infty(\mathbb{R}\times[0,T))$. We say $u$ has the \emph{Strong Trace Property} if for every fixed Lipschitz-continuous map $h\colon [0,T)\to\mathbb{R}$, there exists $u_+,u_-\colon[0,T)\to\mathbb{R}$ such that
\begin{align}
\lim_{n\to\infty}\int\limits_0^{t_0}\esssup_{y\in(0,\frac{1}{n})}\abs{{u}(h(t)+y,t)-u_+(t)}\,dt=\lim_{n\to\infty}\int\limits_0^{t_0}\esssup_{y\in(-\frac{1}{n},0)}\abs{{u}(h(t)+y,t)-u_-(t)}\,dt=0
\end{align}
for all $t_0\in(0,T)$.
\end{definition}

E.g., a function $u\in L^\infty(\mathbb{R}\times[0,T))$ satisfies the Strong Trace Property if, for each fixed $h$, the right and left limits
\begin{align}
\lim_{y\to0^{+}}u(h(t)+y,t)\hspace{.7in}\text{and}\hspace{.7in}\lim_{y\to0^{-}}u(h(t)+y,t)
\end{align}
exist for almost every $t$. Specifically, a function $u\in L^\infty(\mathbb{R}\times[0,T))$ has the strong traces according to \Cref{strong_trace_definition} if $u$ has a representative that is in $BV_{\text{loc}}$. However, the Strong Trace Property is weaker than $BV_{\text{loc}}$. Remark that if a weak solution $u$ to \eqref{system} verifies the entropy inequality \eqref{entropy_integral} for \emph{every} convex entropy function $\eta$ (and associated entropy-flux $q$), then classically it is known that $u(\cdot,t)$ instantaneously regularizes to $BV_{\text{loc}}$ for every $t>0$ (see e.g. \cite{dafermos_big_book,serre_book}), and thus the Strong Trace Property is verified. However, \emph{nota bene}: in the present paper we are assuming very little on weak solutions $u$ to \eqref{system}, in particular only verifying \eqref{entropy_integral} for one convex entropy, in analogy with the systems case, and thus \emph{we do not have the classical theory available to us}.

We can now summarize some of our main results.
\begin{theorem}[Main Theorem]\label{main_theorem}
 Fix $T>0$. Then choose $\epsilon>0$ such that smooth solutions to \eqref{system} with initial data whose slope is at least $-\epsilon$ do not form slopes less than -1 (and in particular do not form discontinuities) on the time interval $[0,T]$. 
 
 Let $\bar{u}$ be an exact solution to our conservation law \eqref{system} with initial data $\bar{u}^0 \in \inisp$. Assume $\bar{u}$ verifies the entropy inequality \eqref{entropy_integral} for the strictly convex entropy $\eta\in C^2(\mathbb{R})$. Assume also that $\bar{u}$ verifies the Strong Trace Property (\Cref{strong_trace_definition}). 
 
 Recalling that $\bar{u}^0 \in \inisp$, let $\bar N \in \mathbb{N} $ be the minimum integer required within the context of the definition of the space $\inisp$ (this prevents us from needlessly breaking the intervals $(x_i,x_{i+1})$ into smaller intervals). Then, let $x_1,...,x_{\bar N}$ denote the points from within the context of the definition of the space $\inisp$.
 
 Fix a (sufficiently small) parameter $\delta>0$. Then, there exists $N^\delta \in \mathbb{N}$ and Lipschitz-continuous functions
 $v^\delta_1, ..., v^\delta_{N^\delta+1}\colon  \mathbb{R} \times  [0,T] \rightarrow  \mathbb{R}$ that verify $\partial_t v^\delta_i+\partial_x A(v^\delta_i)=0$ (for $i=1,\ldots,N^\delta+1$) and verify the ordering condition $ v^\delta_i(x,0) - v^\delta_{i+1}(x,0)\geq \frac{1}{2} \epsilon\delta>0 \ \forall i \text{ and } \forall x$.

 Then, let the $v^\delta_1, ..., v^\delta_{N+1}$ be simulated using a numerical scheme, giving reconstructed Lipschitz-continuous functions $\hat v^\delta_1, ..., \hat v^\delta_{N^\delta+1}\colon  \mathbb{R} \times  [0,T] \rightarrow  \mathbb{R}$. We assume the $\hat v^\delta_1, ..., \hat v^\delta_{N^\delta+1}$ also satisfy the ordering condition 
 \begin{align}\label{ordering_main_thm}
 \hat v^\delta_i(x,t) > \hat v^\delta_{i+1}(x,t)\  \forall i \text{ and }\forall (x,t),
 \end{align}
 which can be checked a posteriori.

 There also exist Lipschitz-continuous curves $\hat h^\delta_1, ...,\hat h^\delta_{N^\delta}\colon [ 0,T] \rightarrow \mathbb{R}$ that satisfy the ordering condition $\hat h^\delta_i (t) \leq \hat h^\delta_{i+1}(t) \, \forall i \text{ and }\forall t$ and also verify
 \[ \{ x_1, \ldots,  x_{\bar N}\} \subset \{ \hat h^\delta_1(0), \ldots, \hat h^\delta_{N^\delta}(0)\}.\]

Then, we define the numerical solution to our conservation law,

 \begin{align}
\hat{u}^\delta(x,t)\coloneqq
    \begin{cases}
    \hat{v}^\delta_1(x,t) &\text{if } x<\hat{h}^\delta_1(t),\\
    \hat{v}^\delta_2(x,t) &\text{if } \hat{h}^\delta_1(t)<x<\hat{h}^\delta_2(t),\\
    &\vdots\\
    \hat{v}^\delta_{N^\delta+1}(x,t) &\text{if } \hat{h}^\delta_{N^\delta}(t)<x.\\
    \end{cases}
\end{align}

We similarly define the ancillary function ${\hat{\psi}^\delta}\colon\mathbb{R}\times[0,T]\to\mathbb{R}$,
\begin{align}
{\hat{\psi}}^\delta(x,t)\coloneqq
    \begin{cases}
    \hat{v}^\delta_1(x,t) &\text{if } x<{h}^\delta_1(t),\\
    \hat{v}^\delta_2(x,t) &\text{if } {h}^\delta_1(t)<x<{h}^\delta_2(t),\\
    &\vdots\\
    \hat{v}^\delta_{N^\delta+1}(x,t) &\text{if } {h}^\delta_{N^\delta}(t)<x,\\
    \end{cases}
\end{align}
for Lipschitz-continuous curves $h^\delta_1, ...,h^\delta_{N^\delta}\colon [ 0,T] \rightarrow \mathbb{R}$ that satisfy the ordering condition $h^\delta_i (t) \leq h^\delta_{i+1}(t) \, \forall i \text{ and }\forall t$. Furthermore, $\hat{h}^\delta_i(0)=h^\delta_i(0)\, \forall i$ (and thus $\hat u^\delta(\cdot,0)=\hat\psi^\delta(\cdot,0)$).  The $h^\delta_i$ are generalized characteristics of the exact solution $\bar{u}$.

Denote by $\mathcal{I}^\delta$ the set of indices belonging to curves $\hat h^\delta_i$ whose starting point coincides with one of the $x_j$.

We denote by $\mathcal{D}$ the set of \emph{rapidly decreasing} parts of $\bar u^0$.

 Then, the distance between $\bar u$ and the approximate solution $\hat u^\delta$ 
 is controlled as follows:
 
Fix $S>0$ sufficiently large.  First, we have the (obvious) estimate for every $L^p$ space
\begin{multline} \label{est:uhu}
    \norm{\bar{u}(\cdot,t)-\hat{u}^\delta(\cdot,t)}_{L^p(-S+st,S-st)}
        \\
   \leq \norm{\bar{u}(\cdot,t)-{\hat{\psi}^\delta}(\cdot,t)}_{L^p(-S+st,S-st)}+\norm{{\hat{\psi}^\delta}(\cdot,t)-\hat{u}^\delta(\cdot,t)}_{L^p(-S+st,S-st)},
\end{multline}
where $s>0$ is the speed of information, a constant depending only on $\norm{\bar{u}}_{L^\infty}$, the flux $A$, and the entropy $\eta$.

Furthermore,
\begin{equation}\label{est:hphu}
    \norm{{\hat{\psi}^\delta}(\cdot,t)-\hat{u}^\delta(\cdot,t)}_{L^1(-S+st,S-st)}\leq  
    \int_{\mathcal{B}(t)} | \hat v^\delta_{m(x,t)} - \hat v^\delta_{M(x,t)} | dx
    +  \sum_{d \in \mathcal{D}} \Big[\Gamma_d(t)+2S\Upsilon_d(t)\Big],
\end{equation}
where $\Gamma_d(t)$ is a computable quantity defined in \eqref{eq:Bressan}, $\Upsilon_d(t)$ is a computable quantity defined in\eqref{eq:psipsift}, 
\[ \mathcal{B}(t):=\{ x \in \mathbb{R} : \exists j \in \mathcal{I}^\delta : x \in [\hat h^\delta_j(t)-  \Delta_j(t) ,\hat h^\delta_j(t)+  \Delta_j(t) ] \}\]
denotes regions (around the $\hat{h}^\delta_j$, $j\in\mathcal{I}^\delta$) where it is unclear which part of the solution is revealed,
and for $x \in \mathcal{B}(t)$
\begin{align}
    \mathcal{J}(x,t) & := \{ j  :x \in [\hat h^\delta_j(t)-  \Delta_j(t) ,\hat h^\delta_j(t)+  \Delta_j(t)]\} \\
    m(x,t) & := \min \{i \in  \mathcal{J}(x,t)\}\\
    M(x,t) & := \max \{i \in  \mathcal{J}(x,t)\} +1.
\end{align}

In all of the above $\Delta_i(t)$ is a (computable) bound on  $\abs{\hat{h}^\delta_i(t)-h^\delta_i(t)}$ (the uncertainty of the position of the $i^{th}$ discontinuity at time $t$). For formulas for $\Delta_i$, see \eqref{def_control_shift}, \eqref{eq:deltabs}, and \Cref{sec:triangles}.

Finally,
\begin{align}\label{control_l21206}
    \norm{\bar{u}(\cdot,t)-{\hat{\psi}^\delta}(\cdot,t)}^2_{L^2(-S+st,S-st)}\leq 
    C\Bigg(
    \norm{\bar{u}(\cdot,0)-{\hat{\psi}^\delta}(\cdot,0)}^2_{L^2(-S,S)} + \mathcal{R}(t) \Bigg)e^{C t},
\end{align}
with 
\begin{equation}\label{def:cR} \mathcal{R} (t):= \int\limits_{0}^{t}
\int\limits_{-S+sr}^{S-sr} \max_{i \in \tilde J(x,r)} (R^\delta_i(x,r))^2\,dx dr
\end{equation}
and where 
\begin{align}
\tilde J(x,t):=
\{ j : \hat h^\delta_j(t) - \Delta_j(t) < x < \hat h^\delta_{j+1} (t) + \Delta_j(t) \}
\end{align}
is the set of potentially revealed solutions, and 
the residuals are defined by 
\begin{equation}\label{system-numerical}
R^\delta_i:= \partial_t \hat v^\delta_i + \partial_x A(\hat v^\delta_i).
\end{equation}

The constant $C>0$ can be explicitly given in terms of  $\norm{\bar{u}}_{L^{\infty}}$, the  flux $A$, entropy $\eta$ and their derivatives and $\norm{\Big[\partial_x \bigg|_{(x,t)} \eta'({\hat{\psi}^\delta})\Big]_{-}}_{L^\infty(K)}$, where $[\hspace{.03in}\cdot\hspace{.03in}]_{-}\coloneqq \min(0,\cdot)$ and $K:= \{(x,t) : {t \in [0,T), \ }-S + st < x < S - st \}$.
\end{theorem}

 We use the notation $\Delta_i$ to denote an upper bound on the difference $\abs{\hat{h}^\delta_i(t)-h^\delta_i(t)}$ for a shock with position $\hat{h}^\delta_i$ in the computer and the corresponding generalized characteristic with position $h^\delta_i$ in the corresponding exact solution.

In fact, we are able to prove much more than \Cref{main_theorem}. See \Cref{algo_section} for a precise discussion of the various quantitative estimates we have on $\Delta_i$ for shocks of various sizes.

We implement our scheme (for the construction of $\hat u$) and our error estimator, following the analysis in the present paper, in MATLAB\footnote{See GitLab: \url{https://git-ce.rwth-aachen.de/jan.giesselmann/shiftsshocksaposteriori.git}}.

\begin{remark}
    We will often drop the superscript $\delta$ from the $\hat u^\delta$, $\hat \psi^\delta$, $v^\delta_i$, $\hat v^\delta_i$, $h^\delta_i$, $\hat h^\delta_i$,  $N^\delta$, and $R^\delta_i$.
\end{remark}

\begin{remark}
    The expected convergence rate for our a posteriori error estimator differs depending on whether or not the initial data $\bar{u}^0$ contains a rapidly decreasing piece  -- even for first order finite volume schemes.
    In what follows, we will denote the mesh width used in the simulations of the $v_i$ as $h$ and the time step as $\Delta t$. We use explicit time-stepping such that the CFL condition implies $\Delta t \lesssim h$.
    In case the initial data $\bar{u}^0$ consists only of nearly non-decreasing pieces separated by shocks (with a size independent of $h$), our error estimator is  $\mathcal{O}(h)$ but not $o(h)$ for stable first order schemes and we expect higher order convergence of the estimator for stable higher order schemes and sufficiently smooth (pieces of the) initial data.
    In contrast, if the initial data $\bar{u}^0$ contains at least one rapidly decreasing piece we only observe order $h^{1/4}$ scaling of our estimator. This is due to very limited control of shocks originating from the boundary of rapidly decreasing intervals. See \Cref{sec:op-ed} for more details.
\end{remark}

\begin{remark}
The functions $\hat h^\delta_i$ are (numerical approximations of) generalized characteristics for the numerical solution $\hat u^\delta$. The functions $h^\delta_i$ are the corresponding generalized characteristics for the exact solution $\bar{u}$. The function ${\hat{\psi}^\delta}$ is the ``shifted'' version of the numerical solution. It is not necessarily a solution to the conservation law. In particular, its shocks are not necessarily travelling with the correct Rankine-Hugoniot speed. 
\end{remark}

\begin{remark}
While for purposes of estimating $\norm{\hat{u}^\delta(\cdot,t)-{\hat{\psi}^\delta}(\cdot,t)}_{L^2}$ we require that $\bar{u}$ have initial data $\bar{u}^0 \in \inisp$, the theory of shift techniques we employ require no regularity assumptions on $\bar{u}$ for positive times $t>0$ besides verifying the Strong Trace Property (\Cref{strong_trace_definition}). This is also a general feature of the $a$-contraction theory for the case of systems of conservation laws: there are no regularity assumptions on the weak, entropic solution in the first slot of the relative entropy $\eta(\cdot|\cdot)$ (the $\bar{u}$ in the context of \Cref{main_theorem}) -- see \Cref{sec:definerelentrop}. It is even allowed to take vacuum states \cite{VASSEUR2008323,MR3519973,Leger2011}. {In fact, all of our results hold also for simply  $\bar{u}^0 \in L^\infty(\mathbb{R})$. The first step is to approximate $\bar{u}^0$ in $L^2$ (on some compact interval) by an element of $\inisp$ (notice $\inisp$ is dense in $L^2$). However,  if we do not assume any additional regularity on $\bar{u}^0$ besides $L^\infty$, and in particular $\bar{u}^0\notin \inisp$, then we do not have control on the expected convergence rate of our error estimator (see \Cref{sec:op-ed}).}
\end{remark}

\begin{remark}
    The Strong Trace Property (\Cref{strong_trace_definition}) is a standard assumption in the theory of shifts and $a$-contraction theory. It is an open question whether the Strong Trace Property truly needs to be assumed in the theory of shifts and $a$-contraction well-posedness theory. Similarly, it is an open question whether solutions automatically regularize enough to attain strong traces. The topic is the focus of much recent work. For the isentropic flow of an ideal gas with $\gamma=3$, Vasseur \cite{MR1720782} is able to show that solutions to this system are regular in time. Further, recent work by Golding \cite{golding2022unconditional} shows that for this particular system, $L^\infty$ solutions regularize, such that entropic solutions have significant regularity. However, \cite{golding2022unconditional} does not quite show that solutions attain the strong traces (\Cref{strong_trace_definition}) we use in the present paper. Other recent work studies whether systems of conservation laws, with one entropy condition, can admit solutions via convex integration \cite{doi:10.1142/S021919972250081X,MR4144350,krupa2022nonexistence}. Such solutions, if they exist, will certainly not have the strong traces.
\end{remark}

\begin{remark}
Estimate \eqref{est:hphu} follows by combining our estimates for the uncertainty in shock positions $\Delta_i$ with the difference of values of parts that might be revealed in the different solutions; plus estimates from the front tracking theory for the $L^1$ error in regions where we use wave front tracking.
\end{remark}

\begin{remark}
Note that there are two types of discontinuities in $\hat u^\delta$: those that correspond to discontinuities in $\bar u^0$ and those that come from breaking up rapidly decreasing parts of $\bar u^0$ into step functions. The former are ``large'' in the sense that their size is independent of any discretization parameter. The size of the latter depend on the parameter $\delta>0$ that we use for approximating the rapidly decreasing parts of the initial data $\bar{u}^0$ by step functions.
\end{remark}

\begin{remark}\label{ordering_remark}
Note that we might need to compute each numerical solution $\hat v^\delta_i$ on  all of the computational domain which makes our approach quite costly.
There are two ways in which this can be mitigated: (1) all solutions $\hat v^\delta_i$ are independent of each other, i.e.,  computing them in parallel is straightforward; however, information from different solutions needs to be combined to evolve shock curves.
Moreover, (2) the simulated numerical solutions result from (nearly) non-decreasing initial data so that we may use higher order schemes without the need to account for shocks.
The only constraint on numerical schemes is that we need to reconstruct all numerical solutions  in such a way that the global ordering condition $\hat v^\delta_i > \hat v^\delta_{i+1}$ is maintained (see \eqref{ordering_main_thm} and \eqref{ordering_vs}).

Ordering of numerical solutions and their reconstructions $\hat v^\delta_i$ as well as  lower bounds for pointwise differences between different $\hat v^\delta_i$ 
is guaranteed \textit{ a priori} for monotone finite volume schemes and the reconstructions from \cite{GiesselmannMakridakisPryer, DednerGiesselmann} but not for higher order schemes.
Indeed, in monotone finite volume schemes, the value at time $t^{n+1}$ in the $i$-th cell, denoted $u_i^{n+1}$ can be expressed as a function
\[ u_i^{n+1} = H(u_{i-1}^{n},u_i^{n},u_{i+1}^{n})\]
where each partial  derivative of $H$ is non-negative and the sum of partial derivatives is $1 + \mathcal{O}(|u_{i-1}^{n}-  u_i^{n}| + | u_i^{n}  - u_{i+1}^{n}|).$
The reconstruction for piece-wise constant schemes simply uses linear interpolation which maintains ordering and gaps.

For higher order schemes and reconstructions the ordering can be verified \textit{ a posteriori} since the reconstructions are computed explicitly.

\end{remark}

\begin{remark}
If we can guarantee that each of the reconstructions $\hat v_i$ is non-decreasing we can drop the term $\norm{\Big[\partial_x \bigg|_{(x,t)} \eta'({\hat{\psi}^\delta})\Big]_{-}}_{L^\infty(K)}$, in the statement of the theorem. 
Otherwise, Lipschitz constants of the $\hat v_i$ are a posteriori computable.
\end{remark}

\begin{remark}
It should be noted that our bound on $\Delta_i$ (see \eqref{def_control_shift}, below) depends on $\mathcal{R}$ and the bound on $\mathcal{R}$ in \eqref{def:cR} depends on the $\Delta_i$.  This interdependence can be overcome by first using a very pessimistic bound for the $\Delta_i$ and subsequent iteration. Details can be found in Section \ref{algo_section}. 
\end{remark}

We also have the following result on the convergence of our error estimator.
 \begin{theorem}[Convergence of the estimator]\label{conv_thm}

We consider the same assumptions as in \Cref{main_theorem} above.
 
In particular, recall that we fix $\bar{u}^0 \in \inisp$. Further, we let $\bar{u}$ be an exact solution to our conservation law \eqref{system} with initial data $\bar{u}^0$.
     
     Let $\{{\delta_n}\}_{n=1}^{\infty}\subset(0,\infty)$ be a sequence of positive numbers such that $\delta_n\to0$ as $n\to\infty$. Then, for each $n\in\mathbb{N}$, we have, from within the context of \Cref{main_theorem} above, $N^{\delta_n}\in\mathbb{N}$, ${v}^{\delta_n}_1(x,t),\ldots, {v}^{\delta_n}_{N^{\delta_n}+1}(x,t)$, as well as the corresponding $\hat{v}^{\delta_n}_1(x,t),\ldots, \hat{v}^{\delta_n}_{N^{\delta_n}+1}(x,t)$. We also have the residuals $R^{\delta_n}_1,\ldots, R^{\delta_n}_{N^{\delta_n}+1}$.
     
     Then assume there is a constant $c_0>0$ (which is uniform in $n$) such that
    \begin{align}\label{bound_lip_thm}
    \sup_t\textrm{Lip}[\hat{v}^{\delta_n}_i(\cdot,t)]\leq \frac{1}{c_0}
    \end{align}
    for all $i\in \mathcal{L}^{\delta_n}$, where $\mathcal{L}^{\delta_n}\coloneqq \{1,\ldots,N^{\delta_n}+1\}$ is the index set for all the $\hat{v}^{\delta_n}_i$, and 
     \begin{align}\label{assume_gap}
       \hat{v}^{\delta_n}_i(x,t) - \hat{v}^{\delta_n}_{i+1}(x,t)\geq c_0 \cdot \inf_{y} \big[v^{\delta_n}_i(y,0) - v^{\delta_n}_{i+1}(y,0)\big],
     \end{align}
     for all $(x,t)\in K$ and for all $i\in \mathcal{L}^{\delta_n}$. Recall $ v^{\delta_n}_i(x,0) - v^{\delta_n}_{i+1}(x,0)\geq \frac{1}{2} \epsilon{\delta_n}>0 \ \forall i \text{ and } \forall x$.

Furthermore, assume that
\begin{multline}
    \frac{1}{{\delta_n}}\Bigg[\int\limits_{0}^{T}
\int\limits_{-S+sr}^{S-sr} \max_{i\in\mathcal{L}^{\delta_n}} (R^{\delta_n}{_i}(x,r))^2 dx dr
\\
+\sum_{i\in\mathcal{L}^{\delta_n}}\Bigg[\norm{v^{\delta_n}_{i}(\cdot,0)-\hat v^{\delta_n}_{i}(\cdot,0)}^2_{L^2([\hat{h}^{\delta_n}_{i-1}(0),\hat{h}^{\delta_n}_{i}(0)]}\Bigg]\Bigg]
  \to 0  \mbox{ as } n\to\infty,
\end{multline}
 with $\hat{h}^{\delta_n}_0(0)\coloneqq -S$ and $\hat{h}^{\delta_n}_{N^{\delta_n}+1}(0)\coloneqq S$.

Moreover, assume that 

\begin{align*}
\sum_{i\in \mathcal{I}^{\delta_n}} \int\limits_0^T \abs{\dot{\hat{h}}^{\delta_n}_i(s)-\sigma(\hat{u}^{\delta_n}(\hat{h}^{\delta_n}_i(s)+,s),\hat{u}^{\delta_n}(\hat{h}^{\delta_n}_i(s)-,s))}\, ds \to 0  \mbox{ as } n\to\infty,
\end{align*}
where the quantity $\sigma(\hat{u}^{\delta_n}(\hat{h}^{\delta_n}_i(s)+,s),\hat{u}^{\delta_n}(\hat{h}^{\delta_n}_i(s)-,s))$ is the Rankine-Hugoniot speed of the shock $(\hat{u}^{\delta_n}(\hat{h}^{\delta_n}_i(s)+,s),\hat{u}^{\delta_n}(\hat{h}^{\delta_n}_i(s)-,s))$ (see \eqref{eq:RH}), and $\mathcal{I}^{\delta_n}$ is as in \Cref{main_theorem}.

Then, as $n\to\infty$ the right-hand sides of the estimates \eqref{est:hphu} and \eqref{control_l21206} go to zero. 
 \end{theorem}
For the proof of \Cref{conv_thm}, see \Cref{sec:proof_conv_thm}.

\begin{remark}
Compare \eqref{bound_lip_thm} and \eqref{assume_gap} with \Cref{control_lipschitz_lem} and \Cref{gap_lemma} below, respectively,  which concern exact solutions. In particular, remark that, by the construction we will give, for the exact solutions we have
\begin{align}
    \sup_n \sup_{i\in\mathcal{L}^{\delta_n}} \sup_t\textrm{Lip}[{v}^{\delta_n}_i(\cdot,t)]<+\infty.
    \end{align}
\end{remark}

\subsection{A high-level overview of the proof of \Cref{main_theorem} (Main Theorem)}
\hspace{.001in}

At a high level, our method works by considering an exact solution $\bar{u}$ to \eqref{system}
 with initial data $\bar{u}^0\in\inisp$. From the initial data $\bar{u}^0\in\inisp$, we find an element $\tilde{u}^0\in\inise$ which is $L^2$-close to $\bar{u}^0$. A parameter $\delta>0$ controls the minimum size of discontinuities in the discretization of the function $\bar{u}^0$ to make $\tilde{u}^0$. For each interval where $\tilde{u}^0$ is continuous, we create a Lipschitz-continuous artificial extension defined globally on all of $\mathbb{R}$ (see \Cref{ext_section}). Then, for each artificial extension we create a numerical solution $\hat v_i$ which has this artificial extension as initial data (although for the small shocks of size proportional to $\delta$, we do not in practice have to numerically simulate the extensions -- see \Cref{do_not_sim_remark}). Then, by gluing together the artificial solutions $\hat v_i$ using the Rankine-Hugoniot condition (with a numerical ODE solver), we construct the numerical solution $\hat u$. On the other hand, we also define a function $\hat\psi$ which is also formed by gluing together the numerical solutions $\hat v_i$. However, the points where the $\hat v_i$ are glued together are not determined by Rankine-Hugoniot, but instead by (unknown) generalized characteristics from the exact solution $\bar{u}$. For the construction of $\hat u$ and $\hat \psi$, see \Cref{construction_section}. 

\begin{figure}[tb]
      \includegraphics[width=.8\textwidth]{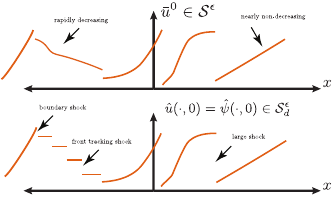}
  \caption{The function $\bar{u}^0$ is divided into \emph{rapidly decreasing} parts and \emph{nearly non-decreasing} parts. In $\hat{u}(\cdot,0)=\hat\psi(\cdot,0)$, we differentiate between three different types of shocks: \emph{large shocks}, \emph{front tracking shocks}, and \emph{boundary shocks}.}\label{shocks_fig}
\end{figure}

At each time $t$, we will discern three different types of shocks in $\hat \psi(\cdot,t)$ and $\hat u(\cdot,t)$ (see \Cref{shocks_fig}): 
 
 \begin{itemize}
     \item \emph{front tracking shocks} -- shocks which are of size proportional to $\delta$ at time $t=0$ and at $t=0$ are in the interior of one of the intervals where $\bar{u}^0$ is rapidly decreasing. To be a ``front tracking'' shock at time $t>0$, the shock must have collided with only other ``front tracking'' shocks up until (and including) time $t$. These shocks in $\hat{u}$ will move according to a simple scalar front tracking scheme;
      \item \emph{large shocks} -- shocks which at time $t$ are not ``front tracking'' shocks and are not adjacent to a ``front tracking'' shock. At time $t=0$, they have a size independent of $\delta$;
     \item \emph{boundary shocks} -- the shocks which are neither ``front tracking'' nor ``large''. They are adjacent to ``front tracking'' shocks. Remark that at later times $t>0$, a ``boundary'' shock might be adjacent to, both on the left and right, a ``front tracking'' shock -- such as when two ``boundary'' shocks, separated by an interval corresponding to a nearly non-decreasing piece of $\bar{u}^0$ (at $t=0$), collide (see \Cref{twobdycollide2_fig} and \Cref{sec:ftbothsides}).
 \end{itemize}
 
 Due to using generalized characteristics from the exact solution $\bar{u}$, $\hat \psi$ and $\bar u$ will be $L^2$-stable relative to each other (see \Cref{Revelation_section}). Then, control on $\bar u-\hat u$ will come from control on $\hat u-\hat \psi$. To control $\hat u-\hat \psi$, we need to determine how much the speed of the shocks in $\hat \psi$ differs from the Rankine-Hugoniot speed of these shocks (recall the shocks in $\hat u$ move with numerically approximated Rankine-Hugoniot speed). The error intervals $\Delta_i$ which measure the difference between shock positions in $\hat u$ and $\hat \psi$ will depend on the type of shock: large shock, front tracking shock, or boundary shock? See \Cref{algo_section}.

 A key part of estimating the error in shock positions between $\hat \psi$ and $\hat u$ will be determining when collisions which have occurred in $\hat u$ have also occurred in $\hat \psi$ (\Cref{sec:when_collided}). The general principle here will be: when two shocks (neither of which is ``front tracking'') collide in $\hat u$, continue the two shocks forward in time as if the collision did not occur (so no change in left- or right-hand states of either shock). While continuing both shocks forward in time, calculate error intervals between the positions of these shocks and the positions of the corresponding shocks in $\hat \psi$ (again, under the assumption the corresponding shocks in $\hat \psi$ have not collided). At such a time when these two error intervals (one for each shock) no longer intersect, we can conclude by contradiction that the shocks must have collided in $\hat\psi$.

 For the ``front tracking'' shocks, we will control them using Bressan's $L^1$-theory \cite{Bressan2000} (see \Cref{front_tracking_shocks}).

\section{Introduction to the theory of shifts, $a$-contraction, and their use in quantitative estimates}\label{describe_theory}

Let us motivate one of the major techniques utilized in this paper: the \emph{theory of shifts} (and the closely related technique of \emph{$a$-contraction}).
The theory of shifts was devised by Vasseur to show $L^2$ stability for shock solutions to conservation laws \cite{VASSEUR2008323}. The technique has reached a certain level of maturity and, recently, it has been used to show uniqueness of solutions to scalar conservation laws with convex flux in one space dimension follows from having one convex entropy, see \cite{2017arXiv170905610K}. This result is not obvious, and in fact has been proven two other times before \cite{delellis_uniquneness,panov_uniquness}, albeit using entirely different techniques. In \cite{chen2020uniqueness}, the method of \emph{$a$-contraction} was used to show  that the small-$BV$ solutions considered by Bressan-Crasta-Piccoli \cite{MR1686652} are in fact stable (and unique) among the large class of entropic, weak solutions with possibly very large data. Again, a single convex entropy is used. 

Remark that for the scalar conservation laws, the solutions to the Riemann problem are non-unique if the flux function has even a single inflection point. This includes for example concave-convex flux functions e.g, $f(u)=u^3+au$ for any $a\in\mathbb{R}$. The non-uniqueness persists even if solutions are required to verify the entropy inequality for at least one strictly convex entropy (see e.g. \cite{lefloch_book}).

\subsection{Shifts and $a$-contraction}\label{sec:definerelentrop}
For the system \eqref{system}, we fix an entropy and entropy-flux $\eta\colon\mathbb{R}\to\mathbb{R}$ and $q\colon\mathbb{R}\to\mathbb{R}$, respectively. They verify the compatibility condition $q'=\eta'A'$, where $A$ is the flux for our system \eqref{system}.

We assume that $\eta$ is strictly convex. Furthermore, for $a,b\in\mathbb{R}$, we define the relative entropy,
\begin{align}\label{rel_entropy}
\eta(a|b)\coloneqq \eta(a)-\eta(b)-\eta'(b)(a-b).
\end{align}

We define the relative entropy-flux associated with the relative entropy \eqref{rel_entropy}:
\begin{align}\label{rel_entropy_flux}
q(a;b)\coloneqq q(a)-q(b)-\eta'(b)(A(a)-A(b)).
\end{align}

We also define the relative flux:
For $a,b\in\mathbb{R}$,
\begin{align}
A(a|b)\coloneqq A(a)-A(b)-A'(b)(a-b). \label{Z_def}
\end{align}

Note that due to the strict convexity of the entropy $\eta$, for all $a$ and $b$ in a  fixed compact set, there exist constants $c^*,c^{**}>0$ such that 
\begin{align}\label{control_rel_entropy}
c^*(a-b)^2\leq \eta(a|b)\leq c^{**}(a-b)^2.
\end{align}

Thus, the method of relative entropy is an $L^2$-type theory.

Now, consider a weak solution $u$ to the scalar conservation law \eqref{system}, entropic for the entropy $\eta$. Moreover, consider an entropic pure shock solution $\bar{u}$,
\begin{align}
\bar{u}(x,t)\coloneqq
\begin{cases}
u_L,\mbox{ if } x<\sigma(u_L,u_R) t\\
u_R,\mbox{ if } x>\sigma(u_L,u_R) t,
\end{cases}
\end{align}
where $u_L,u_R\in\mathbb{R}$ and $\sigma(u_L,u_R)$ is the Rankine-Hugoniot speed naturally associated to the shock $(u_L,u_R)$, i.e.
\begin{equation}
    \label{eq:RH}
 \sigma(v,w)\coloneqq
  \begin{cases}
     \frac{A(v)-A(w)}{v-w} & \text{if } v\neq w,\\
   A'(w) & \text{if } v= w.
  \end{cases}
\end{equation}

The focus of interest of  \cite{VASSEUR2008323} and subsequent research was in stability of the shock $\bar u$ in a large class of weak solutions $u$ with minimal regularity assumptions, asking only for a trace condition which is weaker than $BV_{\text{loc}}$ (see \Cref{strong_trace_definition}). More precisely, given a particular Lipschitz-continuous function of time $h(t)$, how does the quantity 
\begin{align}\label{how_grow}
  \int\limits_{-\infty}^{h(t)} \eta(u(x,t)|u_L)\,dx +    \int\limits_{h(t)}^{\infty} \eta(u(x,t)|u_R)\,dx
\end{align}
grow in time? 

Notice that as in weak-strong stability theory, in \eqref{how_grow} the solution $u$ with low regularity is put in the first slot of the relative entropy and the solutions with more regularity, and on which there is more control, in the second slot (in this case, constant states $u_L$ and $u_R$). Remark that throughout this paper, we will put exact solutions in the first slot of the relative entropy, and numerical solutions in the second slot of the relative entropy. Since we perform  a posteriori analysis, we have large amounts of control over the numerical solutions.

The theory of shifts was devised to overcome the following difficulty: when $h(t)=\sigma(u_L,u_R) t$, the quantity \eqref{how_grow} grows in time. More problematically, estimates in the form of Dafermos and DiPerna's weak-strong stability theory do not hold. Simple examples show that this is true even in the case of scalar conservation laws.

However, if a particular Lipschitz function $h(t)$ is chosen, control over the quantity \eqref{how_grow} can be gained. This is equivalent to studying the stability between $u$ and the function
\begin{align}
\bar{u}_{\text{shifted}}(x,t)\coloneqq
\begin{cases}
u_L,\mbox{ if } x<h(t)\\
u_R,\mbox{ if } x>h(t).
\end{cases}
\end{align}
In other words, the shock in the solution $\bar{u}$ will no longer be moving with its natural Rankine-Hugoniot speed, but instead will be moving with an artificial velocity dictated by the special choice of $h$. Due to $h(t)\neq \sigma(u_L,u_R) t$, the shock in $\bar{u}_{\text{shifted}}$ will not be traveling with Rankine-Hugoniot speed and thus $\bar{u}_{\text{shifted}}$ is no longer an exact solution to the conservation law. We note that when $u_L$ and $u_R$ are truly constants, as in this example,  \eqref{how_grow} is a \emph{decreasing function of time}. The artificial ``shift'' in position of the shock forces a contraction in $L^2$ as measured by the relative entropy functional $(a,b)\mapsto\eta(a|b)$. This is the theory of shifts. In the context of the relative entropy method, this idea was devised by Vasseur \cite{VASSEUR2008323}.

The theory of shifts gives some control on the difference $\sigma-\dot{h}$ (a measure of how artificial the shift is). However, up until this point, the control on the shifts has never played a significant role in the theory. In both \cite{chen2020uniqueness,2017arXiv170905610K}, the stability and uniqueness results do not rely on controlling the shift functions. Instead, the contractive property of \eqref{how_grow} (as a function of time) is key.

A similar procedure works for systems, but the situation is more involved:
Indeed, in the case of systems of conservation laws with multiple conserved quantities, there are examples where no function $h$ exists such that \eqref{how_grow} is decreasing in time \cite{serre_vasseur}. However, by introducing a spatially inhomogeneous pseudo-distance we can regain contractivity \cite{MR3519973}. In particular, we introduce a parameter $a>0$ and consider the quantity
\begin{align}\label{how_grow_a}
  \int\limits_{-\infty}^{h(t)} \eta(u(x,t)|u_L)\,dx +   a \int\limits_{h(t)}^{\infty} \eta(u(x,t)|u_R)\,dx,
\end{align}
where $a$ is chosen sufficiently large or sufficiently small (depending on the wave family of the shock we are considering). Then, for particular values of $a$, there again exists a Lipschitz function of time $h$ such that \eqref{how_grow_a} is decreasing in time. This is the technique of \emph{$a$-contraction}.

\subsection{Uniqueness via shifts}\label{sec:uniqueshifts}
While  \cite{chen2020uniqueness,2017arXiv170905610K}  provide uniqueness results neither gives quantitative stability results. We will outline in this section why this is the case. 

Both \cite{chen2020uniqueness,2017arXiv170905610K} use the following framework in their proofs: given a weak solution $u$, entropic for a convex entropy, approximate the initial data $u(\cdot,0)$ by piecewise-smooth initial data $\hat{u}^0$. Then the initial data $\hat{u}^0$ are evolved forward in time according to the theory of shifts. This gives $\psi$, a function of time and space. One key achievement in \cite{chen2020uniqueness,2017arXiv170905610K} is to ensure  $L^2$ stability between $\psi$ and $u$, i.e., control on $\norm{\psi(\cdot,t)-u(\cdot,t)}_{L^2}$, by letting the discontinuities in $\psi$ not travel with Rankine-Hugoniot speed but instead move according to artificial shift functions as described above. In the systems case with multiple conserved quantities, $\hat{u}^0$ is piecewise-constant and $\psi$ is constructed using the front tracking algorithm, but with wave speeds dictated by shifts. Thus, $\psi$ is in general not a solution to the conservation law under consideration. 
The analysis in \cite{chen2020uniqueness,2017arXiv170905610K}  proceeds by combining the facts that $\psi$ will stay $L^2$ close to the weak solution $u$, and, by construction, $\psi$ will have some regularity property which is uniform in the choice of initial data $\psi(\cdot,0)$. Due to $\psi$ staying $L^2$-close to the weak solution $u$, and $\norm{\psi(\cdot,0)-u(\cdot,0)}_{L^2}$ being arbitrarily small, the regularity property of $\psi$ is transferred to the weak solution $u$. In the scalar case, the function $\psi$, even though it is not a solution to the conservation law, will always obey the decay estimate given by Ole\u{\i}nik's \emph{condition E} \cite{MR0094541}. We briefly recall that a solution $u$ to \eqref{system} verifies condition E if 
\begin{align}\label{condE}
    \begin{cases}
        \mbox{There exists a constant $C>0$ such that}\\
       \hspace{1.3in}  u(x+z,t)-u(x,t)\leq \frac{C}{t}z\\
         \mbox{for all $t>0$, almost every $z>0$, and almost every $x\in\mathbb{R}$.}
    \end{cases}
\end{align}

Condition E is a uniqueness criterion, so transferring this property to $u$ implies uniqueness. In the systems case,  the trace of $\psi$ along space-like curves will have bounded variation. This is the \emph{Bounded Variation Condition}, which is known to be a condition for uniqueness \cite{MR1757395}. The function $\psi$ verifies this condition in a uniform way, and this completes the proof.

\subsection{More details on shifts in the scalar case}\label{sec:dssc}
Because they will play an important role in the present paper, let us present a few more of the ideas behind how the framework for uniqueness is applied in the scalar case \cite{2017arXiv170905610K}. 

Firstly, the initial data $ \bar u(\cdot,0)$ are approximated by piecewise-smooth initial data $\hat{u}^0(\cdot)=\psi(\cdot,0)$. In particular, $\psi(\cdot,0)$ is piecewise-Lipschitz such that the Lipschitz pieces are monotonically increasing (cf. the space $\inise$). Initial data in this form is dense in $L^2$ (see \cite{2017arXiv170905610K}). Then, for each interval $(x_i,x_{i+1})$ where $\psi(\cdot,0)$ is smooth,  $\psi(\cdot,0)|_{(x_i,x_{i+1})}$ is extended to  a global-in-space non-decreasing function $ v_i^0$ such that the extensions satisfy the pointwise ordering condition $ v_i^0 \geq v_{i+1}^0$ (cf. \eqref{ordering_vs}). Then, for each $i$ a global-in-time smooth solution $v_i$ to the conservation law with initial data $v_i^0$ is constructed.

Secondly, the quantity
\begin{align}\label{how_growpsi}
  \sum_{i=0}^N \int\limits_{h_i(t)}^{h_{i+1}(t)} \eta(u(x,t)|v_{i+1}(x,t))\,dx,
\end{align}
is considered,
where $h_0\coloneqq -\infty$ and $h_{N+1}\coloneqq +\infty$, and at the initial time $h_i(0)=x_i$ for $i=1,\ldots,N$. For $i=1,\ldots,N$, the $h_i$ are shift functions constructed according to the theory of shifts. 

Compare \eqref{how_grow} and \eqref{how_growpsi}: A weak solution $u$ was put in the first slot of the relative entropy, and smooth solutions go in the second slot of the relative entropy. When the $h_i$ are appropriately chosen shift functions, \eqref{how_growpsi} will again be a decreasing function of time. This gives a function $\psi\colon\mathbb{R}\times[0,\infty)\to\mathbb{R}$,
\begin{align}\label{psi_firstdef}
\psi(x,t)\coloneqq
\begin{cases}
v_1(x,t) & \text{ if } x<h_1(t)\\
v_2(x,t) & \text{ if } h_1(t)<x<h_2(t)\\
\hspace{.17in}\vdots \\
v_{N+1}(x,t) & \text{ if } h_N(t)<x.
\end{cases}
\end{align}
When two shift functions $h_i$ and $h_j$ collide with each other, the $v_i$ that was between them is simply deleted, and the two shifts  are glued together (i.e., $h_i\equiv h_j$ for all future times), and then we continue running the clock forward in time. Due to there being only finitely many curves of discontinuity $h_i$, this process will only have to be repeated at most finitely many times. This defines the function $\psi$ globally in time.

Then, this gives a function $\psi$ which is $L^2$-close to $u$ for all time, due to \eqref{control_rel_entropy} and \eqref{how_growpsi} being decreasing in time.  Further, $\norm{\psi(\cdot,0)-u(\cdot,0)}_{L^2}$ can be chosen arbitrarily small. To complete the framework, it needs to be shown that $\psi$ verifies condition E (see \eqref{condE}). This follows because in between the $h_i$, $\psi$ is a classical solution to our conservation law, and thus is well-known to verify condition E in a uniform way (independent of the $v_i$ we chose). Furthermore, due to the choices of the extensions $v_i$, the function $\psi$ will verify condition E also along the $h_i$. This implies $u$ verifies condition E, and thus $u$ is the unique solution to verify condition E. This completes the argument.

\subsection{Difficulties in deriving quantitative stability results}

In our program to use the theory of shifts to come up with a posteriori error estimates for numerical approximations we replace $\psi$ by a quantity $\hat \psi$ where the exact solutions $v_i$ are replaced by numerical solutions $\hat v_i$ (or more precisely, by Lipschitz-continuous reconstructions of numerical solutions in the spirit of \cite{GiesselmannMakridakisPryer, DednerGiesselmann}).
The problem with this approach is that in order to keep $\hat \psi$ $L^2$-close to $u$ the discontinuities in $\hat \psi$ need to move with shifts. The speeds of the shifts are dictated by $u$, i.e. they are unknown and not accessible numerically. 
The natural choice of speeds that are accessible are Rankine-Hugoniot speeds computed using the $\hat v_i$ but if the $\hat v_i$ are glued together using curves with these speeds the result is not $\hat \psi$ but some new computable (!) quantity $\hat u$.

The natural way to obtain quantitative control on $\norm{\hat{u}(\cdot,t)-u(\cdot,t)}_{L^2}$ 
is to combine the control on  $\norm{\hat \psi(\cdot,t)-u(\cdot,t)}_{L^2}$,
that follows from including residuals in the theory of shifts, with estimates on $\norm{\hat \psi(\cdot,t)-\hat{u}(\cdot,t)}_{L^2}$ which, in principle, can be obtained by controlling the size of the artificial shift. 
However, there are two difficulties with this which are essentially the same in the scalar and the system case:

For one, consider the following. Let $t^*>0$ be such that for all $t\in[0,t^*]$, there are no interactions between waves in either $\hat{u}$ or $\hat \psi$. Then, for $t\in[0,t^*]$, we can relate $\hat{u}$ and $\hat \psi$ by a simple change of variables. However, if for example after time $t^*$ two waves collide in $\hat\psi$ but not in $\hat{u}$, it is no longer possible to relate $\hat{u}$ and $\hat \psi$ by a simple change of variables. This is what we refer to as the \emph{change of variables problem.}
This problem is even more involved in the systems case \cite{chen2020uniqueness}. Indeed, in the systems case, when the waves collide in $\hat\psi$, the front tracking algorithm stops the clock, solves the new Riemann problem which arises, and then restarts the clock, thus fundamentally altering the relationship between $\hat\psi$ and $\hat{u}$.

The second difficulty is that in both the scalar and systems case, the control on the shift gets worse and worse as the size of the shock (being controlled by shift) decreases. This is highly problematic, as one is interested in  letting $\hat\psi(\cdot,0)\to u(\cdot,0)$ in $L^2$ and thus the shocks in $\hat\psi$ will become smaller when simulations on finer meshes are considered.

This is what we call the \emph{problem of loss of control on small shocks.}

\subsection{A strategy for obtaining quantitative stability estimates}
In this paper we derive \emph{quantitative} error estimates of numerical approximations to scalar conservation laws. This requires quantitative stability estimates. We will consider weak solutions which are entropic for at least one strictly convex entropy. 

In order to obtain quantitative stability estimates, we need to restrict the class of initial data, i.e. we require initial data in $\inisp$ which is a restriction compared to \cite{2017arXiv170905610K}.
However, from a numerical point of view, this does not seem to be a severe restriction since it is not clear how to construct good  approximations of less regular initial data in practice.
Due to restricting ourselves to initial data in $\inisp$, from here on we will call the exact solution we are considering $\bar{u}$ instead of $u$.

The notation used above will be suggestive, but some slight variations are needed. In particular, we want all discontinuities to have some minimal size and this is tricky if we approximate $\bar u^0$ by step functions in areas were it is decreasing with a slope very close to zero. 
Thus, we will pick a threshold $\epsilon >0$ and treat points $x$ with $\partial_x \bar u^0(x)> - \epsilon$ as points where $\bar u^0$ is non-decreasing. This gives rise to the space $\inise$ 

We shall pick $\epsilon>0$ such that smooth solutions to \eqref{system} with slope greater than $- \epsilon$ do not experience shock formation before or at time $T$. In fact, we choose  $\epsilon$ such that no slopes less than $-1$ form.

Our first step is to approximate the initial data $\bar{u}(\cdot,0)$ by some $\hat{u}^0 \in \inise$. 
Since we wish to use the quantitative estimates in  an a posteriori setting, $\hat u^0$ should also live in some finite dimensional space that can be represented numerically, e.g. some space of piecewise polynomials of a certain degree.
On the intervals where $\partial_x \bar{u}(\cdot,0) >- \epsilon$ this can be achieved by first using a standard discretization of $\bar {u}^0$, i.e. projection or interpolation into a piecewise polynomial space and subsequent reconstruction as in \cite{GiesselmannMakridakisPryer}. 
On intervals $(x_i,x_{i+1})$ where $\bar{u}(\cdot,0)$ is a monotonically decreasing Lipschitz function with slope magnitude larger than $\epsilon$, we use step functions of width $\delta$ such that (at $t=0$) we obtain discontinuities with sizes $\mathfrak{s}_i$ between $\delta \cdot \epsilon$ and  $\delta\cdot({\rm Lip} [\bar u^0|_{(x_i,x_{i+1})}]) $ . Note that $\delta>0$ is a parameter that we can choose and that does not have to coincide with the mesh width $h$ in the intervals where we use piecewise polynomials. For these ``front tracking'' shocks (see \Cref{shocks_fig}), the solution is evolving in time according to a simple scalar front tracking process. In this paper, our front tracking solutions will only involve shocks and are thus simple to construct. Compare this to the scalar front tracking algorithm presented in \cite[Section 14.1]{dafermos_big_book} which is capable of handling more general initial data, but which involves some complexity due to eliminating rarefaction waves and instead using composite waves consisting of constant states, admissible shocks, and small inadmissible ``rarefaction'' shocks. Our work compares directly to the result \cite{chen2020uniqueness}, where stability of small-$BV$ solutions, among the class of large data $L^2$ solutions, is shown by utilizing $a$-contraction and the front tracking algorithm for systems of conservation laws. See \Cref{constructuhat} for more details on the construction of $\hat{u}$.

Although shocks are very well-behaved within the context of the theory of shifts, neither the relative entropy method nor the theory of shifts can cope with the \emph{formation} of a shock as the slopes of a solution become increasingly steep. For a scalar conservation law with convex flux, a solution which is strictly decreasing in space will eventually suffer from the catastrophe of the overturning waves.  This is why it is necessary to break the rapidly decreasing parts of $\bar{u}(\cdot,0)$ into step functions to allow for analysis within the theory of shifts.

Based on $\hat{u}^0$, we construct global numerical solutions $\hat v_i$, which are based on the artificial extensions similar to \cite{2017arXiv170905610K} as discussed above (but specialized to the present situation, see \Cref{ext_section}).

As suggested by the notation above and \eqref{psi_firstdef}, in this paper ${\hat{\psi}}$ will be constructed using shift functions such as to keep $\hat{\psi}$ $L^2$-close to $\bar{u}$. However, instead of gluing together exact solutions $v_i$ as in \eqref{psi_firstdef}, we will glue together the $\hat{v}_i$ constructed numerically. Nonetheless, the shifts will follow precisely the same theoretic construction. In the scalar case, the shift functions are simply generalized characteristics of the exact solution $\bar{u}$ (for more on generalized characteristics, see \cite[Chapter 10]{dafermos_big_book}). Thus, the $\hat{\psi}$ becomes an interesting mix between the exact solution $\bar{u}$ and the numerical solution $\hat{u}$. See \Cref{constructpsihat} for details.

\begin{figure}[tb]
      \includegraphics[width=.4\textwidth]{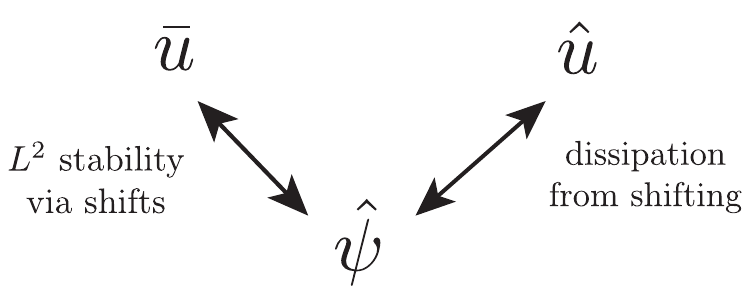}
  \caption{The philosophy of the present paper is to control $\hat\psi-\bar{u}$ via the $L^2$ stability we gain via shifting, while on the other hand measuring $\hat\psi-\hat u$ via the control on the shifts we get due to dissipation.}\label{figphilosophy}
\end{figure}

We are aiming for a quantitative estimate on $\bar{u}-\hat{u}$. To do this, we will in fact give a quantitative estimate on ${\hat{u}(\cdot,t) - {\hat{\psi}}(\cdot,t)}$. Such a quantitative estimate, combined with control on $\norm{{\bar{u}}(\cdot,t)-\hat{\psi}(\cdot,t)}_{L^2}$ due to shifts, will give us a quantitative estimate on $\bar{u}(\cdot,t)-\hat{u}(\cdot,t)$ (see \Cref{figphilosophy}). 

As discussed above, there are two difficulties in deriving a quantitative estimate on ${\hat{u}(\cdot,t) - {\hat{\psi}}(\cdot,t)}$. The first is what to do when collisions between shocks occur in one of $\hat{u}$, $\hat\psi$ but not the other, such that the two functions can no longer be related to each other through a change of variables. This issue is handled numerically by keeping track of all the possibilities for the function $\hat\psi$, based on what we see in the computer during the construction of the numerical solution $\hat u$. 
This process involves remembering shock curves and $\hat v_i$ even after they are no longer needed to construct $\hat{u}$ due to collisions between shocks in $\hat{u}$. In this way, we are building an ``unfolded'' version of $\hat{u}$. As a result of keeping careful track of all the possibilities for the function $\hat\psi$, we gain an understanding about the difference in position, measured by $\Delta_i$ (see \eqref{def_control_shift}), for where each $\hat v_i$ might be revealed in $\hat u$ and $\hat \psi$. This allows us to derive estimates on $\norm{\hat{u}(\cdot,t) - {\hat{\psi}}(\cdot,t)}_{L^1}$.


The second difficulty discussed above in deriving a quantitative estimate on ${\hat{u}(\cdot,t) - {\hat{\psi}}(\cdot,t)}$ is how to control shifts of small shocks. For us, this issue appears when trying to control the position of the very small ``front tracking'' shocks of size $\mathcal{O}(\delta)$ (at least, they are this size at $t=0$). Its resolution involves an ancillary function $\hat{\psi}_{\text{fine}}$ which allows us to gain improved control on small shocks by something of a ``localization'' procedure which minimizes the error in position of a shock due to far away errors in $\hat{u}(\cdot,0)-\bar{u}(\cdot,0)$. With this, we manage to devise improved control on a front tracking algorithm with shifts: a key idea is to use the $L^1$-theory for control on ``$\epsilon$-approximate'' front tracking solutions which have shocks traveling with Rankine-Hugoniot speed \emph{plus an error term} (see \cite{Bressan2000} and also \Cref{L1_diss_lemma}, below). We believe the ideas we introduce here might be applicable to the theory of shifts and $a$-contraction in the well-posedness theory for hyperbolic \emph{systems} of conservation laws.

Let us remark that a fundamental source of error in the numerical solution $\hat{u}$ is due to numerical errors in the $\hat{v}_i$. To glue together the $\hat{v}_i$ to make $\hat{u}$, the Rankine-Hugoniot jump condition is used. More precisely, ODEs are solved numerically (see \eqref{ODEfor_u_hat}). The position of the shock curves thus depends on the pointwise values of the $\hat v_i$ at the exact points the numerical ODE solver sees. These pointwise values are a quantity on whose error we do not have good direct control. 
Roughly speaking, the function $\hat\psi$ corrects this error by always gluing together the $\hat{v}_i$ in the places where they \emph{should have been glued} if it weren't for this numerical error. This is done using the generalized characteristics in the exact solution $\bar{u}$ as the curves of discontinuity -- the ``shift''. This shift in position of the curves of discontinuity between the $\hat{u}$ and $\hat\psi$ causes some dissipation (see \Cref{neg_entropy_diss}), and this allows us to control the magnitude of the shift  (see the estimate \eqref{def_control_shift}). Note that the estimate on the shift \eqref{def_control_shift} relies on the $L^2$ norm of the numerical residual, as well as $\norm{\hat{u}(\cdot,0)-\bar{u}(\cdot,0)}_{L^2}$. It is a striking consequence of the theory of shifts that by controlling non-local integral quantities, we can control the  error in solving the ODE for shock position that results from pointwise errors in the $\hat v_i$. Indeed, the control we gain here is stronger than the control we would gain by transferring $L^2$ estimates to $L^\infty$ estimates on the numerical error in the $\hat v_i$ by using Lipschitz continuity (cf. Section \ref{sec:Linfty}).

\subsubsection{Future work for the systems case}

In this paper, we are careful to construct a theory we hope is flexible enough to extend to the systems case. As a product of using the relative entropy method and theory of shifts, we only use one entropy during our analysis. Furthermore, our analysis is done within the context of the theory of shifts and the framework first developed in \cite{2017arXiv170905610K}. One key tool is the dissipation estimate \Cref{neg_entropy_diss}.

There is hope our work here will generalize to systems. The framework first developed in \cite{2017arXiv170905610K} has had tremendous success when used with $a$-contraction techniques and applied to systems to show stability of small-$BV$ solutions \cite{chen2020uniqueness}. As discussed above, \cite{chen2020uniqueness} utilizes the front tracking algorithm to construct  sequences of $L^2$-stable approximate solutions.  In the present paper, we use the front tracking theory for the analysis of very small shocks in our numerical solution (see \Cref{front_tracking_shocks}). In the systems case, we remark that the front tracking theory is well-adapted for the analysis of small shocks.

Furthermore, the dissipation estimate \Cref{neg_entropy_diss} has near-identical analogues in the $a$-contraction theory (see \cite[Proposition 4.1]{move_entire_solution_system}, for example). Moreover, recently an important technical contribution to the $a$-contraction theory has greatly enhanced the control on small shocks \cite{2020arXiv201209261G}.


\section{Technical preliminary}\label{sec:technical}

For a function $u\in L^\infty(\mathbb{R}\times[0,T))$ that verifies the Strong Trace Property (\Cref{strong_trace_definition}), and a Lipschitz-continuous function $h\colon[0,T]\to\mathbb{R}$, we use the notation $u(h(t)\pm,t)$ to denote the  right- and left-hand traces $u_\pm(t)$ at a time $t$ when $u$ verifies the Strong Trace Property, where $u_\pm(t)$ is within the context of \Cref{strong_trace_definition}. For a time $t$ when $u$ has strong traces,  the values $u_\pm(t)$ are well-defined, and thus the values $u(h(t)\pm,t)$ are also well-defined.

Most of the results in this preliminary section are well-known. However, we include them here (along with proofs) to ensure that we are not using more information than we think. The solution $\bar{u}$ in the Main Theorem (\Cref{main_theorem}) is only assumed to be entropic for the convex entropy $\eta$ -- we do not use the Kruzhkov theory.

\subsection{Front tracking estimate}
The following Lemma will be fundamental to our control on the ``front tracking'' shocks (see \Cref{shocks_fig}).

\begin{lemma}[An estimate on approximate front tracking solutions]\label{L1_diss_lemma}
Consider initial data $v^0\colon\mathbb{R}\to\mathbb{R}$ which is monotonically decreasing, piecewise-constant and has finitely many discontinuities. Let $v$ be the (unique) exact solution to \eqref{system} with initial data $v^0$ and which is monotonically decreasing (in space) for each fixed time. Let $w$ be a function  which is piecewise-constant (in space) and monotonically decreasing (in space) for each fixed time. In particular, there exists $\tilde{N}\in\mathbb{N}$ and Lipschitz functions $\tilde h_1,\ldots,\tilde h_{\tilde {N}}\colon[0,\infty)\to\mathbb{R}$ such that $\tilde h_1(t)\leq\ldots\leq\tilde h_{\tilde {N}}(t)$ for all $t\in[0,\infty)$ (and if $\tilde h_{j}(t_0)=\tilde h_{k}(t_0)$ for some $j,k,t_0$, then $\tilde h_{j}(t)=\tilde h_{k}(t)$ for all $t>t_0$)  and $w(\cdot,t)$ is constant on each interval $[\tilde h_{i}(t),\tilde h_{i+1}(t)]$ for all $i\in\{0,\ldots,N\}$ (when $\tilde h_{i}(t)\neq\tilde h_{i+1}(t)$ and with $\tilde h_0=-\infty,\tilde h_{\tilde {N}+1}=+\infty $). Furthermore, assume

\begin{align}\label{infityassumption}
\norm{v(\cdot,0)-w(\cdot,0)}_{L^1(\mathbb{R})}<\infty.
\end{align}

Then, for almost every $t$,

\begin{multline}
    \frac{d}{dt} \norm{v(\cdot,t)-w(\cdot,t)}_{L^1(\mathbb{R})}
    \\
 \leq  \sum_{\alpha}\Bigg[\abs{w(\tilde{h}_\alpha(t)-,t)-  w(\tilde{h}_\alpha(t)+,t)}\abs{\dot{\tilde{h}}_\alpha-\sigma(w(\tilde{h}_\alpha(t)-,t),w(\tilde{h}_\alpha(t)+,t))}\Bigg].
    \end{multline}
 At time $t$, the summation runs over all $\alpha\in\{1,\ldots,\tilde{N}\}$ with the property that if there exists $\beta\in\{1,\ldots,\tilde{N}\}$ such that $\tilde h_{\alpha}(t)=\tilde h_{\beta}(t)$, then $\alpha\leq\beta$ (this way we are not double counting the dissipation from curves $\tilde h_\alpha$ which are stuck together).
\end{lemma}
\begin{remark}
Remark that $v$ and $w$ are so-called ``$\epsilon$\emph{-approximate}'' front tracking solutions in the sense of \cite{Bressan2000}. The proof of this Lemma is an immediate calculation in the scalar case. A version also holds for $\epsilon$\emph{-approximate} front tracking solutions in the systems case -- see \cite[Theorem 8.2]{Bressan2000}. In particular, \cite[Eqn (8.22)]{Bressan2000} can be roughly summarized as stating that the time derivative of the weighted $L^1$ distance between solutions grows at most like a constant times the velocity error at that time. The weighted $L^1$ norm, which includes a \emph{wave interaction potential} in the systems case, can be replaced by the unweighted $L^1$ norm in the scalar case.
\end{remark}

\begin{proof}
    We give a quick proof, largely following \cite[p.~155]{Bressan2000}. Remark we do not use the Kruzhkov entropies.
    
Let $x_1,\ldots,x_N$ be all the points where $v(\cdot,t)$ or $w(\cdot,t)$ have a jump.

We then have
\begin{multline}\label{disscalcl1}
     \frac{d}{dt} \norm{v(\cdot,t)-w(\cdot,t)}_{L^1(\mathbb{R})}=
     \\
     \sum_{\alpha}\Big[\abs{w(x_\alpha(t)-,t)-  v(x_\alpha(t)-,t)}-\abs{w(x_\alpha(t)+,t)-  v(x_\alpha(t)+,t)}\Big]\dot{x}_\alpha,
\end{multline}
where the sum is over all jumps in $v$ or $w$.

Since $w$ and $v$ are piecewise-constant, for $x\in(x_{\alpha-1},x_{\alpha})$, we immediately have
\begin{multline}\label{addsubtract}
    \abs{w(x,t)-v(x,t)}\sigma(w(x,t),v(x,t)=
    \\
    \abs{w(x_{\alpha-1}+,t)-v(x_{\alpha-1}+,t)}\sigma(w(x_{\alpha-1}+,t),v(x_{\alpha-1}+,t))
    =
    \\
    \abs{w(x_{\alpha}-,t)-v(x_{\alpha}-,t)}\sigma(w(x_{\alpha}-,t),v(x_{\alpha}-,t)).
\end{multline}

Furthermore, from \eqref{infityassumption} we have $\abs{w(x,t)-v(x,t)}\equiv 0$ for $x$ outside a large (but bounded) interval. We can thus add and subtract the terms \eqref{addsubtract} in \eqref{disscalcl1} without changing the total sum.

We receive,

\begin{multline}
     \frac{d}{dt} \norm{v(\cdot,t)-w(\cdot,t)}_{L^1(\mathbb{R})}=
     \\
     \sum_{\alpha}\Bigg[\abs{w(x_\alpha(t)+,t)-  v(x_\alpha(t)+,t)}(\sigma(w(x_\alpha(t)+,t),v(x_\alpha(t)+,t))-\dot{x}_\alpha)
     \\-\abs{w(x_\alpha(t)-,t)-  v(x_\alpha(t)-,t)}(\sigma(w(x_\alpha(t)-,t),v(x_\alpha(t)-,t))-\dot{x}_\alpha)\Bigg].
\end{multline}

Define ${v}^{\alpha\pm}\coloneqq v(x_\alpha(t)\pm,t)$ and $w^{\alpha\pm}\coloneqq w(x_\alpha(t)\pm,t)$.

Furthermore, define 
\begin{multline}\label{addzerotothis}
    E_\alpha\coloneqq \Bigg[\abs{w^{\alpha+}-  {v}^{\alpha+}}(\sigma(w^{\alpha+},{v}^{\alpha+})-\dot{x}_\alpha)
     -\abs{w^{\alpha-}-  {v}^{\alpha-}}(\sigma(w^{\alpha-},{v}^{\alpha-})-\dot{x}_\alpha)\Bigg].
\end{multline}

Then by adding zero, we get
\begin{multline}\label{line740}
  E_\alpha = \Bigg[\abs{w^{\alpha+}-  {v}^{\alpha+}}(\sigma(w^{\alpha+},{v}^{\alpha+})-\sigma(w^{\alpha-},w^{\alpha+}))
  \\
    -\abs{w^{\alpha-}-  {v}^{\alpha-}}(\sigma(w^{\alpha-},{v}^{\alpha-})-\sigma(w^{\alpha-},w^{\alpha+})) \Bigg] 
     \\
     +\Bigg[
     (\sigma(w^{\alpha-},w^{\alpha+})-\dot{x}_\alpha)\Big(\abs{w^{\alpha+}-  {v}^{\alpha+}}
          -\abs{w^{\alpha-}-  {v}^{\alpha-}}\Big)\Bigg]
          \\
          \coloneqq E_{\alpha,1}+E_{\alpha,2}.
\end{multline}

Assume for now  that
\begin{align}
    v(x_\alpha(t)+,t)=v(x_\alpha(t)-,t).
\end{align}

Then, define ${v}^\alpha\coloneqq v(x_\alpha(t)+,t)=v(x_\alpha(t)-,t)$.

Now, consider the following cases:

\begin{enumerate}
    \item[Case 1]$ {v}^\alpha<w^{\alpha+} < w^{\alpha-}$
    \item[Case 2]$w^{\alpha+} < w^{\alpha-} < {v}^\alpha$
    \item[Case 3]$w^{\alpha+} < {v}^\alpha < w^{\alpha-}$
\end{enumerate}

By symmetry $v\longleftrightarrow w$, the cases with a discontinuity in $v$, but not in $w$, are similar; we do not consider them here. 

We now handle the cases:

\uline{Case 1 and Case 2} Case 1 and Case 2 are similar. We handle only Case 1.

We rewrite the first term  ($E_{\alpha,1}$) of \eqref{line740},
\begin{multline}
\abs{w^{\alpha+}-  {v}^\alpha}(\sigma(w^{\alpha+},{v}^\alpha)-\sigma(w^{\alpha-},w^{\alpha+}))
    -\abs{w^{\alpha-}-  {v}^\alpha}(\sigma(w^{\alpha-},{v}^\alpha)-\sigma(w^{\alpha-},w^{\alpha+}))  =
     \\
     \sigma(w^{\alpha-},w^{\alpha+})\Big(\abs{w^{\alpha-}-  {v}^\alpha}-\abs{w^{\alpha+}-  {v}^\alpha}\Big)
     +\abs{w^{\alpha+}-v^\alpha}\sigma(w^{\alpha+},v^\alpha) -\abs{w^{\alpha-}-v^\alpha}\sigma(w^{\alpha-},v^\alpha).
\end{multline}
Then remark that in this case, $\abs{w^{\alpha-}-  {v}^\alpha}-\abs{w^{\alpha+}-  {v}^\alpha}=\abs{w^{\alpha-}-w^{\alpha+}}$ and then from the definition of the Rankine-Hugoniot speed (see \eqref{eq:RH}), we get that the first term of \eqref{line740} is identically zero. 

Thus, we conclude

\begin{equation}
    \abs{E_\alpha}\leq \abs{w^{\alpha-}-  w^{\alpha+}}\abs{\dot{x}_\alpha-\sigma(w^{\alpha-},w^{\alpha+})}.
\end{equation}

\uline{Case 3}

The first term of \eqref{line740} ($E_{\alpha,1}$) is nonpositive due to $\partial_u\sigma(u,v)=\partial_v\sigma(u,v)>0$. 

Thus, we conclude

\begin{equation}
    \abs{\max\{E_\alpha,0\}}\leq \abs{w^{\alpha-}-  w^{\alpha+}}\abs{\dot{x}_\alpha-\sigma(w^{\alpha-},w^{\alpha+})}.
\end{equation}

Assume now that 
\begin{align}
v(x_\alpha(t)+,t)<v(x_\alpha(t)-,t).
\end{align}

We then consider the additional 4 cases:

\begin{enumerate}
    \item[Case 1]$w^{\alpha+}<v^{\alpha+} < v^{\alpha-}<w^{\alpha-}$
    \item[Case 2]$v^{\alpha+}<w^{\alpha+} < w^{\alpha-}<v^{\alpha-}$
    \item[Case 3]$w^{\alpha+}<v^{\alpha+} < w^{\alpha-}<v^{\alpha-}$
    \item[Case 4]$v^{\alpha+}<w^{\alpha+} < v^{\alpha-}<w^{\alpha-}$
\end{enumerate}

Remark that for almost every $t$, either $(v(x_\alpha(t)+,t),v(x_\alpha(t)-,t),\dot{x}_\alpha)$ is a shock verifying Rankine-Hugoniot or $v^{\alpha+}= v^{\alpha-}$. In the case when ${v}$ is in $BV$, this is well-known. See e.g. \cite[Lemma 6]{Leger2011} or \cite[Lemma 5.2]{scalar_move_entire_solution}, where this is shown to hold even under the weaker Strong Trace Property (\Cref{strong_trace_definition}).

Then again by adding zero to \eqref{addzerotothis}, we get 

\begin{multline}\label{line7401}
  E_\alpha = \Bigg[\abs{w^{\alpha+}-  {v}^{\alpha+}}(\sigma(w^{\alpha+},{v}^{\alpha+})-\sigma(v^{\alpha-},v^{\alpha+}))
    \\
    -\abs{w^{\alpha-}-  {v}^{\alpha-}}(\sigma(w^{\alpha-},{v}^{\alpha-})-\sigma(v^{\alpha-},v^{\alpha+})) \Bigg] 
     \\
     +\Bigg[
     (\sigma(v^{\alpha-},v^{\alpha+})-\dot{x}_\alpha)\Big(\abs{w^{\alpha+}-  {v}^{\alpha+}}
          -\abs{w^{\alpha-}-  {v}^{\alpha-}}\Big)\Bigg]
          \\
          \coloneqq \tilde{E}_{\alpha,1}+\tilde{E}_{\alpha,2}.
\end{multline}

Then, in each of the Cases 1-4, we have $\tilde{E}_{\alpha,2}=0$ due to $(v(x_\alpha(t)+,t),v(x_\alpha(t)-,t),\dot{x}_\alpha)$ being a Rankine-Hugoniot shock, and {$\tilde{E}_{\alpha,1}<0$} due to $\partial_u\sigma(u,v)=\partial_v\sigma(u,v)>0$.

This concludes the proof.
\end{proof}

\subsection{Control on Lipschitz constant of a smooth solution}

The following Lemma will give control of the Lipschitz constant of a smooth solution to \eqref{system}.

\begin{lemma}[control on the Lipschitz constant of a smooth solution]\label{control_lipschitz_lem}
Let $\bar{u}^0\colon\mathbb{R}\to\mathbb{R}$ be in $L^\infty$ and Lipschitz-continuous. Let $\bar{u}$ be the unique  entropy weak solution to \eqref{system} with initial data $\bar{u}^0$. As long as $\bar{u}$ is Lipschitz-continuous it verifies 
\begin{align}\label{control_lipschitz}
    \rm{Lip}[\bar{u}(\cdot,t)]\leq \max\Bigg\{\abs{\frac{1}{\frac{1}{\sup (\bar u^0)'} +  t \inf A'' }},\abs{\frac{1}{\frac{1}{\inf (\bar u^0)'} +  t \inf A'' }}\Bigg\},
\end{align}
where $A''$ is the second derivative of the flux function $A$. Here, $\sfrac{1}{0} =+\infty$ and $\sfrac{1}{+\infty}=0$.
\end{lemma}
\begin{remark}

This result follows immediately from the proof of \cite[Theorem 6.1.1]{dafermos_big_book}.

We give a brief summary of the argument, specialized to the scalar conservation laws. We remark that we do not use the large family of Kruzhkov entropies.
\end{remark}

\begin{proof}

As long as the solution is Lipschitz-continuous it is given by the method of characteristics. 

We can express $\bar u$ implicitly as
\[ \bar u(x,t)= \bar u^0 (x- A'(\bar u(x,t))t)\]
and taking $x$-derivatives on both sides leads to
\[ \bar u_x(x,t)= (\bar u^0)'(x- A'(\bar u(x,t)t)  \cdot [ 1 - A''(\bar u(x,t)) t \bar u_x (x,t)]  \]
which (for $(\bar u^0)'\neq0$) can be rearranged as
\begin{align}\label{derivubar}
\bar u_x(x,t)= \frac{1}{ \frac{1}{ (\bar u^0)'(x- A'(\bar u(x,t)t) } +   A''(\bar u(x,t)) t  }.
\end{align} 
\end{proof}

\subsection{Extending the relative entropy inequality}
The results from this section will be instrumental in estimating $\norm{\bar u - \hat \psi}_{L^2}$.

\begin{lemma}[An extension of the entropy inequality]\label{global_entropy_dissipation_rate_systems}
Fix $T>0$. For $i=1,2$, consider  weak solutions $\hat{u}_i \in L^\infty(\mathbb{R}\times[0,T))$  to \eqref{system-numerical} with residuals $R_i \in L^\infty((0,T) \times \mathbb{R})$, where the residuals are defined by 
\begin{equation}
R_i:= \partial_t \hat u_i + \partial_x A(\hat u_i).
\end{equation} 
Moreover, assume that $\hat{u}_2$ is entropic for the strictly convex entropy $\eta\in C^2(\mathbb{R})$. Further, assume that $\hat{u}_1$ is Lipschitz.  


Then for all positive, Lipschitz-continuous test functions $\phi\colon\mathbb{R}\times[0,T)\to\mathbb{R}$ with compact support, we have
\begin{equation}
\begin{aligned}\label{combined1}
&\int\limits_{0}^{T} \int\limits_{-\infty}^{\infty} [\partial_t \phi \eta(\hat{u}_2(x,t)|\hat{u}_1(x,t))+\partial_x \phi q(\hat{u}_2(x,t);\hat{u}_1(x,t))]\,dxdt \\
&\hspace{.5in}+  \int\limits_{-\infty}^{\infty}\phi(x,0)\eta(\hat{u}_2^0(x)|\hat{u}_1^0(x))\,dx \\
&\hspace{.8in}\geq
\\
&\hspace{-.3in}\int\limits_{0}^{T} \int\limits_{-\infty}^{\infty}\phi\Bigg[R_2(x,t)(\eta'(\hat{u}_1(x,t))-\eta'(\hat{u}_2(x,t)))+\partial_x\hat{u}_1(x,t)\eta''(\hat{u}_1(x,t)) A(\hat{u}_2(x,t)|\hat{u}_1(x,t))
\\
&\hspace{1.2in}+R_1(x,t)\eta''(\hat{u}_1(x,t))[\hat{u}_2(x,t)-\hat{u}_1(x,t)]
\Bigg]\,dxdt.
\end{aligned}
\end{equation}

I.e., we have the following inequality in the sense of distributions 
\begin{multline}\label{combined1_distributional}
\partial_t \eta(\hat{u}_2(x,t)|\hat{u}_1(x,t))+\partial_x q(\hat{u}_2(x,t);\hat{u}_1(x,t))
\\
\leq -R_2(x,t)(\eta'(\hat{u}_1(x,t))-\eta'(\hat{u}_2(x,t)))-\partial_x\hat{u}_1(x,t)\eta''(\hat{u}_1(x,t)) A(\hat{u}_2(x,t)|\hat{u}_1(x,t))
\\
-R_1(x,t)\eta''(\hat{u}_1(x,t))[\hat{u}_2(x,t)-\hat{u}_1(x,t)].
\end{multline}

\end{lemma}

The equation \eqref{combined1} bears resemblance to the pivotal inequality employed in Dafermos's demonstration of weak-strong stability, which provides a relative formulation of the entropy inequality (refer to equation (5.2.10) in \cite[p.~122-5]{dafermos_big_book}). However, unlike the original formulation, \eqref{combined1} takes into account residuals present in either of the two solutions. The derivation of \eqref{combined1} extends the well-known weak-strong stability proof by Dafermos and DiPerna \cite[p.~122-5]{dafermos_big_book}. 

For the reader's convenience, the proof of \Cref{global_entropy_dissipation_rate_systems} is in the appendix (\Cref{sec:proof_global_entropy_dissipation_rate_systems}).

\subsection{Control on shifts}
In order to control $\norm{\hat u - \hat \psi}_{L^1}$ we need to 
control shifts. This will be established using  the following lemma, which says that if a discontinuity in the second slot of the relative entropy is moving at the Rankine-Hugoniot speed of the discontinuity in the first slot of the relative entropy, then there is a creation of negative entropy dissipation proportional to the difference between the natural Rankine-Hugoniot speed of the shock in the second slot and the Rankine-Hugoniot speed of the shock in the first slot.

\begin{lemma}[negativity of entropy dissipation \protect{\cite[Proposition 2.1]{2020arXiv201103710C}}]\label{neg_entropy_diss}
Fix $\mathfrak{s},B>0$.  Let $u_+,u_-,\bar{u}_+,\bar{u}_-\in[-B,B]$ verify $u_-\geq u_+$, $\bar{u}_- - \bar{u}_+\geq \mathfrak{s}$.

Then, 
\begin{multline}\label{diss_neg_formula}
q(u_+;\bar{u}_+)-q(u_-;\bar{u}_-) - \sigma(u_+,u_-)\big(\eta(u_+|\bar{u}_+)-\eta(u_-|\bar{u}_-)\big)\\
\leq -\frac{1}{12}\inf{A''}\inf{\eta''}\mathfrak{s} \big((u_+-\bar{u}_+)^2+(u_--\bar{u}_-)^2\big),
\end{multline}
where $\sigma$ is defined in \eqref{eq:RH}, and the infima $\inf{A''}$ and $\inf{\eta''}$ run over the set $[-B,B]$.
\end{lemma}

The proof of \Cref{neg_entropy_diss} follows almost immediately from  \cite{2020arXiv201103710C}, where the result is given (and proved) for the particular entropy $\eta(u)=\frac{u^2}{2}$. See also \cite[Lemma 4.1]{scalar_move_entire_solution} and \cite[p.~9-10]{serre_vasseur}.

\begin{remark}
The $\mathfrak{s}$ in \eqref{diss_neg_formula} shows that smaller shocks dissipate less relative entropy and, therefore, we will have less control on positions of small shocks than on  positions of large shocks.
\end{remark}

\begin{remark}
    For the systems case, the analogue of the dissipation estimate \Cref{neg_entropy_diss} would be along the lines of \cite[Proposition 4.1]{move_entire_solution_system}, for example. 
\end{remark}

\subsection{Pointwise difference between two solutions}
In order to obtain good control of shifts we need that the shocks being shifted do not become too small (see \Cref{neg_entropy_diss}). In practical computations, we can monitor, at any time, the minimal size of each shock by computing $\inf ( \hat v_i (\cdot,t) - \hat v_{i+1} ( \cdot,t )) $ where the infimum is taken over the set of points where the shock might be. 
In addition, we would like to argue that it is plausible that the difference $\hat v_i - \hat v_{i+1}$ does not decay quickly, since there are bounds for the speed of decay of $ v_i -  v_{i+1}$  for exact solutions $ v_i, v_{i+1}$. Here $\hat v_i, \hat v_{i+1},  v_i, v_{i+1}$ are in the context of the Main Theorem (\Cref{main_theorem}).

The pointwise difference between two solutions can be decreasing. However, we can control the rate at which the gap narrows:

\begin{lemma}[Rate of decrease on the gap between two solutions]\label{gap_lemma}
Fix $t,\rho>0$. Let $v_1^0\colon\mathbb{R}\to\mathbb{R}$ and  $v_2^0\colon\mathbb{R}\to\mathbb{R}$ be Lipschitz-continuous functions such that $v_1^0(x)-v_2^0(x)\geq \rho$ for all $x\in\mathbb{R}$. Assume that on the interval $[0,t)$,  unique classical solutions $v_i$  to the system \eqref{system} exist with initial data $v_i^0$ (with $i=1,2$). Then the difference between these solutions satisfies 
\begin{align}\label{gap_result}
v_1(x,t)-v_2(x,t)\geq  \frac{\rho}{1+
\mathfrak{A} \min\{{\rm Lip}[v_1^0],{ \rm Lip}[v_2^0]\} t}
\end{align}
for all $x \in\mathbb{R}$. Here, $\mathfrak{A}:= \sup A''$ where the $\sup$ is taken over the sets of values of $v_1, v_2$.
\end{lemma}

\begin{proof}

We are careful not to use the Kruzhkov entropies.

We assume that $t>0$, otherwise the proof is immediate.

It is well known that the solutions $v_1$ and $v_2$ are given by the method of characteristics. One possible reference, which does not use the Kruzhkov theory, is \cite[p.~176-177]{dafermos_big_book}. See also the proof of \cite[Lemma 2.1]{2017arXiv170905610K}.

Then fix $x\in\mathbb{R}$. Since the solutions $v_1$ and $v_2$ are given by the method of characteristics, there exists $x_1^0$ and $x_2^0$ such that
\begin{equation}
\begin{aligned}\label{transport_values817}
v_1(x,t)=v_1^0(x_1^0) \\
v_2(x,t)=v_2^0(x_2^0),
\end{aligned}
\end{equation}
and
\begin{align} \label{transport_values_formula817}
x=x_1^0+A'(v_1^0(x_1^0))t=x_2^0+A'(v_2^0(x_2^0))t,
\end{align}
where $A'$ is the derivative of the flux function.

Set $\tilde \rho= v_1(x,t) - v_2(x,t) = v_1^0(x_1^0) - v_2^0(x_2^0)$.
Then,
\[ | x_1^0 - x_2^0| = |A'(v_2^0(x_2^0))t - A'(v_1^0(x_1^0))t|  \leq \mathfrak{A} t \tilde \rho\]
and thus,
\[ \tilde \rho = v_1^0(x_1^0) - v_2^0(x_2^0) \geq v_1^0(x_1^0) - v_2^0(x_1^0) - |v_2^0(x_1^0) - v_2^0(x_2^0)|
\geq \rho - \mathfrak{A}  { \rm Lip}[v_2^0]\tilde \rho t.\]
Similarly,
\[ \tilde \rho = v_1^0(x_1^0) - v_2^0(x_2^0) \geq v_1^0(x_2^0) - v_2^0(x_2^0) - |v_1^0(x_1^0) - v_1^0(x_2^0)|
\geq \rho - \mathfrak{A}  { \rm Lip}[v_1^0]\tilde \rho t,\]
so that
\begin{align}
\tilde \rho \geq  \frac{\rho}{1+
\mathfrak{A} \min\{{\rm Lip}[v_1^0],{ \rm Lip}[v_2^0]\} t}.
\end{align}
\end{proof}

\subsection{$L^\infty$-control on numerical solutions} \label{sec:Linfty}

The following Lemma will help us reduce the computational cost of our numerical scheme in \Cref{do_not_sim_remark}.

\begin{lemma}[$L^\infty$-control on Lipschitz-continuous numerical solutions]\label{l_infty_control_lemma}
Fix $T,S>0$. Let $\bar{u}^0\colon\mathbb{R}\to\mathbb{R}$ be in $L^\infty(\mathbb{R})$, Lipschitz-continuous and nondecreasing. Let $\bar{u}$ be the unique  entropy weak solution to \eqref{system} with initial data $\bar{u}^0$. Let $\hat{u}$ be a Lipschitz-continuous solution to \eqref{system-numerical} with initial data $\hat{u}^0$ and residual $R$, defined by $R\coloneqq \partial_t \hat{u}+\partial_x A(\hat{u})$. Then, we have the following $L^\infty$-type estimate:
\begin{multline}\label{l_infty_control_result}
 \norm{\hat{u}(\cdot,T)-\bar{u}(\cdot,T)}_{L^\infty(-S,S)}^3
 \\
  \leq 8\big(\rm{Lip}[\bar{u}\restriction_{(-S,S)}(\cdot,T)]+\rm{Lip}[\hat{u}\restriction_{(-S,S)}(\cdot,T)]\big) 
  \\
  \cdot Ce^{CT}\Bigg(\norm{\hat{u}(\cdot,0)-\bar{u}(\cdot,0)}^2_{L^2(-S-sT,S+sT)}
  +\int\limits_0^T\int\limits_{-S-s(T-t)}^{S+s(T-t)}R^2(x,t)\,dxdt\Bigg),
\end{multline}
where $s,C>0$ are constants that depends only on $A$, $\eta$, $\norm{\bar{u}}_{L^\infty}$, and $\norm{\hat{u}}_{L^\infty}$.
\end{lemma}
\begin{remark}
For the numerical solution $\hat{u}$, the Lipschitz constant can be calculated a posteriori. For the exact solution $\bar{u}$,  we can use \Cref{control_lipschitz_lem} to estimate the Lipschitz constant.
\end{remark}
\begin{proof}

We basically follow the famous weak-strong stability proof of Dafermos and DiPerna (see for example \cite{dafermos_big_book}). We work within the cone of information.  We define
\begin{align}
h_{\text{left}}(t)\coloneqq -S-s(T-t),\\
h_{\text{right}}(t)\coloneqq S+s(T-t),
\end{align}
where $s>0$ is chosen such that $\abs{q(a;b)}\leq s\eta(a|b)$ for all $a,b$ verifying 
\begin{align}
    \abs{a},\abs{b}\leq\max\{\norm{\bar{u}}_{L^\infty},\norm{\hat{u}}_{L^\infty}\}.
\end{align}

Then, we integrate \eqref{combined1_distributional} on the region 
\begin{align}
S\coloneqq \{(x,t) \in \mathbb{R}\times[0,T] | h_{\text{left}}(t) < x < h_{\text{right}}(t)\}.
\end{align}

Within the context of \eqref{combined1_distributional}, we choose $\hat{u}_1\coloneqq \bar{u}$, $R_1\coloneqq 0$, $\hat{u}_2\coloneqq \hat{u}$, and $R_2\coloneqq R$.

This gives

\begin{multline}\label{combined1_distributional_infinity}
\int\limits_{-S}^{S} \eta(\hat{u}(x,T)|\bar{u}(x,T))\,dx\leq \int\limits_{-S-sT}^{S+sT} \eta(\hat{u}^0(x)|\bar{u}^0(x))\,dx
\\
+\int\limits_0^T\int\limits_{h_{\text{left}}(t)}^{h_{\text{right}}(t)} -R(x,t)(\eta'(\bar{u}(x,t))-\eta'(\hat{u}(x,t)))-\partial_x\bar{u}(x,t)\eta''(\bar{u}(x,t)) A(\hat{u}(x,t)|\bar{u}(x,t))\,dxdt.
\end{multline}
Remark that we have used that $\dot{h}_{\text{left}}=s$, $\dot{h}_{\text{right}}=-s$ and $s$ is chosen such that $\abs{q(a;b)}\leq s\eta(a|b)$ for all $a,b$ verifying $\abs{a},\abs{b}\leq\max\{\norm{\bar{u}}_{L^\infty},\norm{\hat{u}}_{L^\infty}\}$.


It is classical  that due to $\bar{u}^0$ being non-decreasing, $\bar{u}(\cdot,t)$ will be non-decreasing for all $t>0$ (for one source of a proof of this without using the Kruzhkov theory, see \cite[Lemma 2.1]{2017arXiv170905610K}). Remark also that due to the convexity of the flux $A$, $A(\hat{u}(x,t)|\bar{u}(x,t))\geq0$. Further, recall that $\eta''>0$.  Thus, 
\begin{align}
-\partial_x\bar{u}(x,t)\eta''(\bar{u}(x,t)) A(\hat{u}(x,t)|\bar{u}(x,t))\leq 0.
\end{align}
Thus, \eqref{combined1_distributional_infinity} simplifies to
\begin{multline}\label{combined2_distributional_infinity}
\int\limits_{-S}^{S} \eta(\hat{u}(x,T)|\bar{u}(x,T))\,dx\leq \int\limits_{-S-sT}^{S+sT} \eta(\hat{u}^0(x)|\bar{u}^0(x))\,dx
\\
- \int\limits_0^T\int\limits_{h_{\text{left}}(t)}^{h_{\text{right}}(t)}R(x,t)(\eta'(\bar{u}(x,t))-\eta'(\hat{u}(x,t)))\,dxdt.
\end{multline}

From \eqref{combined2_distributional_infinity}, Hölder's inequality, and Young's inequality we get
\begin{equation}
\begin{aligned}\label{gronwall_this}
  \int\limits_{-S}^{S} \eta(\hat{u}(x,T)|\bar{u}(x,T))\,dx\leq  \int\limits_{-S-sT}^{S+sT}\eta(\hat{u}^0(x)|\bar{u}^0(x))\,dx
+ \frac{1}{2}{\int\limits_0^T\int\limits_{h_{\text{left}}(t)}^{h_{\text{right}}(t)}R^2(x,t)\,dxdt}
\\
+\frac{1}{2}{\int\limits_0^T\int\limits_{h_{\text{left}}(t)}^{h_{\text{right}}(t)}(\eta'(\bar{u}(x,t))-\eta'(\hat{u}(x,t)))^2\,dxdt}.
\end{aligned}
\end{equation}

Then, we apply the Gronwall technique to \eqref{gronwall_this}. We recall \eqref{control_rel_entropy}.
This yields
\begin{multline}\label{result_gronwall_infty}
\norm{\hat{u}(\cdot,T)-\bar{u}(\cdot,T)}^2_{L^2(-S,S)}
\\
\leq Ce^{CT}\Bigg(\norm{\hat{u}(\cdot,0)-\bar{u}(\cdot,0)}^2_{L^2(-S-sT,S+sT)}+\int\limits_0^T\int\limits_{h_{\text{left}}(t)}^{h_{\text{right}}(t)}R^2(x,t)\,dxdt\Bigg),
\end{multline}
where the constant $C>0$ depends only on $\eta$, $\norm{\bar{u}}_{L^\infty}$, and $\norm{\hat{u}}_{L^\infty}$.

On the other hand, we can use  $\norm{\hat{u}(\cdot,T)-\bar{u}(\cdot,T)}_{L^\infty}$ and the Lipschitz continuity of $\hat{u}$ and $\bar{u}$ to control the $L^2$ norm from below:

Let $D\coloneqq \norm{\hat{u}(\cdot,T)-\bar{u}(\cdot,T)}_{L^\infty(-S,S)}$. Then, due to Lipschitz continuity of $\hat{u}$ and $\bar{u}$, we know that  $\abs{\hat{u}(x,T)-\bar{u}(x,T)}>\frac{D}{2}$ on an interval of length at least $L$, where
\begin{align}
L\coloneqq \frac{\frac{D}{2}}{\rm{Lip}[\bar{u}\restriction_{(-S,S)}(\cdot,T)]+\rm{Lip}[\hat{u}\restriction_{(-S,S)}(\cdot,T)]}.
\end{align}

Then, 
\begin{align}\label{control_l2_below}
\norm{\hat{u}(\cdot,T)-\bar{u}(\cdot,T)}^2_{L^2}\geq L\big(\frac{D}{2}\big)^2=\frac{D^3}{8\big(\rm{Lip}[\bar{u}\restriction_{(-S,S)}(\cdot,T)]+\rm{Lip}[\hat{u}\restriction_{(-S,S)}(\cdot,T)]\big)}.
\end{align}

Then, we combine \eqref{result_gronwall_infty} and \eqref{control_l2_below} to yield
\begin{multline}
L\big(\frac{D}{2}\big)^2=\frac{D^3}{8\big(\rm{Lip}[\bar{u}\restriction_{(-S,S)}(\cdot,T)]+\rm{Lip}[\hat{u}\restriction_{(-S,S)}(\cdot,T)]\big)} 
\\
\leq Ce^{CT}\Bigg(\norm{\hat{u}(\cdot,0)-\bar{u}(\cdot,0)}^2_{L^2(-S-sT,S+sT)}+\int\limits_0^T\int\limits_{h_{\text{left}}(t)}^{h_{\text{right}}(t)}R^2(x,t)\,dxdt\Bigg).
\end{multline}

This gives \eqref{l_infty_control_result}.
\end{proof}

\section{Construction of  $\hat{u}$ and ${\hat{\psi}}$}\label{construction_section}

In this section, we detail how the function $\hat{u}$  is constructed in our numerical experiments. We also give the (theoretic) construction of the ancillary function ${\hat{\psi}}$.

\vspace{.1in}

We approximate initial data $\bar{u}^0 \in \inisp$, by some $\tilde u^0 \in \inise$  by decomposing the rapidly decreasing parts of $\bar{u}^0$ (i.e., intervals where $\partial_x\bar{u}^0<-\epsilon$) into step functions with steps of length $\delta>0$ where $\delta $ is a parameter that we can choose. Let $\tilde u^0$ have discontinuities at points $x_1, \dots, x_N$.  Let $x_0=-\infty$ and $x_{N+1}=+\infty$.
Then, we extend each $\tilde u^0|_{(x_i, x_{i+1})}$ (for $i=0,\ldots,N+1$) by the method described in \Cref{ext_section}
to a function $v_i^0$ on all of $\mathbb{R}$. Let $v_i$ denote the exact (smooth) solution to \eqref{system} emanating from the initial data $v_i^0$.

From now on, we distinguish the set of intervals $(x_i, x_{i+1})$ that correspond to intervals in which $\partial_x \bar u^0$ is bounded from below by $- \epsilon$  from the set of intervals that are sub-intervals of intervals on which $\partial_x \bar u^0$ is bounded from above by $- \epsilon$. We denote these sets of intervals by $\mathcal{I}_{nnd}$ (nnd=nearly non-decreasing) and $\mathcal{I}_{rd}$ (rd=rapidly decreasing), respectively.

For intervals $(x_i, x_{i+1}) \in \mathcal{I}_{nnd} $, we project $v_i^0$ onto a space of piece-wise polynomials with mesh width $h>0$, which leads to functions $u^0_{h,i}$ that can be evolved by a numerical scheme. Reconstructing them according to the strategy given in \cite{GiesselmannMakridakisPryer}, gives us Lipschitz functions $\hat v_i^0$ on all of $\mathbb{R}$. Afterwards, we evolve the functions  $u^0_{h,i}$ using Runge-Kutta discontinuous Galerkin schemes, and reconstruct them based on the techniques described in \cite{DednerGiesselmann} -- giving rise to functions $\hat v_i$. It was demonstrated in \cite{DednerGiesselmann} that these functions have residuals $R_i$ (see \eqref{system-numerical}) of size $h^{q+1}$ {in the $L^2$-norm} if polynomials of degree $q$ and upwind fluxes  are used in the dG scheme, \emph{provided the problems with the extended initial data  $v_i^0$  have smooth solutions.}

If the $v_i$ have kinks, i.e. points with different left- and right-sided derivatives in $x$, this reduces the convergence order of the $L^2$ norm of residuals for meshes that are not aligned with the kinks.

To mitigate this issue, in \Cref{ext_section}, we construct the global extensions of the initial data in such a way that the $v_i$ are $C^1$ in the regions where they are expected to be revealed\footnote{Remark that by \cite[Theorem 6.1.1]{dafermos_big_book}, if $v_i^0$ is $C^1$, then so is $v_i$ for as long as $v_i$ is Lipschitz.} in either $\hat u$ or $\hat \psi$.

For the $v_i^0$ (based on the constructions from \Cref{ext_section}) corresponding to intervals $(x_{i-1}, x_{i})\in  \mathcal{I}_{rd} $ (rapidly decreasing intervals), we do not actually numerically simulate these $v_i^0$ in the computer and instead assume (in the numerical scheme we implement in the computer) the $\hat v_i$ are globally constant. We postpone the discussion of these $v_i^0$ to \Cref{constructuhat} below.

\begin{remark}
It is noticeable that we do not have many requirements on which numerical solver {and reconstruction procedure} we use to create numerical solutions {and reconstructions} for our analysis to be applicable. The solver needs to be such that we can compute reconstructions of the numerical solution that preserve the pointwise ordering of the different $\hat{v}_i$ (see \eqref{ordering_vs}). 
Remark that for the exact solution, the ordering 
follows from the 
ordering of the initial data (see \Cref{gap_lemma}).  See also \Cref{ordering_remark} after the Main Theorem (\Cref{main_theorem}).
\end{remark}

\subsection{Construction of $\hat u$}\label{constructuhat}

The numerical solution $\hat{u}$ is constructed from the $\hat v_i$ as follows. For those discontinuities in $\hat u^0$ that correspond to ``large'' shocks and ``boundary'' shocks (see \Cref{shocks_fig})
the following ODEs are solved numerically starting at time $t=0$, e.g. by some Runge-Kutta method:
\begin{align}\label{ODEfor_u_hat}
    \begin{cases}
    \dot{\hat{h}}_i(t)=\sigma(\hat{v}_i(\hat{h}_i(t),t),\hat{v}_{i+1}(\hat{h}_i(t),t)) \text{ for } t>0,\\
    \hat{h}_i(0)=x_i,
    \end{cases}
\end{align}
where the $x_i$ are as in the definition of the space $\inise$ (see \eqref{LIWASe}).
Since $\sigma$ and the $\hat{v}_i$ are at least Lipschitz this problem is well-posed. If the numerical scheme for computing the $\hat{v}_i$ is higher order we should use a Runge-Kutta scheme of sufficient order so that its contribution to the position error in $\hat{h}_i$ is smaller than other sources of position errors.  Remark that for the ``front tracking'' shocks in $\hat u$, due to the $\hat v_\ell$ and $\hat v_r$, for the left- and right-hand states of such shocks, respectively, being globally constant (see \Cref{constructuhat}) we do not need Runge-Kutta and instead can employ a simple scalar front tracking scheme.

Consider the first time $\hat{t}_1$ such that $\hat{h}_i(t_1)=\hat{h}_{i+1}(t_1)$ for some $i$.
Then we define for $0\leq t\leq t_1$,
\begin{align}
\hat{u}(x,t)\coloneqq
    \begin{cases}
    \hat{v}_1(x,t) &\text{if } x<\hat{h}_1(t),\\
    \hat{v}_2(x,t) &\text{if } \hat{h}_1(t)<x<\hat{h}_2(t),\\
    &\vdots\\
    \hat{v}_{N+1}(x,t) &\text{if } \hat{h}_N(t)<x.\\
    \end{cases}
\end{align}
Note that the $\hat{h}_i$ are generalized characteristics for $\hat{u}$. 

At time $\hat{t}_1$, at least two of the curves $\hat{h}_i$ have touched. In particular, $\hat{h}_i(\hat{t}_1)=\hat{h}_{i+1}(\hat{t}_1)$ for some $i$. For these two curves that touch, we forget about the $\hat{v}_{i+1}$ that was between them. For time $t>\hat{t}_1$, we redefine $\hat{h}_i$ to solve the ODE
\begin{align}
    \dot{\hat{h}}_i(t)=\sigma(\hat{v}_i(\hat{h}_i(t),t),\hat{v}_{i+2}(\hat{h}_i(t),t)) \text{ for } t>\hat{t}_1,
\end{align}
keeping the value of $\hat{h}_i$ at $t=\hat{t}_1$ unchanged.
Further, we define 
\begin{align}\label{label_glue}
 \hat{h}_i(t)=\hat{h}_{i+1}(t)
\end{align} 
for all $t>\hat{t}_1$. In effect, we are gluing the two curves together as soon as they touch, forgetting about the  $\hat{v}_{i+1}$ that was between them and redefining the ODE accordingly. We do a similar thing if $\hat{h}_i(t)$, $\hat{h}_{i+1}(t)$ and $\hat{h}_{i+2}(t)$ touch simultaneously at time $\hat{t}_1$: for time $t>\hat{t}_1$, we redefine $\hat{h}_i$ to solve the ODE
\begin{align}
    \dot{\hat{h}}_i(t)=\sigma(\hat{v}_i(\hat{h}_i(t),t),\hat{v}_{i+3}(\hat{h}_i(t),t)) \text{ for } t>\hat{t}_1,
\end{align}
keeping the value of $\hat{h}_i$ at $t=\hat{t}_1$ unchanged.
Further, we define 
\begin{align}
 \hat{h}_i(t)=\hat{h}_{i+1}(t)=\hat{h}_{i+2}(t)
\end{align} 
for all $t>\hat{t}_1$. Similarly if four curves touch simultaneously, etc.

We do this for each set of curves of discontinuity in $\hat{u}$ that touch at time $\hat{t}_1$.
This inductive process defines the $\hat{u}$. Remark how the labeling given by \eqref{label_glue} allows to refer to the same curve by the same index, even after collisions (and similarly for $\hat\psi$, which we construct below).

\subsubsection{The $\hat v_i$ from rapidly decreasing intervals}\label{do_not_sim_remark}
    For the $v_i^0$ (based on the constructions from \Cref{ext_section}) corresponding to intervals $(x_{i-1}, x_{i})\in  \mathcal{I}_{rd} $ (rapidly decreasing intervals), we do not actually numerically simulate these $v_i^0$. Simulating these $v_i^0$ would require simulating $\mathcal{O}(\delta^{-1})$ numerical solutions, i.e. the number of numerical solutions would increase with mesh refinement, tremendously increasing the computational costs.
    
    However, we cannot simply take $\hat{v}_i\coloneqq v_i$, where $v_i$ is the exact entropic solution with initial data $v_i^0$, because then we do not know (even a posteriori) if the ordering condition $\hat v_j(x,t) > \hat v_{j+1}(x,t)$  (for all $j$ and for all $(x,t)$) is verified\footnote{For the scalar conservation laws it is possible to use the Lax-Ole\u{\i}nik formula (see \cite{MR2597943}) for the exact solution $v_i$ to check a posteriori the ordering condition $\hat v_j(x,t) > \hat v_{j+1}(x,t)$ for all $j$ (with $\hat{v}_i\coloneqq v_i$), but this is extremely computationally expensive and has no analogue for the general systems case.}.

    Instead, we do the following. Consider a global extension $v_i^0$ (see \Cref{ext_section})
    for an interval $(x_{i-1}, x_{i})\in  \mathcal{I}_{rd} $. The $v_i^0$ is given by \eqref{middle_case}, where the slope $M$ is given by \eqref{slopeM}. Then we numerically simulate a solution to \eqref{system} with initial data 

\begin{align}\label{data_for_Lambda}
    v^0(x_{i-1}+)+(x-x_J^i) M
\end{align}
globally on $\mathbb{R}$, where the value $x^i_J$ is defined in \eqref{x_J_def}. Call the numerical solution $\hat \Lambda$. For reference, we will call the exact solution with initial data \eqref{data_for_Lambda} $\Lambda$. Remark that in the context of \Cref{ext_section}, $\hat u^0$ is playing the role of $v^0$.

Then, by simply translating the solution $\hat\Lambda$ in space, we can obtain also a numerical solution with initial data 
\begin{align}
v^0(x_i-)+(x-x_I^i ) M 
\end{align}
globally on $\mathbb{R}$, where the value $x^i_I$ is defined in \eqref{x_I_def}. Call this solution $\hat P$.

Further, based on \eqref{middle_case}, we can define a numerical version of $v_i$ as follows
\begin{multline}\label{numerical_small}
\hat{v}_i(x,t)\coloneqq \min\Bigg\{\max\Big\{v^0(x_i-)+(x_I^i-x_i) \frac{\mathrm{d}-}{\mathrm{d}x}v^0(x_{i}),\\ \sup_{x\in(x_i,\infty)}v^0(x)+\sum_{k=i}^{N}(v^0(x_k-)-v^0(x_k+))\\+\sum_{k={i+1}}^{N}\Big((x_I^k-x_k)\abs{\frac{\mathrm{d} -}{\mathrm{d}x}v^0(x_k)}\Big)\Big\},\\
\max\Big\{\hat P,\min\{v^0(\frac{x_{i-1}+x_{i}}{2}),\hat\Lambda\},\\
\min\Big\{v^0(x_{i-1}+)+(x_J^i-x_{i-1})\frac{\mathrm{d}+}{\mathrm{d}x}v^0(x_{i-1}),
\\\inf_{x\in(-\infty,x_i)}v^0(x)-\sum_{k=1}^{i-1}(v^0(x_k-)-v^0(x_k+)) \\-\sum_{k=3}^{i} \Big((x_J^{k-1}-x_{k-2})\abs{\frac{\mathrm{d}+}{\mathrm{d}x}v^0(x_{k-2})}\Big)\Big\}\Big\}\Bigg\},
\end{multline}
where $\frac{\mathrm{d}\pm}{\mathrm{d}x}v^0$ are defined in \eqref{define_sided_deriv}.
Remark the term  $(x_J^i-x_{i-1})\frac{\mathrm{d}+}{\mathrm{d}x}v^0(x_{i-1})$ from \eqref{middle_case} drops out because $\frac{\mathrm{d}+}{\mathrm{d}x}v^0(x_{i-1})=0$. Similarly, remark that the term $(x_I^i-x_i) \frac{\mathrm{d}-}{\mathrm{d}x}v^0(x_{i})$ from \eqref{middle_case} drops out because $\frac{\mathrm{d}-}{\mathrm{d}x}v^0(x_{i})=0$.

Notice that if there is no numerical error in $\hat P$ and $\hat \Lambda$, then $\hat v_i\equiv v_i$ (for a schematic diagram of the $v_i$, see \Cref{ext_fig} below).

To construct the other $\hat v_j$  corresponding to rapidly decreasing pieces of $\bar{u}^0$, without doing any more numerical simulations, simply remark that we can translate $\hat P$ and $\hat\Lambda$ in space (to match the corresponding $v_j$ of interest) and again get numerical solutions to \eqref{system-numerical} by applying a construction  similar to \eqref{numerical_small}.

Remark that if $\hat v_i$ and $\hat v_{i+1}$ are both constructed using $\hat P$ and $\hat \Lambda$ (following \eqref{numerical_small} for $\hat v_i$ and the analogous construction for $\hat v_{i+1}$), then as long as  $\hat P$ and $\hat \Lambda$  are strictly increasing, we automatically have $ \hat v_i(x,t) > \hat v_{i+1}(x,t)$ for all $(x,t)$. Remark that the strict monotonicity of $\hat P$ and $\hat \Lambda$ can be checked a posteriori. Note also that smooth exact solutions with strictly increasing initial data will be strictly increasing (in $x$) for all time -- this follows from (the proof) of \Cref{control_lipschitz_lem} (in particular, see \eqref{derivubar}).

On the other hand, if $\hat v_i$ is constructed following \eqref{numerical_small} while $\hat v_{i+1}$ (or $\hat v_{i-1}$) corresponds to a nearly non-decreasing piece of $\bar{u}^0$, then the ordering condition $\hat v_i(x,t) > \hat v_{i+1}(x,t)$ (or $\hat v_{i-1}(x,t) > \hat v_{i}(x,t)$) for all $(x,t)$ can be checked a posteriori. See also \Cref{ordering_remark}.

    In the numerical construction of the $\hat u$ in the computer, we simply take $\hat{v}_i$ to be globally constant with value $v^0(\frac{x_{i-1}+x_{i}}{2})$ (recall, for intervals $(x_{i-1}, x_{i})\in  \mathcal{I}_{rd} $, we have that $\hat u^0|_{(x_{i-1}, x_{i})}$  is constant). However, we continue to use $\hat{v}_i$ as defined in \eqref{numerical_small} for the construction of $\hat\psi$ below in \Cref{constructpsihat}. This is justified by the following consideration.
    
    If the Rankine-Hugoniot ODEs, used to glue the various $\hat{v}_i$ together to make $\hat{u}$, are solved with decent accuracy, then the construction of $\hat{u}$ is oblivious to whether $\hat{v}_i$ is a numerically simulated version of the globally-defined extensions from \Cref{ext_section} or we simply take $\hat{v}_i$ to be globally constant. One way to see this is to remark that in the case when there is no numerical error in the $\hat{v}_i$ (so $\hat v_i=v_i$ for all $i$) and no numerical error in $\hat{u}$, and they are thus exact solutions to \eqref{system}, then any part of the functions $\hat v_i$ not revealed in $\hat{u}$ at $t=0$ will  never be revealed for all $t>0$. This itself can be seen in the following way. At time $t$, consider a shock at position $\hat{h}_i$ with $v_\ell$ as the left-hand state and $v_r$ as the right-hand state. Note that the speed of the characteristic of the function $v_\ell$ going through the point $(\hat{h}_i(t),t)$ is travelling faster than the Rankine-Hugoniot speed of the shock $(v_\ell (\hat h_i(t),t),v_r (\hat h_i(t),t))$, while  the speed of the characteristic of the function $v_r$ going through the point $(\hat h_i(t),t)$ is travelling slower. Thus, artificial extensions are never revealed. Remark here that $v_\ell > v_r$ pointwise.
    
If numerical errors are present in solving the Rankine-Hugoniot ODEs, these errors will be of the order $(\Delta t)^q$ where $q$ is the convergence rate of the numerical ODE solver (see \Cref{constructuhat} for more details). { Note that the (absolute) differences between the velocities of the characteristics through the point $(\hat{h}_i(t),t)$ and the shock speed $\sigma(v_\ell (\hat h_i(t),t),v_r (\hat h_i(t),t))$ are  bounded from above (for all $t\in[0,T]$) by 
\begin{align*}
    \tfrac{\delta \varepsilon \inf A''}{4(1+ \mathfrak{A} MT)},
\end{align*}
which follows from \eqref{ordering_vs} and \Cref{gap_lemma}, with $\mathfrak{A}\coloneqq \sup A''$. Moreover,  we choose $\delta$ to be approximately $h^{1/2}$, so that any  reasonable ODE solver  will preserve the ordering of velocities if $\Delta t$ is chosen small enough -- note that the CFL condition implies $\Delta t \lesssim h$.}
Whether this is the case can be verified in  an a posteriori fashion if a posteriori error control for the ODE solver is implemented. The cost for using a high order Runge-Kutta scheme for solving the Rankine-Hugoniot ODEs is {negligible} compared to the other costs of our numerical scheme for $\hat u$.


We remark that these arguments about when the artificial extensions of the $v_i^0$ are actually revealed will be used again in \Cref{sec:cdp}.

\begin{figure}[tb]
      \includegraphics[width=.6\textwidth]{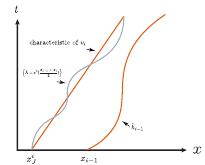}
  \caption{To ensure that our analysis is applicable to a function $\hat u$ that is defined using $\hat{v}_i$ globally constant (as opposed to the global extensions from \Cref{ext_section}), we need to ensure the level set of $\hat\Lambda$ is not too far away from the corresponding characteristic of $v_i$ (and similarly with $\hat P$).}\label{level_set_fig}
\end{figure}

Finally, to ensure that in the construction of the $\hat u$ in the computer, we can simply take $\hat{v}_i$ to be globally constant with value $v^0(\frac{x_{i-1}+x_{i}}{2})$, we need to ensure that at each time $t$ the construction given in \eqref{numerical_small} will never give values for $\hat{v}_i(\cdot,t)$ different from $v^0(\frac{x_{i-1}+x_{i}}{2})$ on an open set containing the interval $[\hat h_{i-1}(t),\hat h_i(t)]$. This follows if the level set (in space-time) $\Big\{\hat\Lambda=v^0(\frac{x_{i-1}+x_{i}}{2})\Big\}$ is close to the characteristic of the function $v_i$ emanating from the point $x_J^i$.
{Note that the position error of the level set can be bounded by the $L^\infty$ error $\norm{\Lambda-\hat\Lambda}_{L^\infty}$ (which can be controlled by \Cref{l_infty_control_lemma}) divided by the minimal slope of $\hat \Lambda$. For a first order finite volume scheme this $L^\infty$ error estimator is expected to be $\mathcal{O}(h^{2/3})$.} 
Remark also that (in the context of \eqref{middle_case}), 
\begin{align}
    x_{i-1}-x_J^i> \tfrac{\varepsilon \delta}{2M} \qquad \text{and} \qquad x_I^i-x_i> \tfrac{\varepsilon \delta}{2M},
\end{align}
and this gap between the characteristic of $v_i$ and the shock $\hat h_{i-1}$ cannot shrink in time (as discussed above, if the ODE solver error is small enough, this gap  is even growing linearly in time). 
Thus, by using a fine enough mesh width, we conclude by considering the numerical residual $\partial_t \hat \Lambda + \partial_x A(\hat \Lambda)$ a posteriori and invoking \Cref{l_infty_control_lemma}. See \Cref{level_set_fig}.

\subsection{Construction of $\hat \psi$}\label{constructpsihat}

The function ${\hat{\psi}}$ is defined similarly. The only difference between ${\hat{\psi}}$ and $\hat{u}$ is that the curves $\hat{h}_i$ (generalized characteristics for $\hat{u}$) are replaced by curves $h_i$ which are generalized characteristics for the exact solution $\bar{u}$. More precisely, until the first time $t_1$  that two of the $h_i$ collide, they are defined as the solution to (in the sense of Filippov)

\begin{align}
    \begin{cases}\label{define_h_exact_gen_char}
    \dot{{h}_i}(t)=\sigma(\bar{u}({h}_i(t)+,t),\bar{u}({h}_i(t)-,t)) \text{ for } t>0,\\
    {h}_i(0)=x_i.
    \end{cases}
\end{align}

Then we define for $0\leq t\leq t_1$,
\begin{align}
{\hat{\psi}}(x,t)\coloneqq
    \begin{cases}
    \hat{v}_1(x,t) &\text{if } x<{h}_1(t),\\
    \hat{v}_2(x,t) &\text{if } {h}_1(t)<x<{h}_2(t),\\
    &\vdots\\
    \hat{v}_{N+1}(x,t) &\text{if } {h}_N(t)<x.\\
    \end{cases}
\end{align}
Remark that in the definition of ${\hat{\psi}}$ we still use the numerical solutions $\hat{v}_i$. Further, $\hat{h}_i(0)=h_i(0)$ for all $i$, and thus $\hat{u}(\cdot,0)=\hat\psi(\cdot,0)$.

Furthermore, remark that instead of defining the various $h_i$ using \eqref{define_h_exact_gen_char}, we could instead define each  $h_i$ as the solution to

\begin{align}\label{define_h_exact}
    \begin{cases}
    \dot{{h}_i}(t)=\sigma({v}_i({h}_i(t),t),{v}_{i+1}({h}_i(t),t)) \text{ for } t>0,\\
    {h}_i(0)=x_i,
    \end{cases}
\end{align}
and use a restarting procedure (as above for $\hat{u}$) when two of the $h_i$ curves collide. This follows immediately from the proof of the main result in \cite{2017arXiv170905610K} (as discussed above in \Cref{sec:uniqueshifts} and \Cref{sec:dssc}).

Both of these ways (\eqref{define_h_exact_gen_char} and \eqref{define_h_exact}) of defining the $h_i$ are equivalent. Note that \eqref{define_h_exact_gen_char} is the definition of a generalized characteristic for the solution $\bar{u}$. The existence of generalized characteristics is well-known when the solution is $BV$. However, when the solution is only known to verify the Strong Trace Property (\Cref{strong_trace_definition}), one possible reference for the existence of generalized characteristics is \cite[Lemma 5.2]{scalar_move_entire_solution}.  

Thus, the function ${\hat{\psi}}$ has discontinuities which follow generalized characteristics for the exact solution $\bar{u}$, but with the functions in between the discontinuities being the functions from the numerical solver, and not the exact solution itself.

\section{Algorithm for controlling shifts }\label{algo_section}
We use different arguments to control shifts of different shocks (and derive estimates on $\Delta_i$). In particular, those ``large'' shocks in $\hat u$ (with size independent of $\delta$) that correspond to discontinuities in $\bar u^0$ will be treated differently from those shocks of size proportional to $\delta$ (at time $t=0$) that arise from discretizing rapidly decreasing parts of $\bar u^0$.

\subsection{Controlling $\Delta_i$ of ``large'' shocks}\label{sec:large}
To compute the $\Delta_i(t)$, the estimate on
\begin{align*}
    \abs{\hat{h}_i(t)-h_i(t)},
\end{align*} alluded to in the Main Theorem (\Cref{main_theorem}), we use the following procedure.

 Let $\hat{L}_i(t)\in\mathbb{N}$ and $\hat{R}_i(t)\in\mathbb{N}$ denote the integers corresponding to the $\hat{v}_k$ which are to the left and right, respectively, of the discontinuity $\hat{h}_i$ in $\hat{u}$ at time $t$.  Thus, $\hat v_{\hat{L}_i(t)}$ is to the left of $\hat h_i$ and $\hat v_{\hat{R}_i(t)}$ is to the right of $\hat h_i$. Likewise, let $L_i(t)\in\mathbb{N}$ and $R_i(t)\in\mathbb{N}$ denote the integers corresponding to the $\hat{v}_k$ which are to the left and right, respectively, of the discontinuity ${h}_i$ in ${\hat{\psi}}$ at time $t$. Thus, $\hat v_{{L}_i(t)}$ is to the left of $h_i$ and $\hat v_{{R}_i(t)}$ is to the right of $h_i$.

An upper bound on $\abs{\hat{h}_i(t)-h_i(t)}$ is then given by,
\begin{equation}\label{def_control_shift}
\Delta_i(t)\leq
A_i + B_i  + C_i,
\end{equation} 
where $A_i, B_i, C_i$ are defined in \eqref{parts_control_shifts}, \eqref{need_supremum_of_this}, and \eqref{need_supremum_of_this_too}:

\begin{equation}\label{parts_control_shifts}
A_i := \int\limits_0^t 
\zeta(t,s)\abs{\dot{\hat{h}}_i(s)-\sigma(\hat{v}_{\hat{L}_i(s)}(\hat{h}_i(s),s),\hat{v}_{\hat{R}_i(s)}(\hat{h}_i(s),s))}\, ds
\end{equation}
\begin{multline}\label{need_supremum_of_this}
B_i := \int\limits_0^t
\zeta(t,s)\abs{\sigma(\hat{v}_{\hat{L}_i(s)}(\hat{h}_i(s),s),\hat{v}_{\hat{R}_i(s)}(\hat{h}_i(s),s)) -\sigma(\hat{v}_{{L}_i(s)}(\hat{h}_i(s),s),\hat{v}_{{R}_i(s)}(\hat{h}_i(s),s))} \, ds
\end{multline}
and
\begin{equation}\label{need_supremum_of_this_too}
C_i := \norm{\frac{
\zeta(t,\cdot)}{\sqrt{\mathfrak{s}_i(\cdot)}} }_{L^2(0,t)}\cdot
\Bigg[ Ce^{C t }\Bigg(\int\limits_{-S}^{S}\eta(\bar{u}(x,0)|\hat{\psi}(x,0))\,dx 
+\mathcal{R}(t)\Bigg)
\Bigg]^{\frac{1}{2}},
\end{equation}
where $\mathcal{R}$ is from \eqref{def:cR}, $\mathfrak{s}_i\colon[0,T]\to\mathbb{R}$ is any lower bound on $\abs{\hat\psi(h_i(t)-,t)-\hat\psi(h_i(t)+,t)}$ and with
\begin{equation}\label{zeta_def}
    \zeta(t,s):= \exp\Bigg(\int_s^t\mathfrak{A}\max_{i\in\{1,\ldots,N+1\}}\text{Lip}[\hat{v}_i(r,\cdot)]\,dr\Bigg).
\end{equation}

The proof of \eqref{def_control_shift} is postponed to \Cref{Revelation_section} (in particular, see \Cref{sec:def_control_shift}). 

Let us briefly comment on the sources of errors represented by $A_i, B_i, C_i$. The term $A_i$ results from not solving the ODE for the position of the discontinuity in $\hat u$, i.e. \eqref{ODEfor_u_hat}, exactly. It can be made small by using a high order Runge-Kutta scheme for solving \eqref{ODEfor_u_hat} independent of other sources of error. The term $B_i$ reflects errors due to uncertainty which solutions $\hat v_i$ furnish the right and left-hand states of the discontinuity in $\hat\psi$ (and the right and left-hand states of the corresponding shock in $\hat u$ may very well be different).
For ``large'' shocks, the term $B_i$ is expected to be zero most of the time but of order one in those (short) periods of time in which there is uncertainty about which collisions have or have not taken place in $\hat \psi$ and thus about which $\hat v_i$ provide left and right states. 
It is expected that under mesh refinement in the numerical scheme the  lengths of time intervals of uncertainty go to zero and as a consequence $B_i$ will go to zero.
The term $C_i$ results from the shocks in $\hat \psi$ not moving with any Rankine-Hugoniot speed according to the $\hat v_i$ but following generalized characteristics of $\bar u$. It is this term where the dissipation estimate Lemma \ref{neg_entropy_diss} is used.

 To use the formula \eqref{def_control_shift}, we need to calculate an upper bound of $B_i$
where the upper bound will depend on  all the possible $L_i(t)$ and $R_i(t)$ (based on possible collisions between the $h_i$ we determine from the algorithm given in this section). This means this quantity, as well as $C_i$, depend on the various $\Delta_j$. 
Thus, the upper bound on $\Delta_i$ depends on $\Delta_i$. To overcome this, we employ the following algorithm.

Recall first that if two characteristics touch, then they are glued together, such that if there is some time $t^*$ such that for some $i$ and $j$, $\hat{h}_i(t^*)=\hat{h}_j(t^*)$ (or ${h}_i(t^*)={h}_j(t^*)$), then for all $t>t^*$, $\hat{h}_i(t)=\hat{h}_j(t)$ (respectively, ${h}_i(t)={h}_j(t)$). See \Cref{construction_section} for more details on this.

Then, control on
the $\Delta_i(t)$ comes from employing the following algorithm:

\subsection{Algorithm for bounding \eqref{need_supremum_of_this} and 
\eqref{need_supremum_of_this_too}} \label{section_frak}
\begin{enumerate}
\item\label{step1} At each time step, for each $i$, use the fact that there is a uniform upper bound on the $\dot{h}_i$ (and $\dot{\hat{h}}_i$) to estimate in a very rough way the largest possible interval where $h_i$ might be. Let us call this interval $I$. 

We work inductively, and based on our knowledge of previous collisions which must have occurred, let $\tilde{L}_i(t)\in\mathbb{N}$ denote an upper bound on the integer $k$ such that $\hat{v}_k$ is on the left of $h_i$ in the function ${\hat{\psi}}$ at time $t$. 

Similarly, let $\tilde{R}_i(t)\in\mathbb{N}$ denote a lower bound on the integer $k$ such that $\hat{v}_k$ is on the right of $h_i$ in the function ${\hat{\psi}}$ at the time $t$. Remark that $\tilde{L}_i(t)\leq i$ and $\tilde{R}_i(t)\geq i+1$.

Then, within the rough interval $I$ where $h_i$ might live, calculate a minimal shock size
\begin{align}\label{inf_I}
\mathfrak{s}_i\coloneqq\inf_{x\in I} [\hat{v}_{\tilde L_i(t)}(x,t)-\hat{v}_{\tilde R_i(t)}(x,t)].
\end{align}

This then gives a preliminary function $\mathfrak{s}_i$ for use in $\Delta_i(t)$. Then, $\Delta_i(t)$ will give an  interval (around  $\hat{h}_i$) of possible positions of $h_i$ which might be narrower  than the roughly calculated  interval $I$. Using this improved interval, the infimum in \eqref{inf_I} can be recalculated over this smaller interval, and thus the $\mathfrak{s}_i$ will increase, which in turn provides for a narrower possible interval for $h_i$. Inductively repeating this process gives a pointwise non-decreasing sequence of $\mathfrak{s}_i$'s bounded below by zero. Thus, the sequence of $\mathfrak{s}_i$'s will have a limit. One can in practice iteratively calculate $\mathfrak{s}_i$ until major increases in $\mathfrak{s}_i$ stop. Remark that if, due to shock interaction, $\hat{v}_{\tilde L_i(t)}$ and $\hat{v}_{\tilde R_i(t)}$ are not to the left and right, respectively, of $h_i$ in the function ${\hat{\psi}}$ at time $t$, then due to the ordering of the $\hat{v}_i$ (see \eqref{ordering_vs}), the gap ${\hat{\psi}}(h_i(t)-,t)-{\hat{\psi}}(h_i(t)+,t)$ might be larger but not smaller than the $\mathfrak{s}_i$ calculated using the above iteration, i.e. we compute an upper bound $C_i$ that is larger than it would need to be.

\item With the $\mathfrak{s}_i$  from step (\ref{step1}),  at each fixed time $t$, use the current location of $\hat{h}_i(t)$ and the formula for $\Delta_i(t)$ to calculate the interval of possible positions of $h_i$.

After each time step, the formula for $\Delta_i(t)$ will give a larger interval of possible positions of $h_i$ (due to
$\mathcal{R}$ growing). Use the error intervals around each $\hat{h}_i$ to determine which curves of discontinuity $h_i$ in ${\hat{\psi}}$ might have touched (and hence, merged). For each $i$, let $\hat{L}_i(t)\in\mathbb{N}$ and $\hat{R}_i(t)\in\mathbb{N}$ denote the integers such that  $\hat{v}_{\hat{L}_i(t)}$ and $\hat{v}_{\hat{R}_i(t)}$  are to the left and right, respectively, of the discontinuity $\hat{h}_i$ in $\hat{u}$ at time $t$.  Likewise, let $L_i(t)\in\mathbb{N}$ and $R_i(t)\in\mathbb{N}$ denote the integers such that $\hat{v}_{L_i(t)}$ and $\hat{v}_{R_i(t)}$ are to the left and right, respectively, of the discontinuity ${h}_i$ in ${\hat{\psi}}$ at time $t$.   
Note that there are situations where due to ambiguity about which shifts $h_i$ have merged the numbers $L_i, R_i$ are not known but usually the set of possible pairs $(L_i, R_i)$ is small (the set of possible pairs can be reduced using the analysis in \Cref{sec:when_collided} below). 

Then, at each time $t$, calculate 
\begin{multline}
    \label{need_supremum_of_this2}
\max_{\{ (L_i,  R_i)\}   }\Big[\sigma\big(\hat{v}_{\hat{L}_i(t)}(\hat{h}_i(t)+,t),\hat{v}_{\hat{R}_i(t)}(\hat{h}_i(t)-,t)\big)
\\
-\sigma\big(\hat{v}_{{L}_i(t)}(\hat{h}_i(t)+,t),\hat{v}_{{R}_i(t)}(\hat{h}_i(t)-,t)\big)\Big],
\end{multline}
where the maximum is taken over the set of possible pairs $L_i(t), R_i(t)$ (based on possible collisions between the $h_i$). Remark that a worst-case bound for the $\abs{\sigma(\ldots)-\sigma(\ldots)}$ term in \eqref{need_supremum_of_this} is simply two times the supersonic speed, $2\sup A'$. 

\end{enumerate}

\subsection{Determining when shocks have collided  in $\hat \psi$}\label{sec:when_collided}

When two curves of discontinuity, corresponding to either ``large'' shocks or ``boundary'' shocks (see \Cref{shocks_fig}) in the numerical solution, $\hat{h}_i$ and $\hat{h}_j$ (for $i<j$) collide (and merge) at a time $t$, we would like to assert whether the corresponding $h_i$ and $h_j$ have collided (and merged) in  the function ${\hat{\psi}}$.

\subsubsection{Collisions involving at least one ``large'' shock}\label{sec:twolargeshocks}
 Assume that at least one of the shocks $\hat{h}_i$ or $\hat{h}_j$ are ``large.'' There are certain scenarios where we can be sure that ${h}_i$ and ${h}_j$ have collided. This can be seen as follows:

We work inductively, and based on our knowledge of previous collisions which must have occurred, let $\tilde{L}_i(t)\in\mathbb{N}$ denote an upper bound on the integer $k$ such that $\hat{v}_k$ is on the left of $h_i$ in the function ${\hat{\psi}}$ at time $t$. The lower our estimate on $\tilde{L}_i(t)$, the better. 

Similarly, let $\tilde{R}_j(t)\in\mathbb{N}$ denote a lower bound on the integer $k$ such that $\hat{v}_k$ is on the right of $h_j$ in the function ${\hat{\psi}}$ at the time $t$. The higher our estimate on $\tilde{R}_j(t)$, the better.

Let  $\hat{R}_i(t-)\in\mathbb{N}$ denote the integer such that the $\hat{v}_{\hat{R}_i(t-)}$ is to the right of the discontinuity $\hat{h}_i$ in $\hat{u}$ just before the collision at time $t$.  Similarly, let  $\hat{L}_j(t-)\in\mathbb{N}$ denote the integer such that the $\hat{v}_{\hat{L}_j(t-)}$ is to the left of the discontinuity $\hat{h}_j$ in $\hat{u}$ just before the collision at time $t$. Remark that $\hat{R}_i(t-)=\hat{L}_j(t-)$. Then, we consider the solutions to the following initial value problem

\begin{align}\label{frak_curves}
\begin{cases}
\mathfrak{h}_{\tilde{L}_i(t),\hat{R}_i(t-)}(t)=\hat{h}_i(t)=\hat{h}_j(t)\\
\dot{\mathfrak{h}}_{\tilde{L}_i(t),\hat{R}_i(t-)}(s)=\sigma(\hat{v}_{\tilde{L}_i(t)}(\mathfrak{h}_{\tilde{L}_i(t),\hat{R}_i(t-)}(s),s),\hat{v}_{\hat{R}_i(t-)}(\mathfrak{h}_{\tilde{L}_i(t),\hat{R}_i(t-)}(s),s) \mbox{ for $s>t$},
\end{cases}
\end{align}
and similarly for $\mathfrak{h}_{\hat{L}_j(t-),\tilde{R}_j(t)}$.

Roughly speaking, we are continuing the curves in $\hat{u}$ as if they did not collide at time $t$.

We argue by contradiction: we assume the two curves $h_i$ and $h_j$ \emph{have not} merged in ${\hat{\psi}}$. 

More precisely, we make the following assumption on the function $\hat\psi$. We assume that  for all $s>t$ the following equalities hold for the $L_j, R_j,L_i$ and $R_i$:
\begin{equation}\label{assume1010}
\aligned
L_j(s)&=\hat{L}_j(t-),\\
R_j(s)&=\tilde{R}_j(t),\\
L_i(s)&=\tilde{L}_i(t),\\
 R_i(s)&=\hat{R}_i(t-).
 \endaligned
\end{equation}

We then use the control on the shifts (see \eqref{def_control_shift} for ``large'' shocks and \Cref{sec:bdyshocks} for ``boundary'' shocks) to estimate  $\abs{\mathfrak{h}_{\tilde{L}_i(t),\hat{R}_i(t-)}-h_i}$ and  $\abs{\mathfrak{h}_{\hat{L}_j(t-),\tilde{R}_j(t)}-h_j}$. Under the assumptions \eqref{assume1010}, we know that the quantity \eqref{need_supremum_of_this} is zero when we apply the estimate \eqref{def_control_shift}.

If the two curves $\mathfrak{h}_{\tilde{L}_i(t),\hat{R}_i(t-)}$ and $\mathfrak{h}_{\hat{L}_j(t-),\tilde{R}_j(t)}$ have a distance large enough apart, i.e.  if there is a time when the two error intervals around $\mathfrak{h}_{\tilde{L}_i(t),\hat{R}_i(t-)}$ and $\mathfrak{h}_{\hat{L}_j(t-),\tilde{R}_j(t)}$   no longer intersect (and $\mathfrak{h}_{\tilde{L}_i(t),\hat{R}_i(t-)}>\mathfrak{h}_{\hat{L}_j(t-),\tilde{R}_j(t)}$),
then we can conclude that $h_i$ and $h_j$ must have merged in the function ${\hat{\psi}}$ before this time. See \Cref{frak_curves_fig}.

\begin{figure}[tb]
      \includegraphics[width=\textwidth]{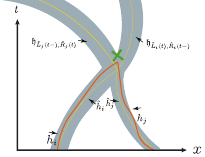}
  \caption{An illustration showing $h_i,h_j,\hat{h}_i,\hat{h}_j,\mathfrak{h}_{\tilde{L}_i(t),\hat{R}_i(t-)}$, and $\mathfrak{h}_{\hat{L}_j(t-),\tilde{R}_j(t)}$. We also show the error interval around $\hat{h}_i$ and $\mathfrak{h}_{\tilde{L}_i(t),\hat{R}_i(t-)}$ which contains $h_i$, as well as the error interval around $\hat{h}_j$ and $\mathfrak{h}_{\hat{L}_j(t-),\tilde{R}_j(t)}$ which contains $h_j$. Note the green X marks where the two error intervals stop overlapping, and thus the first time we can conclude that the $h_i$ and $h_j$ must have merged.}\label{frak_curves_fig}
\end{figure}

If $\tilde{L}_i(t)>L_i(t)$ and/or $\tilde{R}_j(t)<R_j(t)$, this simply means we are underestimating the speed at which the two curves $\mathfrak{h}_{\tilde{L}_i(t),\hat{R}_i(t-)}$ and $\mathfrak{h}_{\hat{L}_j(t-),\tilde{R}_j(t)}$ are pulling apart, and due to this underestimate we determine at a later time that $h_i$ and $h_j$ have merged. This is due to the definition of the $\tilde{L}_i$, $\hat{R}_i$, $\hat{L}_j$, $\tilde{R}_j$, and \eqref{ordering_vs}. Remark also that $\partial_u\sigma(u,v)=\partial_v\sigma(u,v)>0$. 

Similarly, due to running this algorithm on other pairs of shock curves, at some later time $s>t$ we might be able to deduce a collision has occurred between other shock curves in ${\hat{\psi}}$, leading to an improved estimate on either $\tilde{L}_i(s)$ and/or $\tilde{R}_j(s)$. If at this point in time $s$ we have not yet determined if the $h_i$ and $h_j$ have merged, we can use these improved estimate(s) to increase the speed at which the two curves $\mathfrak{h}_{\tilde{L}_i(t),\hat{R}_i(t-)}$ and $\mathfrak{h}_{\hat{L}_j(t-),\tilde{R}_j(t)}$ pull apart. The increase in speed shall start at time $s$.

\subsubsection{Collisions between two ``boundary'' shocks separated  by a nearly non-decreasing piece}\label{nnd_separated}

\begin{figure}[tb]
      \includegraphics[width=.7\textwidth]{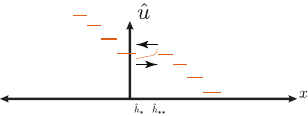}
  \caption{The two ``boundary'' shocks, here denoted by $\hat{h}_{*}$ and $\hat{h}_{**}$, are separated by a Lipschitz-continuous piece (originating from a nearly non-decreasing part of $\bar{u}^0$), and are about to collide in $\hat{u}$.}\label{twobdycollide2_fig}
\end{figure}

Consider two ``boundary'' shocks $\hat{h}_{*}(t)<\hat{h}_{**}(t)$  which are separated by a Lipschitz-continuous piece between them (corresponding to a nearly non-decreasing piece in $\bar{u}^0$ at time $t=0$). Assume that $\hat{h}_{*}$ and $\hat{h}_{**}$ collide in $\hat{u}$ at time $t_1$. Then we want to know when the two corresponding ``boundary'' shocks ${h}_{*}$ and ${h}_{**}$ have also collided in $\hat\psi$ (see \Cref{twobdycollide2_fig}). Let $(x_i,x_{i+1})$ be the nearly non-decreasing interval in $\bar{u}^0$ at time $t=0$ corresponding to the Lipschitz-continuous piece which is between the two ``boundary'' shocks.  The idea is the same as in \Cref{sec:twolargeshocks} above: in the computer, after the time of the collision, continue both of the two ``boundary'' shocks forward in time \emph{as if they did not collide} (so, with the Lipschitz-continuous piece to the left for one of the ``boundary'' shocks, and to the right for the other). Then, under the assumption that the two ``boundary'' shocks also did not collide in $\hat\psi$, use the error intervals from \Cref{sec:bdyshocks}. At the time when the two error intervals have pulled apart and no longer overlap (as in \Cref{frak_curves_fig}), we can conclude by contradiction that the two boundary curves must have collided in $\hat\psi$. Remark that there must be a time $t_0\in[0,t_1]$ such that $\dot{\hat{h}}_{*}(t_0)-\dot{\hat{h}}_{**}(t_0)\geq \frac{x_{i+1}-x_i}{T}$ (otherwise, the two ``boundary'' shocks $\hat{h}_{*}$ and $\hat{h}_{**}$ would never touch). Thus, from the moment the two error intervals around $\hat{h}_{*}$ and $\hat{h}_{**}$ intersect, the length of time until we can determine that ${h}_{*}$ and ${h}_{**}$ have  collided in $\hat\psi$ will be proportional to the sum of the length of the two error intervals.

\subsubsection{Collision between the left and right ``boundary'' shocks of a rapidly decreasing piece}\label{left_right_bdy}

\begin{figure}[tb]
      \includegraphics[width=.7\textwidth]{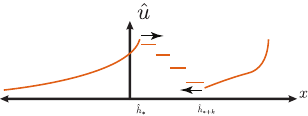}
  \caption{The two ``boundary'' shocks, here denoted by $\hat h_{\ast }$ and $\hat h_{\ast +k}$, are the left and right boundary of a region with ``front tracking'' shocks and they are about to collide in $\hat{u}$.}\label{twobdycollide_fig}
\end{figure}

\hspace{.1in}

The procedure explained above in \Cref{sec:twolargeshocks} for determining whether a collision has to have taken place in $\hat\psi$ works well when at least one ``large'' shock is involved. 

If, however, we consider a possible collision between two curves $h_\ast$ and $h_{\ast+k}$  that are the left and right boundary of a region with ``front tracking'' shocks (corresponding to a rapidly decreasing piece of $\bar{u}^0$, see \Cref{twobdycollide_fig}), then, if we look at the possible speeds  of $h_\ast$ and $h_{\ast+k}$, separately, the lowest possible speed of $h_\ast$ and the largest possible speed of $h_{\ast+k}$ may be very close to each other. However, these largest and smallest values cannot be obtained simultaneously in practice.
If $h_\ast$ and $h_{\ast+k}$ have not collided then $h_\ast$ can only have collided with ``front tracking'' shocks that $h_{\ast+k}$ has not collided with -- indeed the right trace of $h_\ast$ is larger or equal than the left trace of $h_{\ast+k}$. Thus, the minimal Rankine-Hugoniot speed difference is
\begin{align}\label{min_diff}
   m\coloneqq \inf_{(w,v,u)} 
\sigma (w,v) - \sigma(v, u),
\end{align}

where the infimum is over the set 
\begin{align}
\Bigg\{(w,v,u)\in\mathbb{R}^3\colon \abs{w},\abs{v},\abs{u}\leq\max\{\norm{v_\ast}_{L^\infty(K)},\norm{v_{\ast+k+1}}_{L^\infty(K)}\} \text{, } w-u>\rho \text{, } v\in[w,u] \Bigg\},
\end{align}

and where $\rho>0$ is chosen such that $v_{\ast}(x,t)-v_{\ast+k+1}(x,t)>\rho$ on the set $K$. Recall (from \Cref{main_theorem}), $K:= \{(x,t) : -S + st < x < S - st \}$. Note that we have $m= \rho/2$ for Burgers equation

Let $t_{tch}$ be the time when $\hat h_{\ast }$ and $\hat h_{\ast +k}$ touch in the computer. Then, once $t> t_{tch}$ is so large that

\begin{align}\label{must_touch}
  m(t-t_{tch})> \Delta_\ast(t_{tch}) + \Delta_{\ast+k}(t_{tch}) +\Upsilon_\ast(t) + \Upsilon_{\ast + k}(t)
\end{align}
holds (where $m$ is from \eqref{min_diff}), we can be sure that the shocks $h_{\ast }$ and $h_{\ast +k}$ have collided. Note that the $\Upsilon$-terms in \eqref{must_touch} come from the fact that the shocks in $\hat\psi$ are not actually traveling with Rankine-Hugoniot speed but have an error controlled by \eqref{eq:velerr} below.

Remark here that, as above in \Cref{section_frak}, we are assuming that after the $\hat h_{\ast }$ and $\hat h_{\ast +k}$ touch in the computer, the $h_{\ast }$ and $h_{\ast +k}$ have not yet touched in $\hat \psi$. The formulas for $\Delta_\ast(t_{tch})$ and $\Delta_{\ast+k}(t_{tch})$ come from \Cref{sec:bdyshocks} and the assumption that the $h_{\ast }$ and $h_{\ast +k}$ have not yet touched in $\hat \psi$.

Then, we argue by contradiction: when the error intervals around $\hat h_{\ast }$ and $\hat h_{\ast +k}$ have pulled apart (as dictated by \eqref{must_touch} and as shown in \Cref{frak_curves_fig}), we can by contradiction conclude that $h_{\ast }$ and $h_{\ast +k}$ must have touched. Equation \eqref{must_touch} gives us an upper bound on the time when the error intervals must pull apart.  After we can be sure that $h_{\ast }$ and $h_{\ast +k}$ must have touched, we can treat the combined shock as a ``large'' shock using the techniques from \Cref{sec:large} and \Cref{sec:twolargeshocks}, but we start the Gronwall calculation (see \eqref{shift_gronwall_1027}) at the time when we are sure $h_{\ast }$ and $h_{\ast +k}$  have combined, and replace the $\abs{h_i(0)-\hat{h}_i(0)}$ term in \eqref{shift_gronwall_1027} with the position uncertainty at the time when we are sure $h_{\ast }$ and $h_{\ast +k}$ have combined.

\subsection{Improved control on ``front tracking'' shocks} \label{front_tracking_shocks}
{Let us begin by giving some scaling arguments, that suppose that the $L^2$ norm of the initial data approximation error is $\mathcal{O}(\sqrt{\delta ^2 + h^2})$ but not $o(\sqrt{\delta ^2 + h^2})$ and the $L^2$ norm of the residuals (see \eqref{system-numerical}) is $\mathcal{O}(h)$ but not $o(h)$. This is the case in our numerical experiments for first order schemes.
The estimates on uncertainty of shock positions made above (see \eqref{def_control_shift}) at best provide $\mathcal{O}(\frac{1}{\sqrt{\delta}}\sqrt{ \delta^2 + h^2}) \geq \mathcal{O}(\sqrt{\delta})$ upper bounds on position errors of shocks of size proportional to $ \delta$ (originating from  areas where $\bar u^0$ is rapidly decreasing).} Since these shocks are initially only $\delta$ apart there is essentially no control of the order in which their collisions happen and of the consequences of these collisions. Thus, a different methodology is needed in order to control the difference $\hat u - \hat \psi$ in these areas. This will be based on a combination of Bressan's results on wave front tracking solutions \cite[Eqn (8.22)]{Bressan2000} and an improved estimate for velocity errors.

\subsubsection{Improved velocity control}\label{sec:improvevelcontr}

To obtain better control on the velocity of any one of the small ``front tracking'' shocks we fix this shock, denote its position by $h_\star$,  and create another version of ${\hat{\psi}}$, called ${\hat{\psi}}_{\text{fine}}$. At time $t=0$, ${\hat{\psi}}_{\text{fine}}$ is also in the form of \eqref{LIWASe}, and ${\hat{\psi}}_{\text{fine}}(\cdot,0)$ is related to ${\hat{\psi}}(\cdot,0)$ in the following way: ${\hat{\psi}}_{\text{fine}}$ chooses a much finer piecewise-constant approximation (made up of extremely tiny shocks) in areas where $\bar{u}(\cdot,0)$ is rapidly decreasing. This is done in all rapidly decreasing parts \emph{except at the shock} $\hat{h}_\star$ (the size of this shock is the same in both ${\hat{\psi}}$ and ${\hat{\psi}}_{\text{fine}}$ -- let us call this shock the $\star$-th shock in both ${\hat{\psi}}$ and ${\hat{\psi}}_{\text{fine}}$). See \Cref{psi_fine_fig}.  On the domains of the nearly non-decreasing parts of $\bar u(\cdot,0)$, i.e., those with $\partial_x \bar{u}(\cdot,0) > - \epsilon$ we let $ {\hat{\psi}}_{\text{fine}}(\cdot,0)$ and 
$ {\hat{\psi}}(\cdot,0)$ agree.

For both ${\hat{\psi}}(\cdot,0)$ and ${\hat{\psi}}_{\text{fine}}(\cdot,0)$, on each interval where these functions are Lipschitz-continuous, we need to  extend the Lipschitz-continuous pieces to be functions on the whole real line. We extend each of these pieces following \Cref{ext_section}, giving extensions $v_i^0$. We evolve the $v_i^0$ corresponding to the nearly non-decreasing parts of $\bar{u}^0$ numerically (giving numerical solutions $\hat v_i$) 
 whereas we evolve the extended initial data $v_i$ corresponding to rapidly decreasing pieces following \Cref{do_not_sim_remark}.

By doing this, we get that
\begin{multline}\label{imrpoved633}
    \int_{\mathbb{R}} \eta(\bar{u}(\cdot,0)|\hat{\psi}_{\text{fine}}(\cdot,0)) \, dx 
    \\
    \leq 2\delta \mathfrak{H}\Big(\sum_{*\in{\{k\colon h_k(0)\text{ ``boundary'' shock, } \bar{u}^0(h_k(0)+)=\bar{u}^0(h_k(0)-)\}}}\big[\mathfrak{s}_*(0)^2 \big]+\mathfrak{s}_\star(0)^2\Big)
    \\+ \sum_{(x_i, x_{i+1}) \in \mathcal{I}_{nnd}}\int_{x_i}^{x_{i+1}} \eta(v_i^0(x)|\hat v_i(x,0)\,dx,
\end{multline}
 where $\mathfrak{H}:=\sup \eta''$, $\mathfrak{s}_\star(0)$ is the height of the shock at $\hat{h}_\star$ (at time $t=0$) and $\mathcal{I}_{nnd} $ is the set of nearly non-decreasing intervals. For the function $\hat\psi$, we have the ordering \eqref{ordering_main_thm}. In order to make sure the analogous ordering holds also with the function ${\hat{\psi}}_{\text{fine}}$, we may also need to keep the sizes of the ``front tracking'' shocks adjacent to ``boundary'' shocks in $\hat\psi(\cdot,0)$ the same size in ${\hat{\psi}}_{\text{fine}}(\cdot,0)$ as they are in $\hat\psi(\cdot,0)$ (see \Cref{psi_fine_fig}). This is where the term $\sum_{*\in{\{k\colon h_k(0)\text{ ``boundary'' shock, } \bar{u}^0(h_k(0)+)=\bar{u}^0(h_k(0)-)\}}}\big[\mathfrak{s}_*(0)^2 \big]$ in \eqref{imrpoved633} comes from.

\begin{figure}[tb]
      \includegraphics[width=\textwidth]{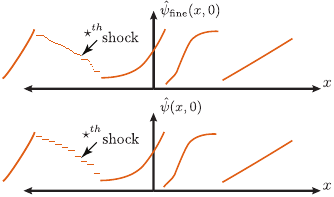}
  \caption{For each ``front tracking'' shock in ${\hat{\psi}}$ (and correspondingly in $\hat{u}$), we introduce a function ${\hat{\psi}}_{\text{fine}}$ which helps us study that particular shock.}\label{psi_fine_fig}
\end{figure}

The function $\mathfrak{s}_\star(t)>0$, which verifies the bound  $\hat{\psi}(h_\star(t)-,t)-\hat{\psi}(h_\star(t)+,t)\geq \mathfrak{s}_\star(t)$, can be chosen proportional to  $\delta$  for all $t\in[0,T]$ (in the case of exact solutions, this follows from \Cref{gap_lemma}, and we can check it a posteriori for the numerical solutions $\hat v_i$).

Then, applying the dissipation estimate \eqref{dissipation_formula_psi4} to the $\star$-th shock in  ${\hat{\psi}}_{\text{fine}}$ implies an (average) velocity error, i.e. the difference of the speed of the shock from its Rankine-Hugoniot speed, measured in the $L^2$-in-time norm, of 

\begin{multline}\label{dissipation_repeated}
\int\limits_{0}^{t} \Bigg[{c{\mathfrak{s}_\star}(r)}\big(\sigma(\bar{u}(h_\star(r)+,r),\bar{u}(h_\star(r)-,r))-\sigma({\hat{\psi}}(h_\star(r)+,r),{\hat{\psi}}(h_\star(r)-,r))\big)^2\Bigg]\,dr
\\
\leq
C\Bigg(\int\limits_{-S}^{S}\eta(\bar{u}(x,0)|\hat{\psi}_{fine}(x,0))\,dx +  \int\limits_{0}^{t}\int\limits_{-S+sr}^{S-sr}(R_{\hat{\psi}_{fine}}(x,r))^2\,dxdr\Bigg)e^{C t},
\end{multline}
where $c>0$ is defined in \eqref{def_c}.  The function $R_{\hat{\psi}_{fine}}$ is defined analogously to $R_{{\hat{\psi}}}$ (see \eqref{residual_psi}). The relative entropy  $\int \eta(\bar{u}(x,0)|\hat{\psi}(x,0))\,dx$ is expected to be $\mathcal{O}( \delta^3+h^2)$ but not $o( \delta^3+h^2)$ and the squared $L^2$ norm of the residual is expected to be $\mathcal{O}(h^\frac{3}{2})$ caused by the kinks in the extensions. Thus the optimal scaling is $\delta \sim h^{1/2}$ and the right-hand side of \eqref{dissipation_repeated} is $\mathcal{O}(\delta^3)= \mathcal{O}( h^{3/2})$.







The point is that because the shift functions for both ${\hat{\psi}}$ and ${\hat{\psi}}_{\text{fine}}$ are simply generalized characteristics of $\bar{u}$,  the position of the $\star$-th shock will be  \emph{the exact same} in both ${\hat{\psi}}$ and ${\hat{\psi}}_{\text{fine}}$ for all time.

\begin{figure}[tb]
      \includegraphics[width=\textwidth]{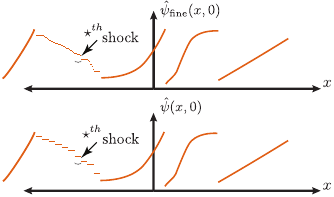}
  \caption{Multiple of the tiny shocks in ${\hat{\psi}}_{\text{fine}}$ (moving with velocity dictated by shift functions), such as indicated here by the bracket to the left of the $\star$-th shock, might interact before any interaction has occurred in ${\hat{\psi}}$. However, if enough of the tiny shocks interact with the $\star$-th shock in ${\hat{\psi}}_{\text{fine}}$, then necessarily one of the shocks to the left of the $\star$-th shock in ${\hat{\psi}}$ (indicated again by a bracket) must interact with the $\star$-th shock, bringing the Rankine-Hugoniot speeds of the $\star$-th shocks in ${\hat{\psi}}_{\text{fine}}$ and ${\hat{\psi}}$ close to parity.}\label{psi_fine_fig2}
\end{figure}

However, due to the  many tiny shocks in ${\hat{\psi}}_{\text{fine}}$ that are not in ${\hat{\psi}}$, some of these tiny shocks will hit and combine with the $\star$-th shock we are concerned with, changing its left or right state and thus its natural Rankine-Hugoniot speed. This means that the Rankine-Hugoniot speed of this shock in ${\hat{\psi}}_{\text{fine}}$ might be slightly different than the Rankine-Hugoniot speed of the same shock in ${\hat{\psi}}$. This will be problematic in \Cref{sec:cdp} since in Bressan's theory we need to control the error of the velocity compared to the natural Rankine-Hugoniot speed. %
However, this difference in Rankine-Hugoniot speeds between ${\hat{\psi}}_{\text{fine}}$ and ${\hat{\psi}}$ is moderated by the fact that 
the shocks in ${\hat{\psi}}$ and ${\hat{\psi}}_{\text{fine}}$ are both travelling with velocities given by  generalized characteristics of $\bar u$, i.e. only the tiny shocks in ${\hat{\psi}}_{\text{fine}}$ that start between the $\star-1$-th and $\star$-th shocks in ${\hat{\psi}}$ can combine with the $\star$-th shock (in ${\hat{\psi}}_{\text{fine}}$)  before the $\star-1$-th shock in ${\hat{\psi}}$ combines with it as well. See \Cref{psi_fine_fig2}. The same holds for shocks to the right and future collisions. 

All in all, we get that, the velocity error of the $\star$-th shock in $\hat \psi$ (in the $L^1(0,t)$ norm)
is bounded by,

\begin{multline}\label{eq:velerr}
    \int_0^t | \dot h_\star(s) - \sigma(v_{L_\star(s)}(h_\star(s),s), v_{R_\star(s)}(h_\star(s),s)) | \, ds\\ \leq 
 \frac{\sqrt{t}}{\min_t \sqrt{\mathfrak{s}_\star (t)} }\Bigg(2\mathfrak{H}\delta\Big(\mathfrak{s}_\star(0)^2 +\sum_{*\in{\{k\colon h_k(0)\text{ ``boundary'' shock, } \bar{u}^0(h_k(0)+)=\bar{u}^0(h_k(0)-)\}}}\big[\mathfrak{s}_*(0)^2 \big]\Big) 
 \\
 + \sum_{(x_i, x_{i+1}) \in \mathcal{I}_{nnd}}\int_{x_i}^{x_{i+1}} \eta(v_i^0(x)|\hat v_i(x,0))\,dx   +
  \int\limits_{0}^{t}\int\limits_{-S+sr}^{S-sr}(R_{\hat{\psi}_{fine}}(x,r))^2\,dxdr\Bigg)^{\frac12} e^{C t}\\+ t\bar{ \mathfrak{s}}\cdot\mathfrak{A}  =: \Upsilon_\star(t),\\
    \end{multline}
    where $\mathfrak{H}$ is an upper bound for $\eta''$, we recall that $\mathfrak{A}$ is an upper bound for the Lipschitz constant of Rankine-Hugoniot speeds (see \eqref{eq:RH}), and
    \begin{align}\label{def_frakbars}
\bar{ \mathfrak{s}}:= \sup_t\max_j \sup_{x\in\mathbb{R}}[\hat{v}_j(x,t)-\hat{v}_{j+1}(x,t)],
\end{align}
 where the maximum over $j$ is taken  over \emph{all} ``front tracking'' shocks $h_j$ in $\hat\psi(\cdot,0)$ (corresponding to all of the rapidly decreasing parts of $\bar u^0$).

 In the situation that we observe, i.e. with kinks in the extensions, the $L^2$norms of $R_{\hat{\psi}_{fine}}$ (and $R_{\hat{\psi}}$)  are  $\mathcal{O}(h^{3/4})$ but not $o(h^{3/4})$ (and $\delta$ is chosen as $h^{\frac{1}{2}}$) then \eqref{eq:velerr} provides an estimate for the uncertainty of shock position of order $\mathcal{O}(\delta)$.

Compare this with the discussion at the beginning of this section (\Cref{front_tracking_shocks}) on uncertainty of shock positions via using \eqref{def_control_shift} directly -- this only provides uncertainty of $\mathcal{O}(\sqrt{\delta})$.

\subsubsection{Control on the regions of $\hat u$ coming from rapidly decreasing parts}\label{sec:cdp}

\begin{figure}[tb]
      \includegraphics[width=.8\textwidth]{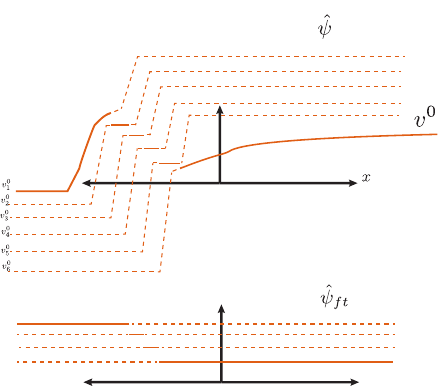}
  \caption{We define the function $\hat \psi_{ft}$, which, in contrast with $\hat\psi$, uses constant extensions. Here the extensions are shown as dotted lines.}\label{ft_fig}
\end{figure}

Let us focus on one interval 
\begin{align}\label{focus_interval}
[h_k(0) , h_{k+\ell}(0)]
\end{align}
where $\bar u^0$ is rapidly decreasing. We wish to use estimates from the front tracking technique but $\hat \psi$ is not a front tracking solution (due to its increasing parts). Thus, we define a function $\hat \psi_{ft}$ that can be understood as an approximate front tracking solution. 
We construct the piecewise-constant (in space) function $\hat \psi_{ft}$ from the same piece-wise constant approximation of $\bar u^0|_{[h_k(0) , h_{k+\ell}(0)]}$  as $\hat \psi$ but using globally constant extensions for each piece (instead of the extensions in \Cref{ext_section}). Furthermore, at time $t=0$, we define $\hat \psi_{ft}$ to be constant outside of $[h_k(0) , h_{k+\ell}(0)]$ with $\hat \psi_{ft}(x,0)=\hat \psi_{ft}(h_k(0)+,0)$ for $x<h_k(0)$ and $\hat \psi_{ft}(x,0)=\hat \psi_{ft}(h_{k+\ell}(0)-,0)$ for $x>h_{k+\ell}(0)$. See \Cref{ft_fig}.

Thus, for all time $t$, $\hat \psi_{ft}(\cdot,t)$ is piecewise-constant on $\mathbb{R}$. For each discontinuity in $\hat \psi_{ft}$, the space-time 
 curves of discontinuity in $\hat \psi_{ft}$ are  the same  as for the corresponding shock in $\hat \psi$ (the movement of the shocks is dictated by the same function).

 For each interval \eqref{focus_interval}, we will also compute a numerical solution $\hat u_{ft}$ with initial data $\hat \psi_{ft}(\cdot,0)$ and with shocks moving with exact Rankine-Hugoniot speed. The curves of discontinuity in $\hat u_{ft}$ will be indexed as $\hat{h}_{ft,i}$. At time $t=0$, $\hat{h}_{ft,i}(0)=\hat{h}_i(0)$ for all $i$ such that $\hat h_i(0)\in[h_{k}(0),h_{k+\ell}(0)]$. Recall that at $t=0$, $\hat h_i=h_i$ for all $i$.

 Similarly, for each $p,q\in\mathbb{N}$, $k\leq p<q\leq k+\ell$, we construct the piecewise-constant (in space) function $\hat \psi^{p,q}_{ft}$ from the same piece-wise constant approximation of $\bar u^0|_{[h_k(0) , h_{k+\ell}(0)]}$  as $\hat \psi$ but only on the interval $[h_p(0) , h_{q}(0)]$. At time $t=0$, we define $\hat \psi^{p,q}_{ft}$ to be constant outside of $[h_p(0) , h_{q}(0)]$ with $\hat \psi^{p,q}_{ft}(x,0)=\hat \psi_{ft}(h_p(0)-,0)$ for $x<h_p(0)$ and $\hat \psi^{p,q}_{ft}(x,0)=\hat \psi_{ft}(h_{q}(0)+,0)$ for $x>h_{q}(0)$.
 
 As with $\hat \psi_{ft}$, we use globally constant extensions for each constant piece of $\hat \psi^{p,q}_{ft}$ (instead of the extensions in \Cref{ext_section}). For each discontinuity in $\hat \psi^{p,q}_{ft}$, the shock 
 curves in $\hat \psi^{p,q}_{ft}$ are the same as for the corresponding shock in $\hat \psi$ (the movement of the shocks is dictated by the same function). In analogy with $\hat u_{ft}$, we also define in a similar way $\hat u^{p,q}_{ft}$. 

 For our numerical scheme to construct the $\hat u$, we will numerically simulate the various $\hat u_{ft}$ associated with  $\hat u$.

 We would like to infer bounds on velocity errors in $\hat \psi_{ft}$ from \eqref{eq:velerr} but we need to account for the fact that the natural Rankine-Hugoniot speeds in $\hat \psi$ and $\hat \psi_{ft}$ are different due to potentially different left- and right-hand states at each shock (due to nonconstant extensions in $\hat\psi$ -- see \Cref{ext_section}).
 To understand how this can be done let us focus on one discontinuity $h_i$ in $\hat{\psi}$, with adjacent solutions $v_\ell$, $v_r$ and hence left-hand state (at time $t$) $v_\ell (h_i(t),t)$ and right-hand state $v_r (h_i(t),t)$. Remark that if for some time $t_0>0$, for all $t>t_0$ this discontinuity happens to be traveling with the natural Rankine-Hugoniot speed, then any part of the functions $v_\ell$, $v_r$  not revealed in $\hat{\psi}$ at $t=t_0$ will then never be revealed for all $t>t_0$. This can be seen in two ways. One way is to note that the speed of the characteristic of the function $v_\ell$ going through the point $(h_i(t),t)$ is travelling faster than the Rankine-Hugoniot speed of the shock $(v_\ell (h_i(t),t),v_r (h_i(t),t))$, while  the speed of the characteristic of the function $v_r$ going through the point $(h_i(t),t)$ is travelling slower. Remark here that $v_\ell > v_r$ pointwise.
 
 A second way to see this is the following. Consider the case when there is no numerical error, i.e. $\mathcal{R}_j\equiv0$ for all $j$. Furthermore, assume that there exists some $t_0>0$ such that for all $t>t_0$, all the discontinuities in $\hat\psi$ suddenly begin to travel with natural Rankine-Hugoniot speed. Then, for $t>t_0$ the function $\hat\psi$ is an exact solution to \eqref{system}. In particular, it verifies Ole\u{\i}nik's \emph{condition E} (see \eqref{condE}) -- which is a uniqueness criterion. Thus, the artificial extensions of the $v_\ell$, $v_r$ can never be revealed after time $t_0$. If they could be revealed, this would violate uniqueness because at time $t_0$ the hidden parts of the extensions can be arbitrarily perturbed (as long as smoothness is preserved).

Thus, the artificial extensions of the $v_\ell$, $v_r$ will only be revealed when $h_i$ is traveling not at Rankine-Hugoniot speed. We are interested in the height difference between the constant values of the $v_\ell$ and $v_r$ which are exposed at $t=0$ and the  values of the extensions which are revealed at later time. Thus, from \eqref{eq:velerr} we see that there {is} a maximum height difference (relative to the constant state at time $t=0$) of


\begin{align}
{\hat{M}}\Upsilon_i(t)\label{height_control_diff}
\end{align}
where 
\begin{align}\label{slopeMhat}
    \hat{M}\coloneqq \sup_t \mathrm{Lip}[\hat{\Lambda}(\cdot,t)],
\end{align}
cf. \eqref{slopeM} and
where we have used the formulas for the extensions (see \Cref{do_not_sim_remark}). 

From \eqref{height_control_diff}, we get an estimate on difference in Rankine-Hugoniot speeds (and thus in velocity errors) that result from the fact that $\hat \psi$ has non-constant extensions, while $\hat \psi_{ft}$ uses constant extensions.

Similarly, we can use \eqref{height_control_diff} in order to bound the difference between $\hat \psi$ and $\hat \psi_{ft}$. In particular,

\begin{equation}\label{eq:psipsift}
    \| \hat \psi(\cdot,t) - \hat \psi_{ft}(\cdot,t) \|_{L^\infty}
    \leq {\hat{M}} 
    \max_\alpha \Upsilon_\alpha(t) \coloneqq \Upsilon,
\end{equation}
where the $\max$ is taken over the set of all shocks in $\hat\psi_{ft}$.

Fix a time $t>0$. Inside the region $\{ (x,s) : x \in [h_k(s) , h_{k+\ell}(s)]\}$ we can interpret $\hat \psi_{ft}$ as an approximate front tracking solution whose shocks move with a velocity error (relative to the natural Rankine-Hugoniot speed of the shocks in $\hat \psi_{ft}$). The velocity error is controlled by \eqref{eq:velerr} and \eqref{height_control_diff}. Then, let $\tilde p \in\mathbb{N}$ be the lowest integer such that $h_{\tilde p}(t)\in (h_k(t) , h_{k+\ell}(t))$ (so in particular  $h_{\tilde p}$ has not collided with $h_{k}$) and similarly let $\tilde q\in\mathbb{N}$ be the largest integer such that $h_{\tilde q}(t)\in (h_k(t) , h_{k+\ell}(t))$ (so in particular, $h_{\tilde q}$ has not collided with $h_{k+\ell}$). Then we receive from \Cref{L1_diss_lemma},

\begin{claim}


 \begin{multline}\label{eq:Bressan}
     \| \hat \psi_{ft}(\cdot,t) - \hat u^{{\tilde p},{\tilde q}}_{ft}(\cdot,t)\|_{L^1([ h_k(t), h_{k+\ell}(t)])}   \\
     \leq    
\Big( (\max \hat \psi_{ft} - \min \hat \psi_{ft}) +\frac{h_{k+\ell}(0)- h_k(0)}{\delta} {\hat{M}}
    \max_\alpha \Upsilon_\alpha(t) \Big)^{\frac12}\times\\
    \sqrt{t}
 \Bigg[C\Bigg(\int\limits_{-S}^{S}\eta(\bar{u}(x,0)|\hat{\psi}(x,0))\,dx +  \int\limits_{0}^{t}\int\limits_{-S+sr}^{S-sr}(R_{\hat{\psi}}(x,r))^2\,dxdr\Bigg)e^{C t}.\Bigg]^\frac12\\
 + 2\frac{h_{k+\ell}(0)- h_k(0)}{\delta} {\hat{M}} 
 (\max_\alpha \Upsilon_\alpha (t))^2\\
 +
 \Big( (\max \hat \psi_{ft} - \min \hat \psi_{ft}) +\frac{h_{k+\ell}(0)- h_k(0)}{\delta} {\hat{M}}
    \max_\alpha \Upsilon_\alpha(t) \Big)
    2 \mathfrak{A}
 {\hat{M}} t \max_\alpha \Upsilon_\alpha(t) 
 \\
 =: \Gamma(t),
\end{multline}
where the $\max$ (in $\max_\alpha \Upsilon_\alpha$) is over the set of all shocks in $\hat \psi^{{\tilde p},{\tilde q}}_{ft}$. See \eqref{eq:velerr} for the definition of $\Upsilon_\alpha$.

\end{claim}
In particular, \eqref{eq:Bressan} comes from comparing $\hat \psi^{{\tilde p},{\tilde q}}_{ft}$ to $\hat u^{{\tilde p},{\tilde q}}_{ft}$ using \Cref{L1_diss_lemma}. Remark that on the interval $[h_{\tilde p}(t),h_{\tilde q}(t)]$ we have $\hat \psi^{{\tilde p},{\tilde q}}_{ft}=\hat \psi_{ft}$.

We now prove the claim.

\begin{claimproof}
    From \Cref{L1_diss_lemma}, we have
    \begin{align}
    \begin{aligned}\label{step1claim}
    \| \hat \psi_{ft}(\cdot,t) - \hat u^{{\tilde p},{\tilde q}}_{ft}(\cdot,t)\|_
    {L^1([ h_k(t), h_{k+\ell}(t)])}  \hspace{3in}
    \\
 \leq   \int\limits_0^t \sum_{\alpha}\Bigg[\abs{\hat\psi_{ft}({h}_\alpha(s)-,s)- \hat\psi_{ft}({h}_\alpha(s)+,s)}\abs{\dot{{h}}_\alpha-\sigma(\psi_{ft}({h}_\alpha(s)-,s),\hat\psi_{ft}({h}_\alpha(s)+,s))}\Bigg]\, ds,
 \end{aligned}
 \end{align}
 where at each time $s$, the sum is over all distinct curves of discontinuity $h_\alpha$ of the function $\hat \psi^{{\tilde p},{\tilde q}}_{ft}$. Then, from \eqref{height_control_diff}, we get

\begin{multline}
\leq    \int\limits_0^t \sum_{\alpha}\Bigg[\abs{\hat\psi({h}_\alpha(s)-,s)-  \hat\psi({h}_\alpha(s)+,s)}\abs{\dot{{h}}_\alpha-\sigma(\hat\psi({h}_\alpha(s)-,s),\hat\psi({h}_\alpha(s)+,s))}\Bigg]\, ds\\
+ \int_0^t 2{\hat{M}} 
    \max_\alpha \Upsilon_\alpha(s) \sum_\alpha \abs{\dot{{h}}_\alpha-\sigma(\hat\psi({h}_\alpha(s)-,s),\hat\psi({h}_\alpha(s)+,s))}\, ds\\
  + 
  \Big( (\max \hat \psi_{ft} - \min \hat \psi_{ft}) + 
  \frac{h_{k+\ell}(0)- h_k(0)}{\delta} {\hat{M}}
    \max_\alpha \Upsilon_\alpha(t) \Big)
    2 \mathfrak{A}
  t \| \hat \psi - \hat \psi_{ft}\|_{L^\infty}
\end{multline}
where we have used 
that
$\sum_{\alpha}\abs{\hat\psi_{ft}({h}_\alpha(s)-,s)-  \hat\psi_{ft}({h}_\alpha(s)+,s)}\leq \max  \psi_{ft} -  \min  \psi_{ft}$. Also recall, at $t=0$ the rapidly decreasing pieces of $\bar{u}^0$ are broken up into steps of width $\delta$ to make $\hat \psi(\cdot,0)$, and thus there are at most $\tfrac{h_{k+\ell}(0)- h_k(0)}{\delta}$ of these steps. 
Then, using Cauchy-Schwartz' inequality, we get the  estimate,

\begin{multline}
    \leq    
 \Bigg[\int\limits_0^t \sum_{\alpha}\abs{\hat\psi({h}_\alpha(s)-,s)-  \hat\psi({h}_\alpha(s)+,s)}
 ds \Bigg]^\frac12 \times 
 \\
 \Bigg[\int\limits_0^t \sum_{\alpha}\abs{\hat\psi({h}_\alpha(s)-,s)-  \hat\psi({h}_\alpha(s)+,s)}\abs{\dot{{h}}_\alpha-\sigma(\hat\psi({h}_\alpha(s)-,s),\hat\psi({h}_\alpha(s)+,s))}^2\, ds\Bigg]^\frac12\\
 + 2\frac{h_{k+\ell}(0)- h_k(0)}{\delta} {\hat{M}} 
 (\max_\alpha \Upsilon_\alpha (t))^2\\
 +
 \Big( (\max \hat \psi_{ft} - \min \hat \psi_{ft}) +\frac{h_{k+\ell}(0)- h_k(0)}{\delta} {\hat{M}}
    \max_\alpha \Upsilon_\alpha(t) \Big)
    2 \mathfrak{A}
 {\hat{M}} t \max_\alpha \Upsilon_\alpha(t)
\end{multline}
 
 Then, note again that $\sum_{\alpha}\abs{\hat\psi_{ft}({h}_\alpha(s)-,s)-  \hat\psi_{ft}({h}_\alpha(s)+,s)}\leq \max  \psi_{ft} -  \min  \psi_{ft}$. Also recall, at $t=0$ the rapidly decreasing pieces of $\bar{u}^0$ are broken up into steps of width $\delta$ to make $\hat \psi(\cdot,0)$, and thus there are at most $\tfrac{h_{k+\ell}(0)- h_k(0)}{\delta}$ of these steps. Recall also \eqref{height_control_diff}. Thus,
 from the dissipation estimate \eqref{dissipation_formula_psi3},

\begin{multline}\label{eq:theend?}
    \leq    
\Big( (\max \hat \psi_{ft} - \min \hat \psi_{ft}) +\frac{h_{k+\ell}(0)- h_k(0)}{\delta} {\hat{M}}
    \max_\alpha \Upsilon_\alpha(t) \Big)^{\frac12}\times\\
    \sqrt{t}
 \Bigg[C\Bigg(\int\limits_{-S}^{S}\eta(\bar{u}(x,0)|\hat{\psi}(x,0))\,dx +  \int\limits_{0}^{t}\int\limits_{-S+sr}^{S-sr}(R_{\hat{\psi}}(x,r))^2\,dxdr\Bigg)e^{C t}.\Bigg]^\frac12\\
 + 2\frac{h_{k+\ell}(0)- h_k(0)}{\delta} {\hat{M}} 
 (\max_\beta \Upsilon_\beta (t))^2\\
 +
 \Big( (\max \hat \psi_{ft} - \min \hat \psi_{ft}) +\frac{h_{k+\ell}(0)- h_k(0)}{\delta} {\hat{M}}
    \max_\alpha \Upsilon_\alpha(t) \Big)
    2 \mathfrak{A}
 {\hat{M}} t \max_\alpha \Upsilon_\alpha(t)
\end{multline}

\end{claimproof}

We can bound the $L^1(h_k(t) , h_{k+\ell}(t))$ distance between $\hat \psi(\cdot,t)$ and $\hat u(\cdot,t)$ by combining \eqref{eq:Bressan} and \eqref{eq:psipsift}.

\subsubsection{From $\hat\psi_{ft}$, we construct $\hat\Psi_{ft}$}

For the next \Cref{sec:triangles} we want estimates on $\norm{\hat u_{ft}-\hat \psi_{ft}}_{L^1(\mathbb{R})}$ (globally on $\mathbb R$!).
However, \eqref{eq:Bressan} only gives us error estimates on the interval $[h_k(t),h_{k+\ell}(t)]  $ and 
for the function $\hat u^{{\tilde p},{\tilde q}}_{ft}$ instead of $\hat u_{ft}$.
To circumvent this issue we construct an intimately connected function $\hat \Psi_{ft}$ on which we have the desired global $L^1$ estimates relative to $\hat u_{ft}$

Indeed, we cannot measure $\norm{\hat u_{ft}-\hat \psi_{ft}}_{L^1(\mathbb{R})}$  using the $L^1$ estimates from front tracking (i.e., \Cref{L1_diss_lemma}) because the shocks at position $h_k(t)$ or $h_{k+\ell}(t)$ in $\hat\psi_{ft}$ will not necessarily be moving with Rankine-Hugoniot speed plus a controllable error. This is because the shocks in $\hat\psi_{ft}$ might start being influenced by pieces of $\hat\psi$ coming from nearly non-decreasing pieces of $\bar{u}^0$ -- the ``boundary'' shocks in particular (see \Cref{shocks_fig}). 

Thus, we are led to the construction of the function $\hat\Psi_{ft}\colon\mathbb{R}\times[0,T]\to\mathbb{R}$. The construction of $\hat\Psi_{ft}$ follows precisely the construction of $\hat\psi_{ft}$, except we will make modifications outside of the interval $[ h_k(t), h_{k+\ell}(t)]$, such that

\begin{align}\label{capvslowercase}
\norm{\hat u_{ft}(\cdot,t)-\hat \Psi_{ft}(\cdot,t)}_{L^1(\mathbb{R})}
\leq \Gamma(t),
\end{align} 
where $\Gamma$ was defined in \eqref{eq:Bressan}, holds for all $t$.

Note \eqref{capvslowercase} is possible by simply slowing down (or speeding up) the shocks in $\hat\psi_{ft}$ when they start moving too fast (or too slow) due to the corresponding ``front tracking'' shock in $\hat\psi$ colliding with a ``boundary'' shock in $\hat\psi$ coming from the left (or the right). This slowing or speeding up process yields the $\hat\Psi_{ft}$. Remark that by slowing down shocks on the left of the function $\hat\psi_{ft}$, which have collided with $h_k$, we are preventing some collisions from occurring, further slowing down some shocks in $\hat\Psi_{ft}$. Similarly, by speeding up shocks on the right of the function $\hat\psi_{ft}$, which have collided with $h_{k+\ell}$, we are preventing some collisions from occurring, further speeding up some shocks in $\hat\Psi_{ft}$. This will ensure 
\begin{align}\label{capvslowerequal}
\hat \Psi_{ft}(\cdot,t)=\hat \psi_{ft}(\cdot,t)
\end{align} 
on the interval $[ h_k(t), h_{k+\ell}(t)]$ (for all $t$).

Recall also that $ \hat v_i(x,t) > \hat v_{i+1}(x,t)$ for all  $i$ and for all $(x,t)$ and $\partial_u\sigma(u,v)=\partial_v\sigma(u,v)>0$.  Remark also  that the argument in \Cref{sec:improvevelcontr} (and in particular, \eqref{eq:velerr}) actually holds for every shock in $\hat\psi$. Recall also \eqref{eq:psipsift}. Finally, recall (the proof of) \eqref{eq:Bressan}. This will ensure \eqref{capvslowercase}.

\subsubsection{Position estimates on ``front tracking'' shocks}\label{sec:triangles}

By construction $\bar u^0$ is strictly decreasing with slope magnitude at least $ \epsilon$ on $[\hat h_k(0) , \hat h_{k+\ell}(0)]$ otherwise we would not have treated this part of the solution as rapidly decreasing. 
This property is preserved in the exact solution $\hat u_{ft}$ which, actually, gets steeper in time. In order to obtain estimates for the positions of discontinuities in $ \hat u$ we would like to argue that $\hat u_{ft}$ has a  minimal slope which is obviously not true, strictly speaking, since it consists of constant pieces.  We make the following observation instead. By construction, $\hat u_{ft}$ can be chosen such that
\begin{align}
    \inf \frac{ \hat u_{ft}(x,t) - \hat u_{ft}(y,t)}{y-x} > \frac{\epsilon}{2},\label{discret_slope}
\end{align} 
where at time $t$ the infimum runs over all $x,y$ in the interval between the first and last shocks in $\hat u_{ft}$ and verifying $y > x + \delta$ . This ``discrete slope'' of $\hat u_{ft}(\cdot, t)$ is computable and we denote it by $\mathcal{S}(\hat u_{ft}(\cdot,t))$. 

\begin{figure}[tb]
      \includegraphics[width=.8\textwidth]{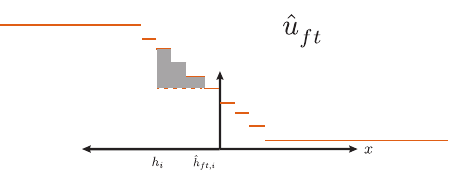}
  \caption{If the distance between $\hat{h}_{ft,i}$ and $h_i$ is too large, the $L^1$ distance between $\hat\Psi_{ft}$ and $\hat u_{ft}$ (represented by the gray ``triangle'' in this diagram) is too large, violating \eqref{capvslowercase}.}\label{triangel_pic}
\end{figure}

Thus, for shocks $h_i$ which are inside the interval $(h_{k}(t),h_{{k+\ell}}(t))$ at time $t$, from  \eqref{capvslowercase}, and \eqref{capvslowerequal} we obtain a maximal position error of each such shock of
\begin{align}\label{eq:puis}
\abs{h_i(t)-\hat{h}_{ft,i}(t)}  \leq  \frac{\sqrt{2}}{\sqrt{\max\{\sfrac{\epsilon}{2},\mathcal{S}(\hat u_{ft}(\cdot,t))  \}}}\bigg( \Gamma(t)\bigg)^{1/2}
     =: \Delta_{\text{inner}}.
\end{align}
Remark that if $\hat{h}_i(t)$ is in the interval $(\hat h_{k}(t),\hat h_{{k+\ell}}(t))$, then $\hat{h}_{ft,i}(t)=\hat{h}_i(t)$ and thus \eqref{eq:puis} gives us  control on $\Delta_i$, (recall $\Delta_i$ simply denotes our best upper bound on $\abs{h_i(t)-\hat{h}_{i}(t)}$). See \Cref{triangel_pic}.

Local versions of $\mathcal{S}(\hat u_{ft}(\cdot,t))$ can be computed to improve the estimate \eqref{eq:puis}.
In general, the formula \eqref{eq:puis} is only valid when the interval $(\hat{h}_{ft,i}-\Delta_{\text{inner}},\hat{h}_{ft,i}+\Delta_{\text{inner}})$ it gives for the position of the shock 
$h_i$ verifies 
\begin{align}\label{needstoverify}
(\hat{h}_{ft,i}-\Delta_{\text{inner}},\hat{h}_{ft,i}+\Delta_{\text{inner}})\subset [\hat{h}_{ft,k+1}(t), \hat{h}_{ft,k+\ell-1}(t)].
\end{align}
This is because outside the interval between the first and last shocks, $\hat u_{ft}$ is constant and thus the estimate \eqref{discret_slope} on the discrete slope is not valid. 

Let us then consider the case when both \eqref{needstoverify} and \eqref{eq:puis} are not verified. We first consider the left-most shocks in $\hat{u}_{ft}$ and $\hat\Psi_{ft}$. 

Our goal is to bound from above the difference 
$\hat{h}_{ft,k+1}(t)-h^{\hat{\Psi}}_{ft,k+1}(t)$, where $h^{\hat{\Psi}}_{ft,k+1}(t)$ is the position of the left-most shock in $\hat\Psi_{ft}$. Then, if 
  \begin{align}\label{difference_case}
   \hat{\Psi}_{ft}(h^{\hat{\Psi}}_{ft,k+1}(t)-,t)-\hat{\Psi}_{ft}(h^{\hat{\Psi}}_{ft,k+1}(t)+,t)\leq \bigg( \Gamma(t)\bigg)^{1/2},
   \end{align}
then $h^{\hat{\Psi}}_{ft,k+1}(t)$ and $\hat{h}_{ft,k+1}(t)$ verify 
\begin{equation}\label{resultcase1}
   \hat{h}_{ft,k+1}(t)-h^{\hat{\Psi}}_{ft,k+1}(t) \leq
   (\mathfrak{A} \hat{M}T+1) 
    \max_\alpha \Upsilon_\alpha(t) +\mathfrak{A}T\bigg( \Gamma(t)\bigg)^{1/2},
\end{equation}
where we have used \eqref{eq:psipsift} and \eqref{eq:velerr} and we recall that $\mathfrak{A}$ is an upper bound for the Lipschitz constant of Rankine-Hugoniot speeds (see \eqref{eq:RH}). As above, the $\max$ is taken over the set of all shocks in $\hat\psi_{ft}$.

If \eqref{difference_case} does not hold, then from \eqref{capvslowercase} we can deduce
\begin{align}\label{resultcase2}
  \hat{h}_{ft,k+1}(t)- h^{\hat{\Psi}}_{ft,k+1}(t) \leq \bigg( \Gamma(t)\bigg)^{1/2}.
\end{align}

We have similar estimates on the right-most shocks in $\hat{u}_{ft}$ and $\hat\Psi_{ft}$.

We conclude that if \eqref{eq:puis} is not verified, then
\begin{align}\label{conclusion_cases}
\abs{h_i(t)-\hat{h}_{ft,i}(t)} \leq \Delta_{\text{inner}}+ (\mathfrak{A} \hat{M}T+1)
    \max_\alpha \Upsilon_\alpha(t) +(\mathfrak{A}T+1)\bigg( \Gamma(t)\bigg)^{1/2},
\end{align}
because we know that $h_i(t)$ must be between either the left-most shock in $\hat\Psi_{ft}$ and the left-most shock in $\hat{u}_{ft}$, or the between the right-most shock  in $\hat\Psi_{ft}$ and the right-most shock in $\hat{u}_{ft}$. Remark that the right-hand side of \eqref{conclusion_cases} is always greater than or equal to the right-hand side of \eqref{eq:puis}, so \eqref{conclusion_cases} gives a worst-case bound.

As above with \eqref{eq:puis}, the formula \eqref{conclusion_cases} clearly only holds for shocks $i$ such that $h_i\in (h_{k}(t),h_{k+\ell}(t))$. Further, remark that if $\hat{h}_i(t)$ is in the interval $(\hat h_{k}(t),\hat h_{{k+\ell}}(t))$, then $\hat{h}_{ft,i}(t)=\hat{h}_i(t)$ and thus \eqref{conclusion_cases} gives us  control on $\Delta_i$.

If $h_i\notin (h_{k}(t),h_{k+\ell}(t))$ or $\hat h_i\notin (\hat h_{k}(t),\hat h_{k+\ell}(t))$, then there are three cases (assuming $\hat h_{k}(t),\hat h_{k+\ell}(t)$ are sufficiently far apart, otherwise the estimates are trivial):
\begin{itemize}
    \item[Case 1] Both $h_i$ and $\hat h_i$ have collided with a ``boundary'' shock. In this case, control on $\abs{h_i(t)-\hat{h}_{i}(t)}$ comes from \Cref{sec:bdyshocks}.
    \item[Case 2] If only $h_i$ collided with a ``boundary'' shock $h_\ast$ (here, $h_\ast=h_k$ or $h_\ast=h_{k+\ell}$), then use $\abs{h_i(t)-\hat{h}_{i}(t)}\leq $\abs{h_i(t)-\hat{h}_{\ast}(t)}+\abs{\hat h_\ast(t)-h_{\ast}(t)} and see again \Cref{sec:bdyshocks}.
    \item[Case 3] If only $\hat h_i$ collided with a ``boundary'' shock $\hat h_k$ (or $\hat h_{k+\ell}$), then use that $\hat h_{i+j}$ (or $\hat h_{i-j}$) has not collided with a ``boundary'' shock (for some $j\in\mathbb{N}$, choose the smallest possible) and neither has $h_{i+j}$ (or $h_{i-j}$). Then, we have control on $\abs{\hat h_{i+j}-h_{i+j}}$ ($\abs{\hat h_{i-j}-h_{i-j}}$) from \eqref{conclusion_cases}.  Moreover, $\abs{\hat h_{i+j}-\hat h_i}$ ($\abs{\hat h_{i-j}-\hat h_i}$) is expected to be small ($\mathcal{O}(\delta)$, but this can be checked a posteriori). Finally, remark that $h_i$ is between $h_{i+j}$ (or $h_{i-j}$) and $h_k$ (or $h_{k+\ell}$).
\end{itemize}




\subsection{Control on ``boundary'' shocks}\label{sec:bdyshocks}

Section \ref{sec:large} has provided control of position uncertainty for ``large'' shocks and Section \ref{sec:cdp} has provided control on position uncertainty for ``front tracking'' shocks (coming from the rapidly decreasing regions of $\bar{u}^0$) as long as they do not interact with other parts of the solution. It remains to control position uncertainty of left and right-most shocks in the ``front tracking'' parts of $\hat{u}, \hat \psi $ (corresponding to rapidly decreasing regions) -- the \emph{boundary shocks}. We are going to discuss the ``left-most shock'' case in detail where to the left of the shock is a $\hat v_i$ corresponding to a nearly non-decreasing region of $\bar{u}^0$ (see \Cref{boundry_fig}); the ``right-most shock'' case (with a nearly non-decreasing region to the right) is analogous. The case of a ``front tracking'' (corresponding to rapidly decreasing) region to both the left \emph{and} right of the shock is also analogous and we omit the details (see \Cref{sec:ftbothsides}).

\begin{figure}[tb]
      \includegraphics[width=.8\textwidth]{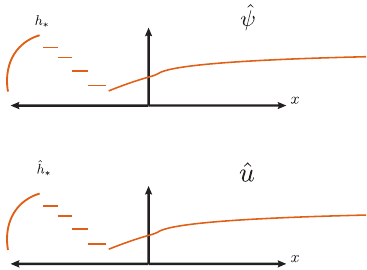}
  \caption{The definition of the $\ast$-th shock.}\label{boundry_fig}
\end{figure}

Note that the estimates from \Cref{sec:large} apply to these shocks but are not particularly useful (in particular, see \eqref{def_control_shift}). One issue is that dissipation estimate \eqref{neg_entropy_diss} only provides rather weak velocity control due to the small size of the shock but the key issue that needs to be sorted out is how many  of the ``front tracking'' shocks might collide with the shock in question, and to show that this cannot be a process feeding on itself.
Let us denote the shock in question as the $\ast$-th shock (see \Cref{boundry_fig}) and notice that, due to the position uncertainty of the small shocks to its right \eqref{eq:puis}, there is an immediate contribution to $\Delta_\ast$ due to the uncertainty about its right-hand state. The main idea is to transfer the control we have on ``front tracking'' shocks (see \Cref{sec:triangles}) to in turn control the ``boundary'' shocks.

\begin{figure}[tb]
    \centering
    \begin{subfigure}[b]{0.3\textwidth}
        \includegraphics[width=\textwidth]{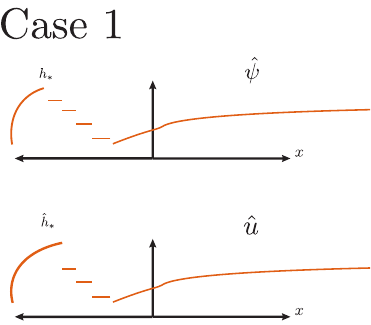}
        \label{fig:case1}
    \end{subfigure}
    ~ 
    \begin{subfigure}[b]{0.3\textwidth}
         \includegraphics[width=\textwidth]{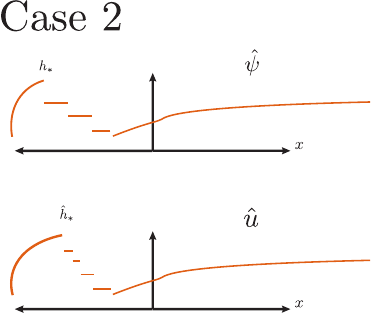}
        \label{fig:case2}
    \end{subfigure}
    ~ 
    
    \begin{subfigure}[b]{0.3\textwidth}

       \includegraphics[width=\textwidth]{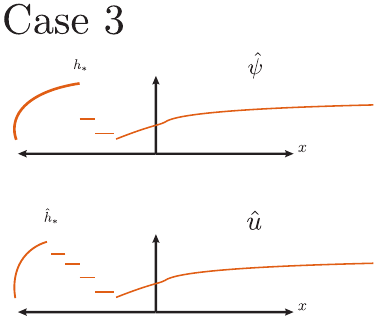}
        \label{fig:case3}  
    \end{subfigure}
\begin{subfigure}[b]{0.3\textwidth}
 \includegraphics[width=\textwidth]{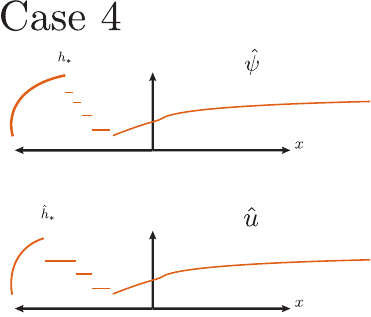}
        \label{fig:case4}
 \end{subfigure}

    \caption{The four cases used to control $\abs{\hat h_\ast- h_\ast}$.}\label{fig:4cases}
\end{figure}

The contributions $A_\ast, C_\ast$ to $\Delta_\ast$ are as above (see \eqref{def_control_shift}); it is the size of $B_\ast$ we wish to discuss. Let us stress that velocity errors only add to position errors if both have the same sign. Thus, we distinguish four cases (see \Cref{fig:4cases}):

\begin{enumerate}
    \item[Case 1]{\bf $\hat h_\ast> h_\ast$ , $\hat h_\ast$ has collided with more small shocks than $h_\ast$}:\\
    In this case, the erroneous collisions actually reduce the position error, $B_\ast$ does not contribute to the growth of $\Delta_\ast$.
    
     \item[Case 2:]{\bf $\hat h_\ast> h_\ast$ , $\hat h_\ast$ has collided with less small shocks than $h_\ast$}:\\
   Consider the last of the ``front tracking'' shocks to collide with $h_\ast$, and whose collision with $h_\ast$ was not reproduced in $\hat h_\ast$: let's call it $h_j$. Before $h_j$  collides with $h_\ast$, $h_j$ has a position error (from \eqref{conclusion_cases}) of 
   \begin{align}
       \abs{h_j-\hat{h}_j}<\Delta_{\text{inner}}+ (\mathfrak{A} \hat{M}T+1) 
    \max_\alpha \Upsilon_\alpha(t) +(\mathfrak{A}T+1)\bigg( \Gamma(t)\bigg)^{1/2},
   \end{align} 
   where $\Delta_{\text{inner}}$ is from \eqref{eq:puis}. The important remark here is that the formula \eqref{conclusion_cases} continues to hold for the shock $h_j$ even after the ``boundary'' shock $h_\ast$ has collided with the shock $h_j$: this will only push $h_j$ and $\hat{h}_j$ closer together, due to the collision with the ``boundary'' shock $h_\ast$ increasing the value of the left-hand state $\hat\psi(h_j(t)-,t)$. Remark here that we are using that $h_j$ is the last of the ``front tracking'' shocks to collide with $h_\ast$, and whose collision with $h_\ast$ was not reproduced in $\hat h_\ast$.

   To conclude, due to $\hat h_* < \hat h_j$, we have 

    \begin{align}
       \abs{h_*-\hat{h}_*}<\Delta_{\text{inner}}+ (\mathfrak{A} \hat{M}T+1) 
    \max_\alpha \Upsilon_\alpha(t) +(\mathfrak{A}T+1)\bigg( \Gamma(t)\bigg)^{1/2}.
   \end{align} 

    \item[Case 3:]{\bf  $\hat h_\ast< h_\ast$ , $\hat h_\ast$ has collided with fewer small shocks than $h_\ast$}:\\
    Here the erroneous collisions actually reduce the position error, $B_\ast$ does not contribute to the growth of $\Delta_\ast$.

    \item[Case 4:]{\bf $\hat h_\ast<  h_\ast$ , $\hat h_\ast$ has collided with more small shocks than $h_\ast$}:\\
    This case is exactly analogous to Case 2 above. 
    
    Consider the last of the ``front tracking'' shocks to collide with $\hat h_\ast$, and whose collision with $\hat h_\ast$ was not reproduced in $h_\ast$: let's call it $\hat h_j$. Before $\hat h_j$  collides with $\hat h_\ast$, $\hat h_j$ has a position error (from \eqref{conclusion_cases}) of
    \begin{align}
    \abs{h_j-\hat{h}_j}<\Delta_{\text{inner}}+ (\mathfrak{A} \hat{M}T+1) 
    \max_\alpha \Upsilon_\alpha(t) +(\mathfrak{A}T+1)\bigg( \Gamma(t)\bigg)^{1/2}.
    \end{align}

    Thus, due to $h_*<h_j$,
    \begin{align}
    \abs{h_*-\hat{h}_*}<\Delta_{\text{inner}}+ (\mathfrak{A} \hat{M}T+1) 
    \max_\alpha \Upsilon_\alpha(t) +(\mathfrak{A}T+1)\bigg( \Gamma(t)\bigg)^{1/2}.
    \end{align}
   
\end{enumerate}

All in all, in Case 1 and Case 3 above, this means that the impact of  $B_\ast$ on $\Delta_\ast$ is bounded by a term that contributes to the exponential growth factor in a Gronwall-type argument analogous to \eqref{gronwall_0_7252020_1}. We write,
\begin{equation}\label{eq:deltabs}
    \Delta_\ast \leq \tilde A_\ast + \tilde B_\ast  + \tilde C_\ast+\tilde D_\ast \\
\end{equation}
with
\begin{equation}\label{parts_control_shifts3}
\tilde A_\ast := \int\limits_0^t 
\tilde \zeta(t,s)\abs{\dot{\hat{h}}_\ast(s)-\sigma(\hat{v}_{\hat{L}_\ast(s)}(\hat{h}_\ast(s),s),\hat{v}_{\hat{R}_\ast(s)}(\hat{h}_i(s),s))}\, ds
\end{equation}
and
\begin{equation}\label{need_supremum_of_this3}
\tilde B_\ast := \int\limits_0^t
\tilde \zeta(t,s)\Big(\mathfrak{A}\Big|\hat{v}_{\hat{L}_\ast/\hat{R}_\ast(s)}(\hat{h}_\ast(s),s)
-\hat{v}_{L_\ast/R_\ast(s)}(\hat{h}_\ast(s),s)\Big|\Big) \, ds
\end{equation}
where the subscript is $L, R$ depending on whether $h_\ast$ is on the left or the right of the wave front tracking area.  The integrand  in $\tilde B_\ast$ accounts for uncertainty of trace values resulting from shocks coming from outside the wave front tracking area.
 
Lastly,

\begin{equation}\label{need_supremum_of_this_too3}
\tilde C_\ast := \norm{\frac{
\tilde \zeta(t,\cdot)}{\sqrt{\mathfrak{s}_\ast(\cdot)}} }_{L^2(0,t)}\cdot
\Bigg[ Ce^{C t }\Bigg(\int\limits_{-S}^{S}\eta(\bar{u}(x,0)|\hat{\psi}(x,0))\,dx 
+\mathcal{R}(t)\Bigg)
\Bigg]^{\frac{1}{2}},
\end{equation}
and
\begin{equation}
\tilde D_\ast := \zeta(t,0)\cdot \big[\mbox{maximum error from Case 2 or Case 4 above}\big],
\end{equation}
with $\zeta$ as in \eqref{zeta_def}. The ``maximum error from Case 2 or Case 4 above'' in $\tilde D_\ast$ means the largest possible value of $\abs{\hat h_\ast - h_\ast}$ assuming we are in Case 2 or Case 4 for the entire time interval $[0,T]$ -- this is necessary due to potentially switching into Case 1 or Case 3 after having been in Case 2 or Case 4 for a while.
Thus, we get
\begin{equation}\label{D_ast}
\tilde D_\ast \leq \zeta(t,0)\cdot \Bigg[\Delta_{\text{inner}}+ (\mathfrak{A} \hat{M}T+1) 
    \max_\alpha \Upsilon_\alpha(t) +(\mathfrak{A}T+1)\bigg( \Gamma(t)\bigg)^{1/2}\Bigg].
\end{equation}



\subsubsection{``Front tracking'' shock on both sides}\label{sec:ftbothsides}

As mentioned above, we now consider the case  when the ``boundary'' shock $\hat h_\ast, h_\ast$ might be adjacent to, both on the left and right, a ``front tracking'' shock. In other words, there is an interval corresponding to a rapidly decreasing region in $\bar{u}^0$ at ($t=0$) to both the left \emph{and} right of the shock $\hat h_\ast, h_\ast$ (for example, this can occur as in \Cref{twobdycollide2_fig}). This situation is analogous to the above and we omit the details. In particular, two major cases occur here. First, note there are two perspectives: the  ``left-most shock'' or ``right-most shock'', depending on which rapidly decreasing region we consider. The first case: in both the ``left-most shock'' or ``right-most shock'' perspective,  Case 1 or Case 3 above applies. The second case: from at least one of the  ``left-most shock'' or ``right-most shock'' perspective, Case 2 or Case 4 applies, and this gives us control on $\abs{\hat h_\ast - h_\ast}$.

\section{The $L^2$ estimate and dissipation calculations}\label{Revelation_section}

To prove \Cref{main_theorem}, we need to show \eqref{control_l21206} holds. Furthermore, we need to show that \eqref{def_control_shift} is a suitable upper bound for the quantity  $\abs{\hat{h}_i(t)-h_i(t)}$.

\subsection{Step 1: proof of \eqref{control_l21206}}

We now show how to control the growth in time of $\norm{\bar{u}(\cdot,t)-{\hat{\psi}}(\cdot,t)}_{L^2}$. We follow an argument from  \cite[Section 7]{chen2020uniqueness} (see also \cite{2017arXiv170905610K})\footnote{Our arguments in this section will involve stopping and restarting the clock every time there is a collision between generalized characteristics $h_i$. Weak solutions $u$ to \eqref{system} will automatically have $C^0([0,\infty);W^{-1,\infty}(\mathbb{R}))$ regularity. The integral formulation of the entropy inequality \eqref{entropy_integral} includes the $t=0$ boundary term, and from this we get $u$ is continuous at $t=0$ with values in $L^1_{\text{loc}}(\mathbb{R})$. Because $L^1_{\text{loc}}(\mathbb{R})$ is a strong topology, we get the same regularity for the function $t\mapsto \eta(u(\cdot,t))$. For positive times $t>0$, we do not enjoy this regularity and $\eta(u)$ is only defined at almost every time. However, this difficulty with stopping and restarting the clock at positive times is not a real roadblock and can handled with the use of approximate limits. We omit these technical details here. For a reference, see  \cite[Section 7 and Lemma 7.1]{chen2020uniqueness}.}.

We work on the cone of information. First, choose some $S>0$ such that 
\begin{align}
{S}>\max\set{\abs{x_1},\abs{x_N}}+sT+T\sup |A'|,
\end{align}
and where $s>0$ is chosen such that $\abs{q(a;b)}\leq s\eta(a|b)$ for all $a,b$ verifying $\abs{a},\abs{b}\leq\max\{\norm{\bar{u}}_{L^\infty},\norm{\hat{u}}_{L^\infty}\}$.
Then define $h_0(t)\coloneqq -S+st$ and $h_{N+1}(t)\coloneqq S-st$. Note that the speeds of the generalized characteristics $h_i$ are bounded by $\sup |A'|$ and thus $S$ is chosen such that the curves $h_0$ and $h_{N+1}$ do not intersect the generalized characteristic curves $h_i$ on the time interval $[0,T]$.  

Consider two successive times $\hat{t}_j<\hat{t}_{j+1}$ such that there is no interaction between the $h_i$ on the interval $[\hat{t}_j,\hat{t}_{j+1}]$. 

We denote by $\mathfrak{l}_j$ the set of all $i$ such that $h_i(t)\neq h_{i+1}(t)$ for all $t\in[{t}_j,{t}_{j+1}].$  Then, we fix $i \in \mathfrak{l}_j$ and  for $t\in[{t}_j,{t}_{j+1}]$, integrate \eqref{combined1_distributional} over the region 
\begin{align}
\set{(x,r)|{t}_j<r<t, h_i(r)<x<h_{i+1}(r)}.
\end{align}
Remark that on this set ${\hat{\psi}}=\hat{v}_{i+1}$. In the context of \eqref{combined1_distributional}, pick  $\bar{u}$ to be playing the role of $\hat{u}_2$ and pick $\hat{v}_{i+1}$ to be playing the role of $\hat{u}_1$. This gives,
\begin{multline}\label{dissipation_formula_psi1026}
\int\limits_{h_i(t)}^{h_{i+1}(t)}\eta(\bar{u}(x,t)|\hat{\psi}(x,t))\,dx-\int\limits_{h_i({t}_j)}^{h_{i+1}({t}_j)}\eta(\bar{u}(x,{t}_j)|\hat{\psi}(x,{t}_j))\,dx
\\
\leq\int\limits_{{t}_j}^{t} F_i^{+}(r)-F_{i+1}^{-}(r)\,dr
-\int\limits_{{t}_j}^{t} \int\limits_{h_i(r)}^{h_{i+1}(r)}\Bigg[\Bigg(\partial_x \bigg|_{(x,r)}\hspace{-.05in} \eta'({\hat{\psi}})\Bigg) A(\bar{u}(x,r)|{\hat{\psi}}(x,r))
\\
+R_i(x,t)\eta''({\hat{\psi}}(x,r))\big[\bar{u}(x,r)-{\hat{\psi}}(x,r)\big]
\Bigg]\,dxdr,
\end{multline}
where $R_{i}$ is the residual for the numerical solution $\hat{v}_{i}$ (see \eqref{system-numerical}), and
\begin{align}
    F_i^{\pm}(r)\coloneqq q(\bar{u}(h_i(r)\pm,r);{\hat{\psi}}(h_i(r)\pm,r))-\dot{h}_i(r)\eta(\bar{u}(h_i(r)\pm,t)|{\hat{\psi}}(h_i(r)\pm,r)).
\end{align}

We sum \eqref{dissipation_formula_psi1026} over all $i \in \mathfrak{l}_j$. 
We group the terms corresponding to $F_i^{\pm}$ together, and the $F_{i+1}^{\pm}$ together. This yields,
\begin{multline}\label{dissipation_formula_psi1026_1}
\int\limits_{-S+st}^{S-st}\eta(\bar{u}(x,t)|{\hat{\psi}}(x,t))\,dx-\int\limits_{-S+s{t}_j}^{S-s{t}_j}\eta(\bar{u}(x,\hat{t}_j)|{\hat{\psi}}(x,{t}_j))\,dx
\\
\leq\int\limits_{\hat{t}_j}^{t} \sum_{i} F_i^{+}(r)-F_{i}^{-}(r)\,dr
-\int\limits_{\hat{t}_j}^{t} \int\limits_{-S+sr}^{S-sr}\Bigg[\Bigg(\partial_x \bigg|_{(x,r)}\hspace{-.15in} \eta'({\hat{\psi}})\Bigg) A(\bar{u}(x,r)|{\hat{\psi}}(x,r))
\\
+R_{{\hat{\psi}}}(x,t)\eta''({\hat{\psi}}(x,r))\big[\bar{u}(x,r)-{\hat{\psi}}(x,r)\big]
\Bigg]\,dxdr,
\end{multline}
where we have used that $F_{N+1}^-,F_{0}^+\leq 0$ due to the definition of $s$ and $\dot{h}_0=s$, $\dot{h}_{N+1}=-s$. Furthermore, we define 
\begin{align}\label{residual_psi}
R_{{\hat{\psi}}}(x,t)\coloneqq
    \begin{cases}
    R_1(x,t) &\text{if } x<{h}_1(t),\\
    R_2(x,t) &\text{if } {h}_1(t)<x<{h}_2(t),\\
    &\vdots\\
    R_{N+1}(x,t) &\text{if } {h}_N(t)<x,\\
    \end{cases}
\end{align}
where $R_i$ is the residual for the numerical solution $\hat{v}_i$ (see \eqref{system-numerical}). Compare \eqref{residual_psi} with \eqref{def:cR}. In practice, we cannot compute $\norm{R_{{\hat{\psi}}}}_{L^2}$ due to the $h_i$ being unknown, so we must use \eqref{def:cR}.

Then, due to \Cref{neg_entropy_diss}, we have that 
\begin{align}\label{control_Fs}
    \sum_{i} F_i^{+}(r)-F_{i}^{-}(r)\leq 0.
\end{align}    
Observe that we can apply \Cref{neg_entropy_diss} because for all $i$ and almost every $r$, $\bar{u}(h_i(r)-,r)\geq \bar{u}(h_i(r)+,r)$. This is in fact true for any Lipschitz function $h_i$. In the case when $\bar{u}$ is in $BV$, this is well-known. When $\bar{u}$ is only known to verify the Strong Trace Property (\Cref{strong_trace_definition}), then this follows from Lemma 6 in \cite{Leger2011} (see also \cite[Lemma 5.2]{scalar_move_entire_solution}). Further, by assumption on the reconstruction process\footnote{For the $v_i$ corresponding to the ``front tracking'' shocks, see \Cref{do_not_sim_remark}.} we will also have $\hat{v}_{i}(x,t) > \hat{v}_{i+1}(x,t)$ for all $i$ and for all $(x,t)\in\mathbb{R}\times[0,T]$.

Consider now any $0<t<T$ and let $0<t_1<\cdots<t_J$ denote the times at which there is interaction between the $h_i$, with $t_0\coloneqq 0$ and $t_{J+1}\coloneqq t$. Then, we can write a telescoping sum 
\begin{multline}\label{almost_dissipation1017}
    \int\limits_{-S+st}^{S-st}\eta(\bar{u}(x,t)|{\hat{\psi}}(x,t))\,dx-\int\limits_{-S}^{S}\eta(\bar{u}(x,0)|{\hat{\psi}}(x,0))\,dx \\
    =\sum_{j=1}^{J+1}\Bigg[\int\limits_{-S+st_j}^{S-st_j}\eta(\bar{u}(x,t_j)|{\hat{\psi}}(x,t_j))\,dx-\int\limits_{-S+st_{j-1}}^{S-st_{j-1}}\eta(\bar{u}(x,t_{j-1})|{\hat{\psi}}(x,t_{j-1}))\,dx\Bigg]
    \end{multline}
    
 Then, from \eqref{dissipation_formula_psi1026_1} 
\begin{multline}\label{almost_dissipation1017_1}   
     \int\limits_{-S+st}^{S-st}\eta(\bar{u}(x,t)|{\hat{\psi}}(x,t))
     \,dx-\int\limits_{-S}^{S}\eta(\bar{u}(x,0)|{\hat{\psi}}(x,0))\,dx
     \\
     \leq
\int\limits_{0}^{t} \sum_{i} F_i^{+}(r)-F_{i}^{-}(r)\,dr
-\int\limits_{0}^{t} \int\limits_{-S+sr}^{S-sr}\Bigg[
\Bigg(\partial_x \bigg|_{(x,r)}\hspace{-.05in}
\eta'({\hat{\psi}})\Bigg) A(\bar{u}(x,r)|{\hat{\psi}}(x,r))
\\
+R_{{\hat{\psi}}}(x,t)\eta''({\hat{\psi}}(x,r))
\big[\bar{u}(x,r)-{\hat{\psi}}(x,r)\big]
\Bigg]\,dxdr,
\end{multline}
for all $t\in[0,T]$. Remark that we are only evaluating $\partial_x  \eta'({\hat{\psi}})$ away from the points $x=h_i(t)$ ($i=1,\ldots,N$), so we do not pick up any Dirac masses in the derivative. 

We move the first term on the right-hand side of \eqref{almost_dissipation1017_1} to the left-hand side and also use \eqref{control_Fs}, to get
\begin{equation}
    \begin{aligned}\label{dissipation1017}
     &  \int\limits_{0}^{t} \sum_{i} F_i^{-}(r)-F_{i}^{+}(r)\,dr
       \\
    & +   \int\limits_{-S+st}^{S-st}\eta(\bar{u}(x,t)|{\hat{\psi}}(x,t))\,dx-\int\limits_{-S}^{S}\eta(\bar{u}(x,0)|{\hat{\psi}}(x,0))\,dx \\
    \leq&
-\int\limits_{0}^{t}\bigg[\int\limits_{0}^{r} \sum_{i} F_i^{+}(r)-F_{i}^{-}(r)\,dr+ \int\limits_{-S+sr}^{S-sr}\Big[\bigg(\partial_x \bigg|_{(x,r)}\hspace{-.05in} \eta'({\hat{\psi}})\bigg) A(\bar{u}(x,r)|{\hat{\psi}}(x,r))
\\
&\qquad\qquad\qquad\qquad +R_{{\hat{\psi}}}(x,r)\eta''({\hat{\psi}}(x,r))\big[\bar{u}(x,r)-{\hat{\psi}}(x,r)\big]
\Big]\,dx\bigg]dr,
    \end{aligned}
\end{equation}
for all $t\in[0,T]$.

We now apply the Gronwall inequality to \eqref{dissipation1017}.

Recall that the relative flux $A(a|b)$ is locally quadratic in $a-b$.  Recall also \eqref{control_rel_entropy} and that $A(a|b)\geq 0$ due to the convexity of $A$. Thus, we only have to consider 
\begin{align}
\Bigg(\partial_x \bigg|_{(x,t)}\hspace{-.05in} \eta'({\hat{\psi}})\Bigg) A(\bar{u}(x,t)|{\hat{\psi}}(x,t))
\end{align}
when $\partial_x \bigg|_{(x,t)}\hspace{-.05in} \eta'({\hat{\psi}}) < 0$. 

The Gronwall and Young's inequality then give

\begin{multline}\label{result_gronwall1027}
\int\limits_{0}^{t} \sum_{i} F_i^{-}(r)-F_{i}^{+}(r)\,dr+\int\limits_{-S+st}^{S-st}\eta(\bar{u}(x,t)|{\hat{\psi}}(x,t))\,dx\\
\leq
C\Bigg(\int\limits_{-S}^{S}\eta(\bar{u}(x,0)|\hat{\psi}(x,0))\,dx +  \int\limits_{0}^{t}\int\limits_{-S+sr}^{S-sr}(R_{\hat{\psi}}(x,r))^2\,dxdr\Bigg)e^{C t},
\end{multline}
where the constant $C>0$ depends on $\max\set{\norm{\bar{u}}_{L^{\infty}},\norm{{\hat{\psi}}}_{L^{\infty}}}$, the  flux $A$, entropy $\eta$ and their derivatives. Furthermore, $C$ depends\footnote{Remark that if $\partial_x \bigg|_{(x,r)}\hspace{-.15in} \eta'({\hat{\psi}})\geq 0$, then we can drop the term $\bigg(\partial_x \bigg|_{(x,r)}\hspace{-.15in} \eta'({\hat{\psi}})\bigg) A(\bar{u}(x,r)|{\hat{\psi}}(x,r))$ from \eqref{almost_dissipation1017}.} on $\norm{\Big[\partial_x \bigg|_{(x,t)}\hspace{-.05in} \eta'({\hat{\psi}})\Big]_{-}}_{L^\infty}$, where $[\hspace{.03in}\cdot\hspace{.03in}]_{-}\coloneqq \min(0,\cdot)$. Note that by construction, we choose\footnote{Although, we remark that this choice is not strictly necessary for our arguments, and if $h_i(0)\neq\hat{h}_i(0)$, it will simply add a term to the Gronwall argument in \Cref{sec:def_control_shift} -- see \eqref{gronwall_0_7252020}.} that $h_i(0)=\hat{h}_i(0)$ (see \Cref{construction_section}) and thus ${\hat{\psi}}(\cdot,0)=\hat{u}(\cdot,0)$.

\subsection{Step 2A: calculation of dissipation due to shifting}

We now give control on $\abs{\hat{h}_i-h_i}$, but first we need to derive a key dissipation estimate. 

From \eqref{diss_neg_formula}, and the fact that the $h_i$ are generalized characteristics for $\bar{u}$,
we have for $i=1,\ldots,N$,
\begin{equation}
\begin{aligned}\label{control_shifts1}
F_i^{+}(t)-F_{i}^{-}(t)=&q(\bar{u}(h_i(t)+,t);{\hat{\psi}}(h_i(t)+,t))-q(\bar{u}(h_i(t)-,t);{\hat{\psi}}(h_i(t)-,t))
\\
&\hspace{.4in}-\dot{h}_i(t)\big(\eta(\bar{u}(h_i(t)+,t)|{\hat{\psi}}(h_i(t)+,t))-\eta(\bar{u}(h_i(t)-,t)|{\hat{\psi}}(h_i(t)-,t))\big)\\
&\hspace{-.7in}\leq - \frac{1}{12} \inf A'' \inf \eta''{\mathfrak{s}_i} \big((\bar{u}(h_i(t)+,t)-{\hat{\psi}}(h_i(t)+,t))^2+(\bar{u}(h_i(t)-,t)-{\hat{\psi}}(h_i(t)-,t))^2\big),
\end{aligned}
\end{equation}
where $\mathfrak{s}_i(t)>0$ is any lower bound for the shock size, i.e. any function of $t$ that verifies ${\hat{\psi}}(h_i(t)-,t)-{\hat{\psi}}(h_i(t)+,t)\geq\mathfrak{s}_i(t)$ for all $t$.

Then,
\begin{multline}\label{control_shifts2}
\big(\sigma(\bar{u}(h_i(t)+,t),\bar{u}(h_i(t)-,t))-\sigma({\hat{\psi}}(h_i(t)+,t),{\hat{\psi}}(h_i(t)-,t))\big)^2\\
\leq 2 \mathfrak{A} \big((\bar{u}(h_i(t)+,t)-{\hat{\psi}}(h_i(t)+,t))^2+(\bar{u}(h_i(t)-,t)-{\hat{\psi}}(h_i(t)-,t))^2\big),
\end{multline}
for a detailed proof of this, see \cite[p.~2516]{scalar_move_entire_solution}. Recall $\mathfrak{A}=\sup A''$.

Then, from \eqref{control_shifts1} and \eqref{control_shifts2}, we get
\begin{equation}
\begin{aligned}\label{control_shifts3}
&q(\bar{u}(h_i(t)+,t);{\hat{\psi}}(h_i(t)+,t))-q(\bar{u}(h_i(t)-,t);{\hat{\psi}}(h_i(t)-,t))
\\
&\hspace{.7in}-\dot{h}_i(t)\big(\eta(\bar{u}(h_i(t)+,t)|{\hat{\psi}}(h_i(t)+,t))-\eta(\bar{u}(h_i(t)-,t)|{\hat{\psi}}(h_i(t)-,t))\big)\\
&\leq -\frac{\inf A'' \inf \eta''}{24\mathfrak{A}} \mathfrak{s}_i(t)\big(\sigma(\bar{u}(h_i(t)+,t),\bar{u}(h_i(t)-,t))-\sigma({\hat{\psi}}(h_i(t)+,t),{\hat{\psi}}(h_i(t)-,t))\big)^2.
\end{aligned}
\end{equation}

Finally, use \eqref{control_shifts3}, \eqref{result_gronwall1027} and also remark that because $\eta(\cdot,\cdot)\geq 0 $, due to the convexity of $\eta$, we can drop the second term on the left-hand side of \eqref{result_gronwall1027}, to get
\begin{multline}\label{dissipation_formula_psi3}
\int\limits_{0}^{t}\sum_{i} \Bigg[{c{\mathfrak{s}_i}(r)}\big(\sigma(\bar{u}(h_i(r)+,r),\bar{u}(h_i(r)-,r))-\sigma({\hat{\psi}}(h_i(r)+,r),{\hat{\psi}}(h_i(r)-,r))\big)^2\Bigg]\,dr
\\
\leq
C\Bigg(\int\limits_{-S}^{S}\eta(\bar{u}(x,0)|\hat{\psi}(x,0))\,dx +  \int\limits_{0}^{t}\int\limits_{-S+sr}^{S-sr}(R_{\hat{\psi}}(x,r))^2\,dxdr\Bigg)e^{C t},
\end{multline}
with 
\begin{align}\label{def_c}
c\coloneqq \frac{\inf A'' \inf \eta''}{24\mathfrak{A}}.
\end{align}

Remark that following the definition of the set $\mathfrak{l}_j$, the sum $\sum_i [\ldots]$ in \eqref{dissipation_formula_psi3} is only over all those curves, at time $t=r$, that are distinct in $\hat \psi$ --  we are not allowed to count curves that have already collided multiple times.

In particular, for $i=1,\ldots,N$, we have 
\begin{multline}\label{dissipation_formula_psi4}
\int\limits_{0}^{t} \Bigg[{c{\mathfrak{s}_i}(r)}\big(\sigma(\bar{u}(h_i(r)+,r),\bar{u}(h_i(r)-,r))-\sigma({\hat{\psi}}(h_i(r)+,r),{\hat{\psi}}(h_i(r)-,r))\big)^2\Bigg]\,dr
\\
\leq
C\Bigg(\int\limits_{-S}^{S}\eta(\bar{u}(x,0)|\hat{\psi}(x,0))\,dx +  \int\limits_{0}^{t}\int\limits_{-S+sr}^{S-sr}(R_{\hat{\psi}}(x,r))^2\,dxdr\Bigg)e^{C t}.
\end{multline}

\subsection{Step 2B: proof that \eqref{def_control_shift} is upper bound}\label{sec:def_control_shift}

We gain control on the shock positions via a Gronwall argument.

Let $\hat{u}$ be from the context of the Main Theorem (\Cref{main_theorem}), with Lipschitz-continuous curves of discontinuity $\hat{h}_i(t)$. Note that $\dot{\hat{h}}_i$ is not necessarily the Rankine-Hugoniot speed of the discontinuity 
\begin{align}
    (\hat{u}(\hat{h}_i(t)+,t),\hat{u}(\hat{h}_i(t)-,t))
\end{align} 
due to numerical error.

We now give control on $h_i-\hat{h}_i$. Let $\hat{L}_i(t)\in\mathbb{N}$ and $\hat{R}_i(t)\in\mathbb{N}$ denote the integers corresponding to the $\hat{v}_i$ which are to the left and right, respectively, of the discontinuity $\hat{h}_i$ in $\hat{u}$ at time $t$.  Likewise, let $L_i(t)\in\mathbb{N}$ and $R_i(t)\in\mathbb{N}$ denote the integers corresponding to the $\hat{v}_i$ which are to the left and right, respectively, of the discontinuity ${h}_i$ in ${\hat{\psi}}$ at time $t$. Then, we can control $h_i-\hat{h}_i$ using the following Gronwall argument:
\begin{equation}
\begin{aligned}\label{control_shifts_Gronwall7242020_1}
\dot{\hat{h}}_i(t)-\dot{h}_i(t)=\Big(\dot{\hat{h}}_i(t)-\sigma(\hat{v}_{\hat{L}_i(t)}(\hat{h}_i(t),t),\hat{v}_{\hat{R}_i(t)}(\hat{h}_i(t),t))\Big)\qquad \qquad\qquad \qquad\qquad 
\\
+\Big(\sigma(\hat{v}_{\hat{L}_i(t)}(\hat{h}_i(t),t),\hat{v}_{\hat{R}_i(t)}(\hat{h}_i(t),t)) -\sigma(\hat{v}_{{L}_i(t)}(\hat{h}_i(t),t),\hat{v}_{{R}_i(t)}(\hat{h}_i(t),t)) \Big)
\\
+\Big(\sigma(\hat{v}_{{L}_i(t)}(\hat{h}_i(t),t),\hat{v}_{{R}_i(t)}(\hat{h}_i(t),t)) -\sigma(\hat{v}_{{L}_i(t)}(h_i(t),t),\hat{v}_{{R}_i(t)}(h_i(t),t)) \Big)
\\
+\Big(\sigma(\hat{v}_{{L}_i(t)}(h_i(t),t),\hat{v}_{{R}_i(t)}(h_i(t),t))-\sigma(\bar{u}(h_i(t)+,t),\bar{u}(h_i(t)-,t))\Big)
\end{aligned}
\end{equation}

Then, from \eqref{control_shifts_Gronwall7242020_1}, we get
\begin{equation}
\begin{aligned}
\abs{\dot{\hat{h}}_i(t)-\dot{h}_i(t)}\leq
\abs{\dot{\hat{h}}_i(t)-\sigma(\hat{v}_{\hat{L}_i(t)}(\hat{h}_i(t),t),\hat{v}_{\hat{R}_i(t)}(\hat{h}_i(t),t))}\qquad \qquad\qquad \qquad\qquad 
\\
+\abs{\sigma(\hat{v}_{\hat{L}_i(t)}(\hat{h}_i(t),t),\hat{v}_{\hat{R}_i(t)}(\hat{h}_i(t),t)) -\sigma(\hat{v}_{{L}_i(t)}(\hat{h}_i(t),t),\hat{v}_{{R}_i(t)}(\hat{h}_i(t),t)) }
\\
+\abs{\sigma(\hat{v}_{{L}_i(t)}(\hat{h}_i(t),t),\hat{v}_{{R}_i(t)}(\hat{h}_i(t),t)) -\sigma(\hat{v}_{{L}_i(t)}(h_i(t),t),\hat{v}_{{R}_i(t)}(h_i(t),t)) }
\\
+\abs{\sigma(\hat{v}_{{L}_i(t)}(h_i(t),t),\hat{v}_{{R}_i(t)}(h_i(t),t))-\sigma(\bar{u}(h_i(t)+,t),\bar{u}(h_i(t)-,t))}
\end{aligned}
\end{equation}

from which we infer
\begin{equation}
\begin{aligned}
\abs{\dot{\hat{h}}_i(t)-\dot{h}_i(t)}
\leq 
\abs{\dot{\hat{h}}_i(t)-\sigma(\hat{v}_{\hat{L}_i(t)}(\hat{h}_i(t),t),\hat{v}_{\hat{R}_i(t)}(\hat{h}_i(t),t))}\qquad \qquad\qquad \qquad
\\
+\abs{\sigma(\hat{v}_{\hat{L}_i(t)}(\hat{h}_i(t),t),\hat{v}_{\hat{R}_i(t)}(\hat{h}_i(t),t)) -\sigma(\hat{v}_{{L}_i(t)}(\hat{h}_i(t),t),\hat{v}_{{R}_i(t)}(\hat{h}_i(t),t)) }
\\
+\abs{\hat{h}_i(t)-h_i(t)}\mathfrak{A} \max\{\text{Lip}[\hat{v}_{{L}_i(t)}],\text{Lip}[\hat{v}_{{R}_i(t)}]\}
\\
+\abs{\sigma(\hat{v}_{{L}_i(t)}(h_i(t),t),\hat{v}_{{R}_i(t)}(h_i(t),t))-\sigma(\bar{u}(h_i(t)+,t),\bar{u}(h_i(t)-,t))}
\end{aligned}
\end{equation}

We then apply Gronwall's inequality, which yields, for $0<t<T$,
\begin{multline}\label{gronwall_0_7252020}
\abs{h_i(t)-\hat{h}_i(t)}\leq 
\zeta(t,0)\Bigg(\abs{h_i(0)-\hat{h}_i(0)} \\
+ \int\limits_0^t 
\zeta(0,s)\Bigg[\abs{\dot{\hat{h}}_i(s)-\sigma(\hat{v}_{\hat{L}_i(s)}(\hat{h}_i(s),s),\hat{v}_{\hat{R}_i(s)}(\hat{h}_i(s),s))}
\\
+\abs{\sigma(\hat{v}_{\hat{L}_i(s)}(\hat{h}_i(s),s),\hat{v}_{\hat{R}_i(s)}(\hat{h}_i(s),s)) -\sigma(\hat{v}_{{L}_i(s)}(\hat{h}_i(s),s),\hat{v}_{{R}_i(s)}(\hat{h}_i(s),s)) }
\\
+\abs{\sigma(\hat{v}_{{L}_i(s)}(h_i(s),s),\hat{v}_{{R}_i(s)}(h_i(s),s))-\sigma(\bar{u}(h_i(s)+,s),\bar{u}(h_i(s)-,s))}\Bigg]\,ds\Bigg).
\end{multline}

Note that to use the Gronwall inequality, we utilize 
\begin{align}
    \abs{\hat{h}_i(t)-h_i(t)}\leq \abs{\hat{h}_i(0)-h_i(0)}+\int\limits_0^t \abs{\dot{\hat{h}}_i(s)-\dot{h}_i(s)}\,dt,
\end{align}
which follows from the Fundamental Theorem of Calculus for $W^{1,1}_{\text{loc}}$ functions.

Moreover, remark that in \eqref{gronwall_0_7252020}, using the $\zeta$ function (see \eqref{zeta_def}) is a rough upper bound. In fact, within the $\zeta$ function, we only need to consider the Lipschitz constants of the $\hat v_i$ which might possibly be to the left or right of $h_i$ in $\hat \psi$ at time $s$.

Then \eqref{gronwall_0_7252020} can be rewritten as
\begin{multline}\label{gronwall_0_7252020_1}
\abs{h_i(t)-\hat{h}_i(t)}\leq 
\zeta(t,0)\abs{h_i(0)-\hat{h}_i(0)}
\\
+\int\limits_0^t 
\zeta(t,s)\Bigg[\abs{\Big(\dot{\hat{h}}_i(s)-\sigma(\hat{v}_{\hat{L}_i(s)}(\hat{h}_i(s),s),\hat{v}_{\hat{R}_i(s)}(\hat{h}_i(s),s))\Big)}
\\
+\abs{\sigma(\hat{v}_{\hat{L}_i(s)}(\hat{h}_i(s),s),\hat{v}_{\hat{R}_i(s)}(\hat{h}_i(s),s)) -\sigma(\hat{v}_{{L}_i(s)}(\hat{h}_i(s),s),\hat{v}_{{R}_i(s)}(\hat{h}_i(s),s))}
\\
+\abs{\sigma(\hat{v}_{{L}_i(s)}(h_i(s),s),\hat{v}_{{R}_i(s)}(h_i(s),s))-\sigma(\bar{u}(h_i(s)+,s),\bar{u}(h_i(s)-,s))}\Bigg]\,ds.
\end{multline}

Then, remark that from $L^2$-$L^2$ Hölder duality, we get,
\begin{multline}\label{gronwall_0_7252020_4}
    \abs{h_i(t)-\hat{h}_i(t)}\leq
\zeta(t,0)\abs{h_i(0)-\hat{h}_i(0)}
\\
+\int\limits_0^t 
\zeta(t,s)\Bigg[\abs{\dot{\hat{h}}_i(s)-\sigma(\hat{v}_{\hat{L}_i(s)}(\hat{h}_i(s),s),\hat{v}_{\hat{R}_i(s)}(\hat{h}_i(s),s))}
\\
+\abs{\sigma(\hat{v}_{\hat{L}_i(s)}(\hat{h}_i(s),s),\hat{v}_{\hat{R}_i(s)}(\hat{h}_i(s),s)) -\sigma(\hat{v}_{{L}_i(s)}(\hat{h}_i(s),s),\hat{v}_{{R}_i(s)}(\hat{h}_i(s),s)) } \Bigg]\, ds
\\
+
\norm{\frac{
\zeta(t,\cdot)}{\sqrt{\mathfrak{s}_i(\cdot)}} }_{L^2(0,t)}
\norm{\sqrt{\mathfrak{s}_i(\cdot)}\Big(\sigma(\hat{v}_{{L}_i(\cdot)}(h_i(\cdot),\cdot),\hat{v}_{{R}_i(\cdot)}(h_i(\cdot),\cdot))-\sigma(\bar{u}(h_i(\cdot)+,\cdot),\bar{u}(h_i(\cdot)-,\cdot))\Big)}_{L^2(0,t)},
\end{multline}
where we recall $\mathfrak{s}_i(t)>0$ is a lower bound for the size of the $i$-th shock.

Due to the definition of $L_i$ and $R_i$ we may use the dissipation estimate \eqref{dissipation_formula_psi4}, which implies
\begin{multline}\label{shift_gronwall_1027}
    \abs{h_i(t)-\hat{h}_i(t)}\leq
\zeta(t,0)\abs{h_i(0)-\hat{h}_i(0)}
\\
+\int\limits_0^t 
\zeta(t,s)\Bigg[\abs{\dot{\hat{h}}_i(s)-\sigma(\hat{v}_{\hat{L}_i(s)}(\hat{h}_i(s),s),\hat{v}_{\hat{R}_i(s)}(\hat{h}_i(s),s))}
\\
+\abs{\sigma(\hat{v}_{\hat{L}_i(s)}(\hat{h}_i(s),s),\hat{v}_{\hat{R}_i(s)}(\hat{h}_i(s),s)) -\sigma(\hat{v}_{{L}_i(s)}(\hat{h}_i(s),s),\hat{v}_{{R}_i(s)}(\hat{h}_i(s),s))} \Bigg]\, ds
\\
+
\norm{\frac{
\zeta(t,\cdot)}{\sqrt{\mathfrak{s}_i(\cdot)}} }_{L^2(0,t)}
 \Bigg[ Ce^{C t }\Bigg(\int\limits_{-S}^{S}\eta(\bar{u}(x,0)|\hat{\psi}(x,0))\,dx + \int\limits_{0}^{t}\int\limits_{-S+sr}^{S-sr}(R_{\hat{\psi}}(x,r))^2\,dxdr\Bigg) \Bigg]^{\frac{1}{2}}.
\end{multline}

Finally, \eqref{shift_gronwall_1027} gives us control on the  difference $\abs{h_i(t)-\hat{h}_i(t)}$. This yields the formula \eqref{def_control_shift}.

\section{Proof of \Cref{conv_thm}}\label{sec:proof_conv_thm}
\Cref{conv_thm} follows from the formulas for $\Delta_i$, see \eqref{def_control_shift}, \eqref{eq:deltabs}, and \Cref{sec:triangles}. This includes the logic for determining when shocks have collided  in $\hat \psi$ (see \Cref{sec:when_collided}). We also use the formulas for $\Gamma_d(t)$ (see \eqref{eq:Bressan}) and  $\Upsilon_d(t)$ (see in\eqref{eq:psipsift}).
Finally, we use that the function $\bar{u}^0$ is fixed, and $\bar{u}^0$ is broken into only a finite number $\bar N+1$ of intervals such that on each interval, $\bar{u}^0$ is either rapidly decreasing or nearly non-decreasing (this is the $\bar N$ in the context of \Cref{main_theorem}).

\section{Expected convergence of estimator}\label{sec:op-ed}

Let us discuss how we expect our error estimator to scale for mesh refinement, i.e. in the case $h, \delta \rightarrow 0$  and how $h, \delta$ should be related. Here, $\delta$ is from the context of the Main Theorem (\Cref{main_theorem}) and $h$ is the mesh width.

 In order to keep the discussion concise let us discuss the case of a dG scheme with polynomials of degree $0$ and an upwind type numerical flux, say Engquist-Osher, and an explicit SSP Runge-Kutta time discretization with a time step chosen according to a suitable CFL condition. In this case, ($L^2$-norms of) residuals are expected to be  $\mathcal{O}(h)$ provided the corresponding exact solution is $C^1$. This is the case if we can build extensions where kinks are very unlikely to be revealed in $\hat\psi$. In particular, this is the case if we only have ``large'' shocks.  This situation is discussed in \Cref{sec:crl}. Remark that by ``kink'', we mean a point that is not differentiable (in space) because it has differing one-sided derivatives.
 
 If the exact solution contains kinks (which are potentially revealed) then ($L^2$-norms of) residuals are expected to be  $\mathcal{O}(h^{3/4})$. This situation is discussed in \Cref{sec:crs}.

Let us remark, that we can only expect the error estimator to converge, for $\delta , h \rightarrow 0$, if the underlying scheme (as described above) is reasonable. For a scheme that is unstable or fails to converge the error estimator cannot converge. 

We stress that our error estimate is valid even if total variation and oscillation of the numerical solution is large, but, obviously, only convergent schemes can lead to convergent estimators. We will also assume that time step sizes are chosen proportional to $h$, which makes sense from the point of view of CFL conditions.

\subsection{``Large'' shocks}\label{sec:crl}
Let us focus, in this section, on the case that $\bar u^0$ only has nearly non-decreasing parts. In this case $\delta$ does not appear and there is a number of discontinuities in $\hat u$ that is independent of $\delta$.

The analysis and the numerical experiments in \cite{GiesselmannMakridakisPryer, DednerGiesselmann}
show that we can expect the residuals $R_i$ (see \eqref{system-numerical}) to have $L^2$ norms that are $\mathcal{O}(h)$ and the same is true for the initial data approximation error.  
Thus, until the first time that two shocks might have collided, we expect 
$C_i$ in \eqref{def_control_shift} to be  $\mathcal{O}(h)$, whereas $B_i$ is zero during this time and the size of $A_i$ 
depends on the Runge-Kutta scheme used for solving \eqref{ODEfor_u_hat} -- it can be easily made {much} smaller than $\mathcal{O}(h)$.

Once there is uncertainty of whether two shocks have merged, $B_i$ is of order one and the question is how long this time interval of uncertainty is. Since the shocks in question have a speed difference that is independent of $h$ and the position error of each is  proportional to $h$ we expect the interval of uncertainty to be proportional to $h$.

Thus, for each possible collision, the uncertainty in $\Delta_i$ increases by a term proportional to $h$.
Since the number of possible collisions is finite and independent of $h$ we expect $\Delta_i$ to be proportional to $h$.
Inserting this into \eqref{est:hphu}   leads to a bound for $\| \hat \psi - \hat u\|_{L^\infty(0,T; L^1(\mathbb{R}))}$ of order $h$, since $\mathcal{B}(t)$ (in the context of the Main Theorem \Cref{main_theorem}) consists of a number of intervals 
 each of which has a size proportional to $h$ and the function that is integrated has size independent of $h$. Moreover, since the initial data approximation error and the residuals are of order $h$
  the bound for $\| \hat \psi - \bar u\|_{L^\infty(0,T; L^2(\mathbb{R}))}$ from \eqref{control_l21206} is of order $h$.
Thus, since on a compact domain the $L^2$ norm controls the $L^1$ norm, the overall bound for $\| \hat u - \bar u\|_{L^\infty(0,T; L^1(\mathbb{R}))}$ (see \eqref{est:uhu}) is also of order $h$, which is optimal from an approximation theory viewpoint, although one might hope to obtain such estimates also in $L^p$-norms with $1 < p \leq 2$, cf. \cite{ZhangShu}. 
In the scenario where $\bar u^0$ has no rapidly decreasing parts, we indeed also obtain error estimates in the  $L^p(-S - st, S+st)$ norm for $1 < p \leq 2$ but the position uncertainty of the shocks leads  to (sub-optimal) error bounds of order $h^{\tfrac{1}{p}}$.
Note that the a posteriori error bound for the $L^\infty$-in-time $L^1$-in-space norm we obtain is $\mathcal{O}(h)$ which is much smaller than the $\mathcal{O}(h^{\tfrac12})$ bound obtained by Kruzhkov-type error estimates or the $\mathcal{O}(h^{\tfrac13})$ obtained in \cite{Bressan2020}. This comes at the price of significantly modifying the numerical scheme, though.

\subsection{``Front tracking'' shocks}\label{sec:crs}

If there are rapidly decreasing parts of $\bar u^0$ we have two smallness parameters $\delta, h$.
Compared to the only ``large'' shocks scenario discussed above we also have additional error estimator components: uncertainty of positions of ``boundary'' shocks and estimates for the $L^1$ error in the front tracking regions with ``front tracking'' shocks coming from parts where $\bar u^0$ is rapidly decreasing (see \Cref{shocks_fig}). These $L^1$ estimates are based on \eqref{eq:Bressan} and \eqref{eq:psipsift}.

The initial data approximation error $(\int \eta(\bar u(\cdot,0)| \hat \psi(\cdot,0)) dx)^{\frac{1}{2}}$ 
is expected to be  
\begin{align*}
    \mathcal{O}(\sqrt{\delta^2 + h^2})
\end{align*}
but not $o(\sqrt{\delta^2 + h^2})$.
In contrast, the corresponding term in \eqref{eq:velerr}, using $\hat \psi_{fine}$ instead of $\hat \psi$, is $\mathcal{O}(\sqrt{\delta^3 + h^2})$ since all the $\mathfrak{s}_\star$ are proportional to $\delta$ and we retain $h^2$ for the nearly non-decreasing pieces. 
Since the extensions of the initial data, from \Cref{ext_section}, contain (a finite number) of kinks (nondifferentiable points) which are likely to be revealed in $\hat\psi$, we expect the squared space-time $L^2$ norm of the residual to be  $\mathcal{O}(h^{3/2})$ but not $o(h^{3/2})$ for both $\hat \psi$ and $\hat \psi_{fine}$. 
Thus, the right-hand side of \eqref{eq:velerr} is
\[ \mathcal{O}\Big(\frac{1}{\sqrt{\delta}} (\delta^3 + h^2 + h^{3/2})^{\frac12} + \delta\Big) \]
but not smaller, in general, and we see that the optimal scaling for $\delta$ is $\delta \sim \sqrt{h}$
leading to $\Upsilon_\alpha \sim \delta \sim \sqrt{h}$ for all $\alpha$.
Using the facts that $\Upsilon_\alpha \sim \delta$ for all $\alpha$ and that the initial data approximation error is $\mathcal{O}(\sqrt{\delta^2 + h^2})$ and that the squared space-time $L^2$ norm of the residual is $\mathcal{O}(h^{3/2})$, we can infer the scaling of $\Gamma$ from  \eqref{eq:Bressan}
\[ 
\Gamma = \mathcal{O}\left( \sqrt{C + \frac{\Upsilon}{\delta} } \sqrt{ \delta^2 + h^2 + h^{3/2} } + \frac{\Upsilon^2}{\delta} + 
(C + \frac{\Upsilon}{\delta} )\Upsilon \right)
=\mathcal{O}(\delta) \sim \mathcal{O}(\sqrt{h})
\]
and finally $\Delta_{\text{inner}} \sim \sqrt{\Gamma}  \sim \sqrt{\delta} \sim h^{1/4}$ (see \eqref{eq:puis}). Here $C>0$ is a constant and  $\Upsilon$ denotes the maximum of the $\Upsilon_\alpha$ coming from the corresponding  rapidly decreasing piece.


All in all, the position uncertainty of ``boundary'' shocks is expected to be $\mathcal{O}(h^{1/4})$. The size of $C_i$ (see \eqref{need_supremum_of_this_too}) in position estimates of ``large'' shocks is $\mathcal{O}(\sqrt{\delta^2 + h^2})= \mathcal{O}(h^{1/2})$, but not smaller in general.

In contrast, 
$A_i$ stays small. In cases where ``large'' shocks might have interacted with ``boundary'' shocks the $B_i$ term for these shocks is expected to be of order $\mathcal{O}(h^{1/4})$ since the length of time intervals of uncertainty whether the collisions have taken place is $\mathcal{O}(h^{1/4})$ as it is dictated by the position uncertainty of the shock positions.  An analogous argument holds for collision of two ``boundary'' shocks originating from the same rapidly decreasing piece. The minimal velocity difference of these shocks is order $1$ and the position uncertainty of both is $\mathcal{O}(h^{1/4})$, thus the time interval of uncertainty whether the collision has happened is $\mathcal{O}(h^{1/4})$ as is the resultant additional position uncertainty (see \Cref{left_right_bdy}). The case of two ``boundary'' shocks separated by a $\hat v_i$ corresponding to a nearly non-decreasing interval of $\bar{u}^0$, and then colliding, is similar, with again a resultant additional position uncertainty of $\mathcal{O}(h^{1/4})$ See \Cref{nnd_separated}.

In any case, the uncertainty of ``boundary'' shock positions dominates all other terms in the $\sup_{t \in [0,T]} \| \bar u(\cdot,t) - \hat u(\cdot,t)\|_{L^1(-S +s t, S-st)
}$ error estimate such that we expect \[\sup_{t \in [0,T]} \| \bar u(\cdot,t) - \hat u(\cdot,t)\|_{L^1(-S +s t, S-st)}=
\mathcal{O}(h^{1/4}).\]

Note that this is a rather low convergence rate. Indeed, it is less than what is obtained in \cite{Bressan2020} and \cite{GiesselmannSikstel2023} for systems; and using Kruzhkov-type estimates for scalar problems.
Part of the reason for this is the suboptimal rate $h^{3/2}$ for the squared $L^2$ norm of the residuals, that is caused by the exact $v_i$ having kinks. In the case with only ``large'' shocks we could build the extensions in such a way that kinks are not seen (for $h$ small enough). This is not possible, if there are transitions from nearly non-decreasing to rapidly decreasing pieces in $\bar u^0$ without any jump at a point $x_i$, e.g. $\bar u^0$ is nearly non-decreasing on $(x_{i-1},x_i)$ and rapidly decreasing on $(x_{i},x_{i+1})$ and $\bar{u}^0(x_i-)=\bar{u}^0(x_i+)$. In this case, then there is only a very short area (with length proportional to $\delta$) on which we can extend the nearly non-decreasing piece without prescribing the slope $M$ - usually this piece is so short that is is revealed after a short simulation time. See the construction in \Cref{ext_section} and \eqref{slopeM}. Instead of having a kink in this piece, we could use a quadratic polynomial to build a $C^1$ extension but the $C^2$ norm of this extension would have to be proportional to $\delta^{-1}$. This extension, would improve the rate of the squared $L^2$ norm of the residual -- but not to $h^2$. Thus, we have opted for the simpler extensions including kinks.

One way to improve the rate of the squared $L^2$ norm of the residual  to $h^2$ would be to use numerical schemes on meshes that move such that the kink is aligned with some cell boundary for all times (using the fact that each extended initial data, following \Cref{ext_section}, has at most one or two kinks likely to be revealed and their positions can be computed for all times). Redoing the above computations in this case would suggest to choose $\delta \sim h^{2/3}$ and lead to an $\sup_{t \in [0,T]} \| \bar u(\cdot,t) - \hat u(\cdot,t)\|_{L^1 (-S +s t, S-st)}$ error estimator that is $\mathcal{O}(h^{1/3})$, i.e. with the same rate as the estimates in \cite{Bressan2020} and \cite{GiesselmannSikstel2023}. In particular, this coincides with the scaling of a posteriori error bounds obtained by Bressan and co-workers \cite{Bressan2020}.

\section{Numerical experiments}\label{sec:numerical_experiments}


We have implemented our error estimator in MATLAB and carried out several numerical experiments in order to experimentally asses its scaling behaviour, i.e. to check whether it scales as discussed in \Cref{sec:op-ed}. 
Our experiments are carried out for Burgers equation and use a first-order piece-wise constant finite volume scheme, with explicit Euler time-stepping and upwind numerical flux. Our code is available at 
\begin{center}
\url{https://git-ce.rwth-aachen.de/jan.giesselmann/shiftsshocksaposteriori.git}
\end{center}
The whole purpose of the code is to provide validation of the scaling; it is not optimized in any way. In particular, the preprocessing of the initial data, i.e., its decomposition into nearly non-decreasing and rapidly decreasing parts, is not automated.

We have conducted numerical experiments for two initial data: One in which only ``large'' shocks occur and a second one in which a rapidly decreasing piece forms a shock during the simulation time.  

Before we present the experiments, let us mention one trick we used to slightly decrease the computational costs of the code:
Our error estimator contains several terms in which integrals are multiplied by exponential functions, and approximating these integrals by quadrature in each time step is inconvenient. Instead, we compute these terms iteratively based on their value  in the previous time step which is simplified by the fact that the integrands are constant in each time step. Indeed, for any two non-negative sequences $\{\alpha_k\}_{k \in \mathbb{N}},\ \{\beta_k\}_{k \in \mathbb{N}}$, the following  holds:
\begin{equation}
     F_{i} :=
    \sum_{j =0}^{i} \alpha_j \Delta t \exp\Big(\sum_{k=j}^{i} \beta_k \Delta t\Big)
\end{equation}
 satisfies
 \begin{equation}
    F_{i+1}    = F_i \exp\Big(\beta_{i+1} \Delta t\Big)
 +  \alpha_{i+1} \Delta t \exp\Big(\beta_{i+1} \Delta t\Big).
 \end{equation}

\subsection{Experiment with ``large'' shocks only}\label{sec:numls}
In this experiment the computational space-time domain is $[0,1]\times[0,.3]$ and the initial data are given by 
\begin{equation}\label{experiment:large}
    \bar u^0(x)= \left\{\begin{array}{ccc}
        3 & \text{if} & x < 0.25\\
        1 + 2x & \text{if} & 0.25 < x < 0.5\\
        x &\text{if} & 0.5 < x < 0.625\\
        0 & \text{if} & x > 0.625
    \end{array}
    \right.
\end{equation}

In \Cref{tab:1}, we report the upper bounds for position errors, $\sup_{t \in [0,T]} \|\bar u (\cdot, t) - \hat \psi(\cdot, t)\|_{L^2}$, and $\sup_{t \in [0,T]}\|\bar u (\cdot, t) - \hat u(\cdot, t)\|_{L^1}$ provided by our error estimator as well as $\| \hat u_{\text{fine}} (\cdot, T) - \hat u(\cdot, T)\|_{L^1}$, where  $\hat u_{\text{fine}}$ is a finite volume solution on a much finer mesh, and we use this as an approximation of the ``true'' value of $\|\bar u (\cdot, T) - \hat u(\cdot, T)\|_{L^1}$.

In \Cref{fig:snapshots} we show snapshots of the numerical solution at times 
\begin{align*}
    t\in \{ 0.05, 0.1, 0.15, 0.2, 0.25, 0.3\}.
\end{align*}
In the areas where it is unclear which $\hat v_i$ is revealed in $\hat \psi$, we color in red the branches of the $\hat v_i$ that are not revealed in $\hat u$ but that might be revealed in $\hat \psi$.
We observe that, as expected, all monitored quantities converge with order $\mathcal{O}(h)$. 

For each quantity in an even column we display its \textit{Experimental order of Convergence} (EoC) in the adjacent column to the right.
For any quantity $E(h)$ that depends on the mesh width $h$ we define the EoC for two consecutive mesh widths $h_1$ and $h_2$ as
\[ 
\operatorname{EoC}_{12}= \frac{\log \left( \frac{E(h_1)}{E(h_2)}\right)}{\log\left( \frac{h_1}{h_2} \right)},
\]
the idea being that if $E(h)=C h^p$ then $\operatorname{EoC}=p$.

\begin{table}[!htb]
  \begin{center}
    \begin{tabular}{c|c|c|c|c|c|c|c|c|c|c}
      $\# cells$ &
      $\max_i \Delta_i$ &EoC &
      $\|\bar u - \hat \psi\|_{L^2}\! $  & EoC &
       $\|\bar u - \hat u\|_{L^1}\! $ & EoC &
      $\| \hat u_{\text{fine}}  - \hat u\|_{L^1}$ & EoC  
      \\
      \hline
     3200 & $4.13 \cdot 10^{-3}$ & ---  & $3.26 \cdot 10^{-4}$  & ---  & $9.46\cdot 10^{-3}$ & ---  & $1.44\cdot 10^{-4}$   & ---    \\
     6400 & $2.06 \cdot 10^{-3}$ & 1.0 & $1.63 \cdot 10^{-4}$  & 1.0 & $4.73\cdot 10^{-3}$ & 1.0 &  $8.21\cdot 10^{-5}$  & 0.82  \\
    12800 & $1.03 \cdot 10^{-3}$ & 1.0 &  $8.14 \cdot 10^{-5}$ & 1.0 & $2.35 \cdot 10^{-3}$  & 1.0 & $3.59\cdot 10^{-5}$  & 1.19  \\
    25600 & $5.14 \cdot 10^{-4}$   & 1.0 &  $4.07 \cdot 10^{-5}$   & 1.0 & $1.18 \cdot 10^{-3}$   & 1.0  &   $2.05\cdot 10^{-5}$     &   0.81   \\
    \end{tabular}
  \end{center}
  \caption{Error estimator results for the example with ``large'' shocks only: estimates on $\|\bar u - \hat u\|_{L^1}\! $ and  $\|\bar u - \hat \psi\|_{L^2}\! $. We also include the values of $\| \hat u- \hat{u}_{\text{fine}}\|_{L^1} $.}
  \label{tab:1}
\end{table}
\begin{figure}[ht!]
\caption[]
          {
            \label{fig:snapshots} 
            From initial data \eqref{experiment:large}, we have snapshots of the numerical solution, with 800 cells, at different times. The parts of the $\hat v_i$ which are not revealed in $\hat u$ but may be revealed in $\hat\psi$ are colored in red.}
          \begin{center}
            \begin{subfigure}[b]{0.30\linewidth}
              \includegraphics[width=\linewidth]
                              {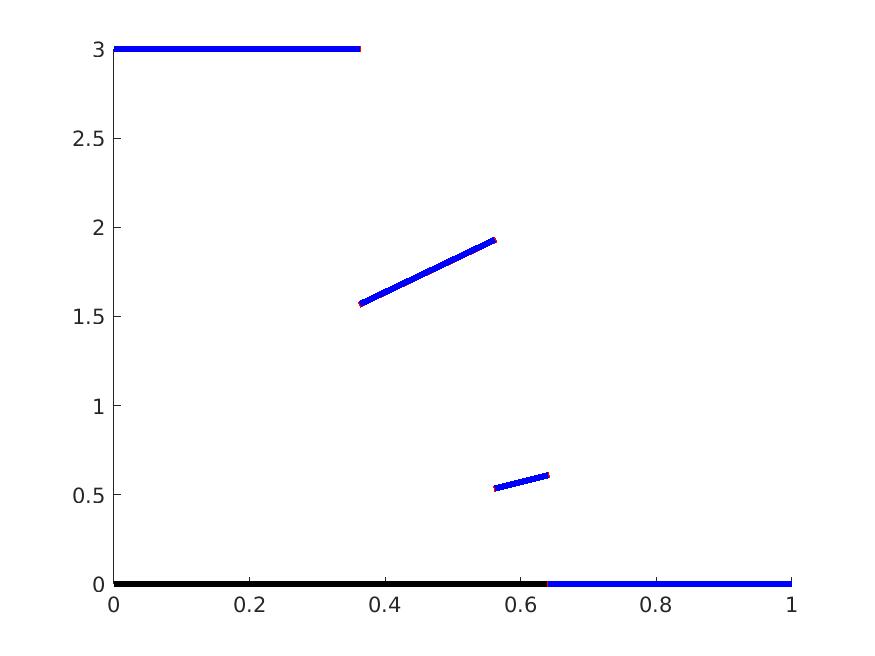}
              \caption{$t=0.05$}
            \end{subfigure}
            \begin{subfigure}[b]{0.30\linewidth}
              \includegraphics[width=\linewidth]
                              {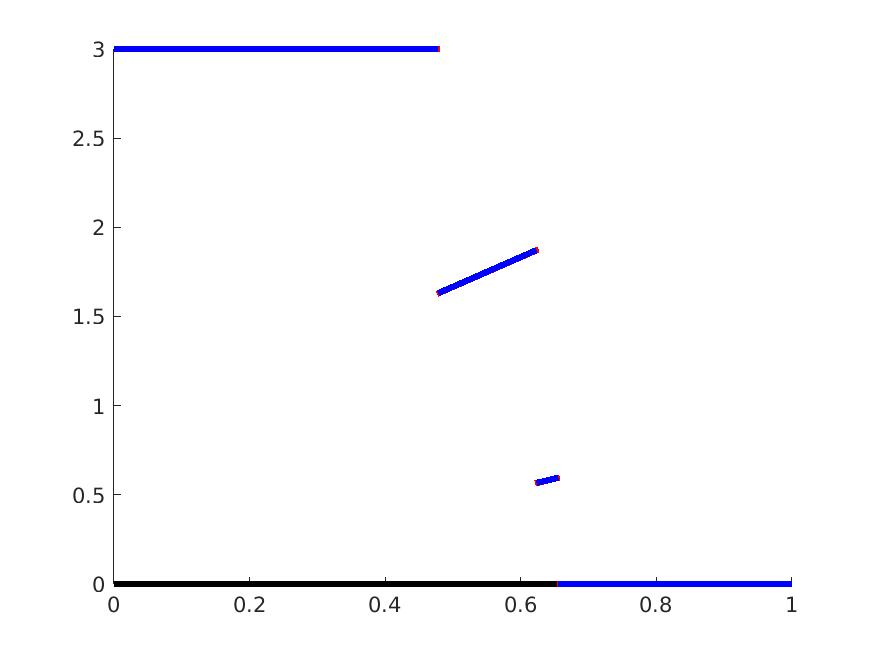}
              \caption{$t=0.1$}
            \end{subfigure}
            \begin{subfigure}[b]{0.30\linewidth}
              \includegraphics[width=\linewidth]
                              {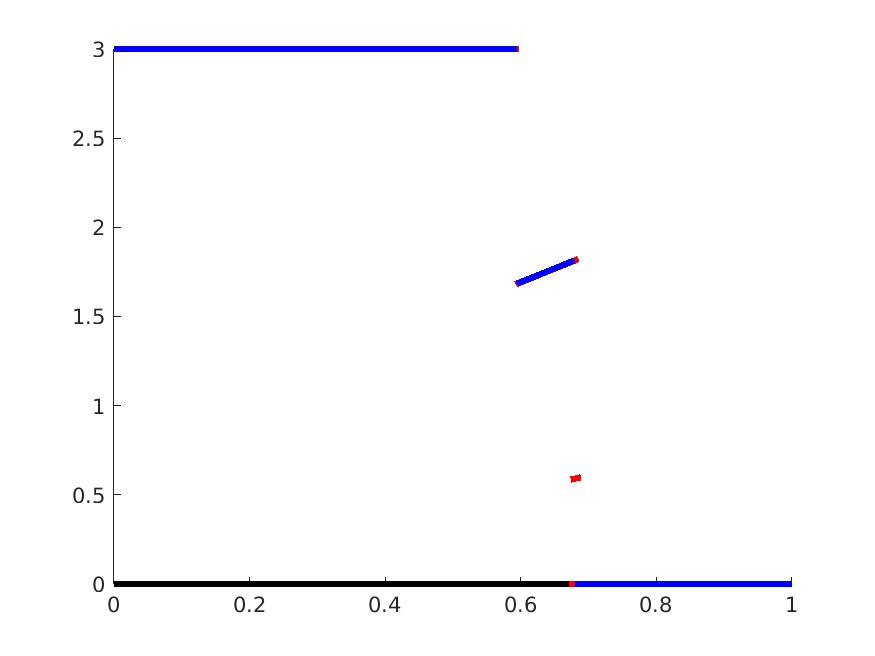}
              \caption{$t=0.15$}
            \end{subfigure}

            \begin{subfigure}[b]{0.30\linewidth}
              \includegraphics[width=\linewidth]
                              {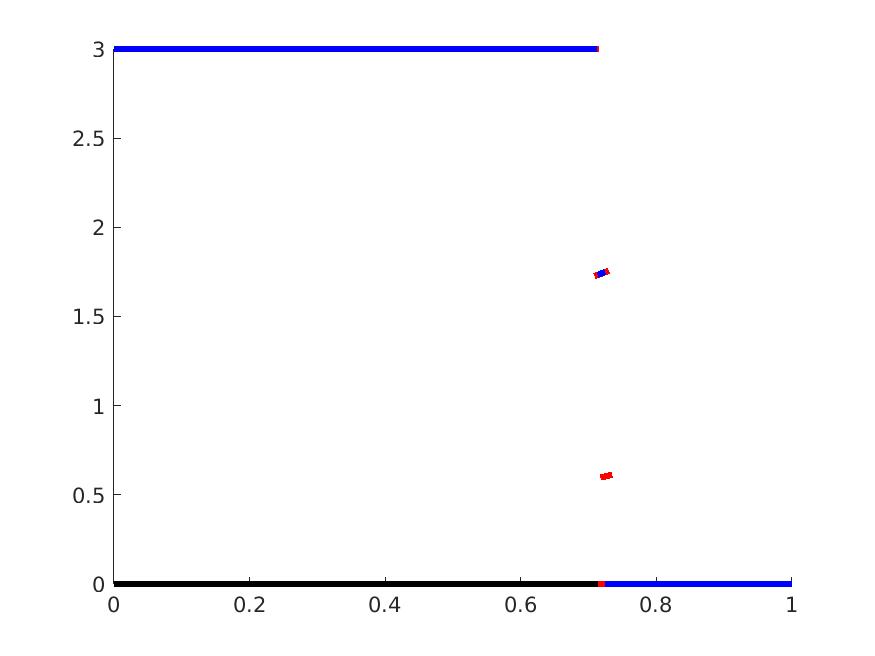}
              \caption{$t=0.2$}
            \end{subfigure}
            \begin{subfigure}[b]{0.30\linewidth}
              \includegraphics[width=\linewidth]
                              {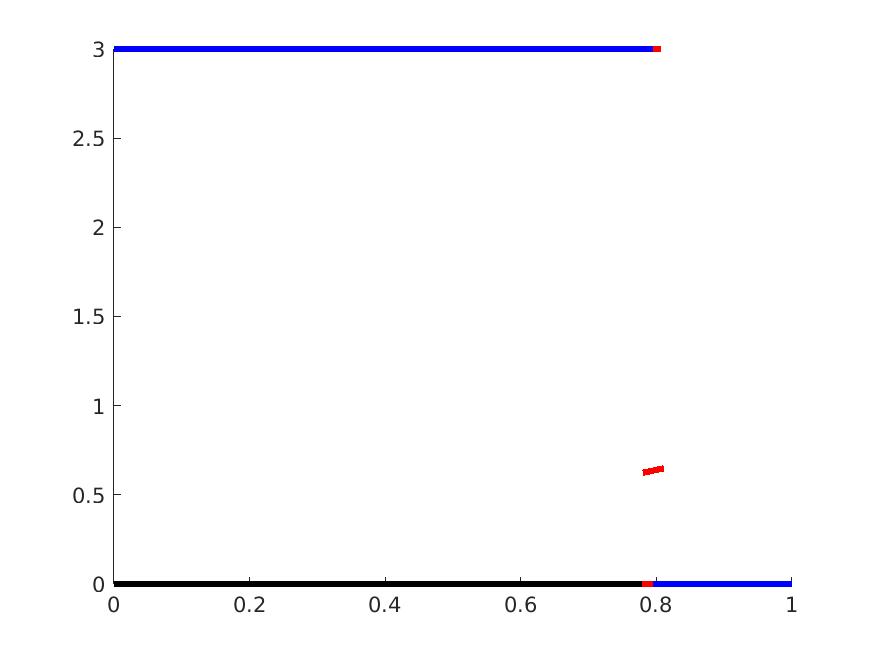}
              \caption{$t=0.25$}
            \end{subfigure}
            \begin{subfigure}[b]{0.30\linewidth}
              \includegraphics[width=\linewidth]
                              {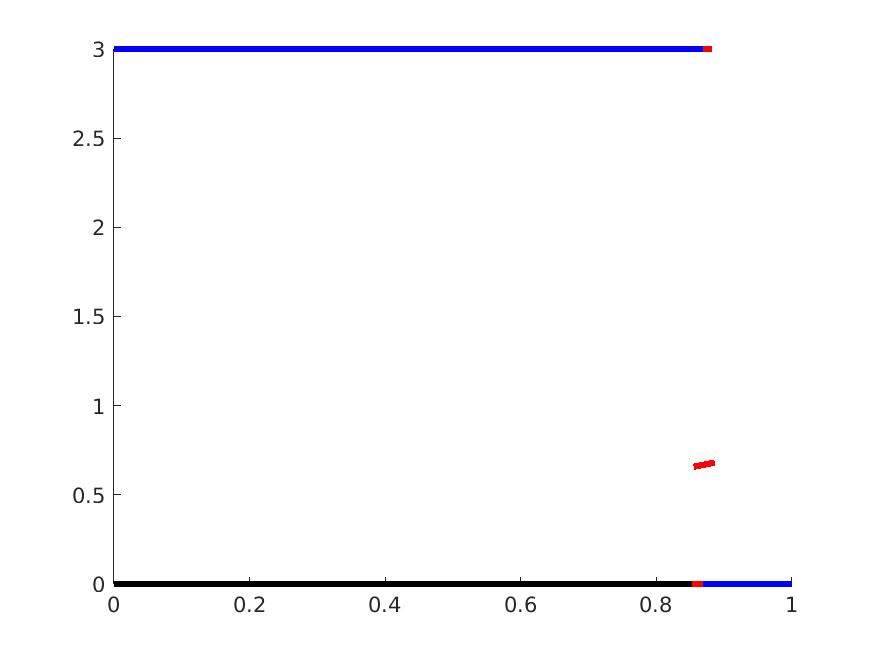}
              \caption{$t=0.3$}
            \end{subfigure}            
          \end{center}
\end{figure}

\subsection{Experiment including a rapidly decreasing piece}\label{sec:numdp}
In this experiment, the computational space-time domain is $[0,1]\times[0,.22]$ and the initial data are given by 
\begin{equation}
    \bar u^0(x)= \left\{\begin{array}{ccc}
        2.5 & \text{if} & x < 0.2\\
        3.5 -5x & \text{if} & 0.2 < x < 0.4\\
        1.1 + x &\text{if} & 0.4 < x < 0.625\\
        0 & \text{if} & x > 0.625
    \end{array}
    \right.
\end{equation}

The time interval in the example is long enough for the ``boundary'' shocks of the rapidly decreasing piece to collide (in the exact solution and in the computer).

In \Cref{tab:2}, we report $\max_\alpha \Upsilon_\alpha(T)$, $\Gamma$ and $\Delta_{\text{inner}}$ for the rapidly decreasing piece for a sequence of meshes. Note that these terms scale as predicted in \Cref{sec:crs}.
In \Cref{tab:3}, we also report the upper bounds on $\sup_{t \in [0,T]} \|\bar u (\cdot, t) - \hat \psi(\cdot, t)\|_{L^2}$ and $\sup_{t \in [0,T]}\|\bar u (\cdot, t) - \hat u(\cdot, t)\|_{L^1}$ provided by our analysis as well as $\| \hat u_{\text{fine}} (\cdot, T) - \hat u(\cdot, T)\|_{L^1}$, where  $\hat u_{\text{fine}}$ is a finite volume solution on a much finer mesh and we use this as an approximation of the ``true'' value of $\|\bar u (\cdot, T) - \hat u(\cdot, T)\|_{L^1}$.
Note that in this experiment the simulation ``sees'' the kinks in the extensions of the increasing pieces (as discussed in \Cref{sec:op-ed}, above). Thus $\sup_{t \in [0,T]} \|\bar u (\cdot, t) - \hat \psi(\cdot, t)\|_{L^2}$ is $\mathcal{O}(h^{3/4})$ but not $\mathcal{O}(h)$ as in the previous experiment.

We also note that the bound for $\sup_{t \in [0,T]} \| \bar u (\cdot, t) - \hat u(\cdot, t)\|_{L^1}$ decreases much more slowly than anticipated. This can be understood as follows: 

One major source of position uncertainty of shocks is potential velocity errors during time intervals  when it is not clear whether ``boundary'' shocks have collided in $\hat \psi$.
The length of these time intervals of uncertainty does not show the scaling behavior we expect. We attribute this to the strong time dependence of our bounds on the position error of shocks.
Indeed, if we plot $\Delta_{\text{inner}}$ (see \eqref{eq:puis})  we see that for any fixed $t$ we have 
$\Delta_{\text{inner}}(t)= \mathcal{O}(h^{1/4})$ but the constant grows quickly in time.

In order to understand how this time-dependence leads to the observed behavior of the lengths of the time intervals of uncertainty mentioned above, let us consider, as an heuristic example,
 the time when two shocks actually collide in $\hat\psi$ compared to the time
 when our code concludes that a \emph{collision might have happened}. Remark this is a subset of the length of time during which the code is uncertain whether or not a collision has occurred in $\hat\psi$. Suppose the initial position difference of the shocks is $b$ and their speed difference is the constant value $a$. Then they collide at time $\tfrac{b}{a}$ in the exact solution. Heuristically, we suppose $a$ and $b$ are reproduced accurately by our code and we will ignore any errors in this regard in the current discussion.
 If our bounds on position errors of shocks were independent of time, then 
the time $t^*$, when the code concludes that a collision might have happened is (approximately) determined as the solution of
\begin{equation}
    b = a t^* + c(h),
\end{equation}
where $c(h)$ denotes our bound on the sum of the position errors of the two shocks. Then, 
$t^*= \frac{b - c(h)}{a}$ and thus the length of the time between when the code determines a collision might have occurred, and when it actually does occur in $\hat\psi$ is  given by 
\[\frac{b}{a} - \frac{b - c(h)}{a} = \frac{c(h)}{a},\] 
i.e. a quantity that will exactly reproduce the scaling behaviour of $c(h)$.

In contrast, suppose that our bounds on position errors of shocks depend on time linearly. Then $t^*$ is given as the solution of
\begin{equation}\label{eq:tint}
    b = a t^* + c(h) t^*
\end{equation}
i.e., 
$t^*= \frac{b }{a+  c(h)}$.

Thus, in this case the length of time between when the code determines a collision might have occurred in $\hat\psi$, and when it actually occurs in $\hat\psi$, is given by 
\[\frac{b}{a} - \frac{b }{a+ c(h)} = \frac{b c(h)}{a+ c(h)}.\] 
This quantity shows the same behavior as $c(h)$ when $c(h)$ is small (compared to $a$) but not when $c(h)$ is not small: in the latter (pre-asymptotic) case the quantity from \eqref{eq:tint} decays much more slowly than $c(h)$.
Analogous considerations apply when the time dependence is nonlinear.



In our numerical experiments, the length of the interval of collision uncertainty is clearly still in the pre-asymptotic regime, where its experimental convergence rate is less than that of $c(h)$.
The argument above shows that asymptotically for $h \rightarrow 0$ (where also $c(h) \rightarrow 0$)
we can expect the length of the time interval of uncertainty to behave like $c(h)$, i.e., to be $\mathcal{O}(h^{1/4})$.

The long computing times of our code and the low convergence rate of $\Delta_{\text{inner}}$ have convinced us not to pursue computations where the asymptotic behavior of our estimate for $\sup_{t \in [0,T]}\|\bar u (\cdot, t) - \hat u(\cdot, t)\|_{L^1}$ can be seen.

\begin{table}[!htb]
  \begin{center}
    \begin{tabular}{c|c|c|c|c|c|c}
      $\# cells$ &
      $\!\max_\alpha \Upsilon_\alpha\!$ &EoC &
      $\Gamma$  & EoC &
      $\Delta_{\text{inner}} $ & EoC 
      \\
      \hline
     3200 & 0.099 & ---  & 0.213  & ---  & 0.292 & ---    \\
     6400 & 0.073 & 0.44 & 0.156  & 0.45 & 0.249 & 0.23  \\
    12800 & 0.050 & 0.55 & 0.106  & 0.56 & 0.206 & 0.27   \\
    25600 & $0.036$   & 0.47  & $0.075$    & 0.50  & 0.174   & 0.24    \\
    \end{tabular}
  \end{center}
  \caption{Observed values of $\max_\alpha \Upsilon_\alpha$, $\Gamma$, and $\Delta_{\text{inner}}$ (see \eqref{eq:puis}) for the example with one rapidly decreasing piece.}
  \label{tab:2}
\end{table}
\begin{table}[!htb]
  \begin{center}
    \begin{tabular}{c|c|c|c|c|c|c}
      $\# cells$ &
      $\| \hat u- \hat{u}_{\text{fine}}\|_{L^1} $ & EoC &
      $\|\bar u - \hat u\|_{L^1}\! $& EoC  &
      $\|\bar u - \hat \psi\|_{L^2}\! $& EoC  
      \\
      \hline
     3200 & $1.07 \cdot 10^{-3}$ & ---  &  2.61  & ---  & $2.6 \cdot 10^{-3}$ & ---  \\
     6400 & $4.97 \cdot 10^{-4}$ & 1.1 &  2.55  & 0.03 & $1.5 \cdot 10^{-3}$ & 0.79 \\
    12800 &  $2.40 \cdot 10^{-4}$ & 1.0&  2.47  & 0.05 & $9.1 \cdot 10^{-4}$ & 0.72 \\
    25600 &  $1.22 \cdot 10^{-4}$   & 1.0  &     2.38   &   0.05   & $5.4 \cdot 10^{-4}$&0.75  \\
    \end{tabular}
  \end{center}
  \caption{Error estimator results for the example with one rapidly decreasing piece: estimates on $\|\bar u - \hat u\|_{L^1}\! $ and  $\|\bar u - \hat \psi\|_{L^2}\! $. We also include the values of $\| \hat u- \hat{u}_{\text{fine}}\|_{L^1} $.} 
  \label{tab:3}
\end{table}

\appendix
\section{ }

\subsection{Defining the extensions to make the $v_i^0$. }\label{ext_section}

\begin{figure}[tb]
      \includegraphics[width=.8\textwidth]{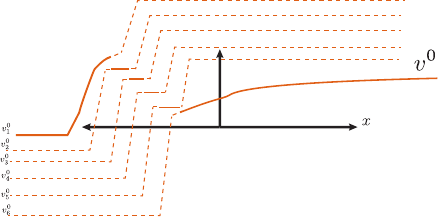}
  \caption{Given a function $v^0\in\inise$, in the present paper we create piecewise-linear artificial extensions defined on $\mathbb R$. In this diagram, the extensions are denoted by dotted lines (note the diagram is not to scale).}\label{ext_fig}
\end{figure}

We use the process of creating artificial extensions as given in \cite{2017arXiv170905610K}. However, we modify the procedure given in \cite{2017arXiv170905610K}. Unlike in \cite{2017arXiv170905610K}, we want the artificial extensions to have a pointwise lower bound on their difference (see \eqref{ordering_vs}, below). We also want our extensions to be $C^1$, at least in regions where we expect the extensions might be revealed (however, even with the extensions constructed in this section, it is still possible that nondifferentiable points, i.e. kinks, are revealed, see the discussion in \Cref{sec:op-ed}).  See \Cref{ext_fig}.

Consider initial data $v^0\in\inise$ (see \eqref{LIWASe}). 

We use the $N\in \mathbb{N}$ and  $-\infty=x_0<x_1<\cdots<x_{N}<x_{N+1}=+\infty$ from within the context of \eqref{LIWASe}.

We also define the one-sided derivatives 
\begin{align}\label{define_sided_deriv}
    \frac{\mathrm{d} \pm}{\mathrm{d}x} v^0(x) \coloneqq \lim_{y\to x^{\pm}} \frac{v^0(y)-v^0(x)}{y-x}.
\end{align} 

\begin{claim}
    For $i=1,\ldots,N$, the one-sided derivatives $\frac{\mathrm{d} \pm}{\mathrm{d}x} v^0(x_i)$ exist by the definition of  $\inise$. 
\end{claim}
\begin{claimproof}
    Let $x_i$ be the left (right) endpoint of a small open interval $I\subset\mathbb{R}$. By the definition of  $\inise$, $(v^0)'$ is uniformly continuous on $I$. Thus, the derivative (uniquely) extends to a uniformly continuous function $D\colon \bar{I}\to\mathbb{R}$. Then by L'Hôpital's rule, 
    \begin{align}
        \frac{\mathrm{d} \pm}{\mathrm{d}x} v^0({x_i}) &= \lim_{y\to {x_i}^{\pm}} \frac{v^0(y)-v^0({x_i})}{y-{x_i}}
        =\lim_{y\to {x_i}^{\pm}} \frac{(v^0)'(y)}{1}=D({x_i}),
    \end{align}
    with $+$ for the left endpoint and $-$ for the right endpoint.
\end{claimproof}

Define the slope

\begin{align}\label{slopeM}
    M\coloneqq \max\Bigg\{1,\max_{k=0,\ldots,N} \text{Lip}[v^0 \restriction_{(x_k,x_{k+1})}]\Bigg\}.
\end{align}

If $v^0$ contains at least one discontinuity, we define
\begin{align}
v_1^0(x)\coloneqq
\begin{cases}
v^0(x), &\mbox{ if } x < x_1,\\
v^0(x_1-)+(x-x_1)\frac{\mathrm{d} -}{\mathrm{d}x} v^0(x_1), &\mbox{ if }   x_1 < x < x_I^1,\\
v^0(x_1-)+(x_I^1-x_1)\frac{\mathrm{d} -}{\mathrm{d}x} v^0(x_1)+(x-x_I^1) M , &\mbox{ if }   x_I^1 < x < x_R^1,\\
\max\Big\{v^0(x_1-)+(x_I^1-x_1)\frac{\mathrm{d} -}{\mathrm{d}x}v^0(x_1),\\
\sup_{x\in(x_1,\infty)}v^0(x)+\sum_{k=1}^{N}(v^0(x_k-)-v^0(x_k+))\\+\sum_{k=2}^{N}\Big((x_I^k-x_k)\abs{\frac{\mathrm{d} -}{\mathrm{d}x}v^0(x_k)}\Big)\Big\}, &\mbox{ if }  x_R^1 < x  ,
\end{cases}
\end{align}
where the $x_i$ are as in \eqref{LIWASe}. The number $x_I^1$ is defined as

\begin{align}
    x_I^1\coloneqq  x_1+\min\Bigg\{1,\frac{v^0(x_1-)-v^0(x_1+)}{2\Big(M+\abs{\frac{\mathrm{d} -}{\mathrm{d}x} v^0(x_1)}\Big)}\Bigg\}.
\end{align}

If 
\begin{multline}
\sup_{x\in(x_1,\infty)}v^0(x)+\sum_{k=1}^{N}(v^0(x_k-)-v^0(x_k+))
\\
+\sum_{k=2}^{N}\Big((x_I^k-x_k)\abs{\frac{\mathrm{d} -}{\mathrm{d}x}v^0(x_k)}\Big)< v^0(x_1-)+(x_I^1-x_1)\frac{\mathrm{d} -}{\mathrm{d}x}v^0(x_1)
\end{multline}
then we define $x_R^1\coloneqq x_I^1$. Otherwise, we let $x_R^1$ be defined as the solution to
\begin{multline}
v^0(x_1-)+(x_I^1-x_1)\frac{\mathrm{d} -}{\mathrm{d}x}v^0(x_1)+(x_R^1-x_I^1) M=
\\
\sup_{x\in(x_1,\infty)}v^0(x)+\sum_{k=1}^{N}(v^0(x_k-)-v^0(x_k+))
+\sum_{k=2}^{N}\Big((x_I^k-x_k)\abs{\frac{\mathrm{d} -}{\mathrm{d}x}v^0(x_k)}\Big).
\end{multline}

 Similarly, for $i=2,\ldots, N$, we also define
\begin{align}\label{middle_case}
v_i^0(x)\coloneqq
\begin{cases}
\min\Big\{v^0(x_{i-1}+)+(x_J^i-x_{i-1})\frac{\mathrm{d}+}{\mathrm{d}x}v^0(x_{i-1}),
\\\inf_{x\in(-\infty,x_{i-1})}v^0(x)-\sum_{k=1}^{i-1}(v^0(x_k-)-v^0(x_k+)) \\-\sum_{k=3}^{i} \Big((x_J^{k-1}-x_{k-2})\abs{\frac{\mathrm{d}+}{\mathrm{d}x}v^0(x_{k-2})}\Big)\Big\}, &\mbox{ if }  x < x_L^i,\\
v^0(x_{i-1}+)+(x_J^i-x_{i-1})\frac{\mathrm{d}+}{\mathrm{d}x}v^0(x_{i-1})+(x-x_J^i) M , &\mbox{ if } x_L^i < x <  x_J^i ,\\
v^0(x_{i-1}+)+(x-x_{i-1})\frac{\mathrm{d}+}{\mathrm{d}x}v^0(x_{i-1}), &\mbox{ if } x_J^i < x < x_{i-1},\\
v^0(x) , &\mbox{ if }  x_{i-1}<x<x_i,\\
v^0(x_i-)+(x-x_i) \frac{\mathrm{d}-}{\mathrm{d}x}v^0(x_{i}), &\mbox{ if }  x_i < x < x_I^i,\\
v^0(x_i-)+(x_I^i-x_i) \frac{\mathrm{d}-}{\mathrm{d}x}v^0(x_{i})+(x-x_I^i ) M , &\mbox{ if }  x_I^i < x < x_R^i,\\
\max\Big\{v^0(x_i-)+(x_I^i-x_i) \frac{\mathrm{d}-}{\mathrm{d}x}v^0(x_{i}),\\ \sup_{x\in(x_i,\infty)}v^0(x)+\sum_{k=i}^{N}(v^0(x_k-)-v^0(x_k+))\\+\sum_{k={i+1}}^{N}\Big((x_I^k-x_k)\abs{\frac{\mathrm{d} -}{\mathrm{d}x}v^0(x_k)}\Big)\Big\}, &\mbox{ if }   x_R^i < x.
\end{cases}
\end{align}

We define $\sum_\alpha^\beta \coloneqq 0$, for any indices $\beta<\alpha$.

The number $x_J^i$ is defined as

\begin{align}\label{x_J_def}
    x_J^i\coloneqq  x_{i-1}-\min\Bigg\{1,\frac{v^0(x_{i-1}-)-v^0(x_{i-1}+)}{2\Big(M+\abs{\frac{\mathrm{d} +}{\mathrm{d}x} v^0(x_{i-1})}\Big)}\Bigg\}.
\end{align}

The number $x_I^i$ is defined as

\begin{align}\label{x_I_def}
    x_I^i\coloneqq  x_{i}+\min\Bigg\{1,\frac{v^0(x_{i}-)-v^0(x_{i}+)}{2\Big(M+\abs{\frac{\mathrm{d} -}{\mathrm{d}x} v^0(x_{i})}\Big)}\Bigg\}.
\end{align}

If 
\begin{multline}
\sup_{x\in(x_i,\infty)}v^0(x)+\sum_{k={i}}^{N}(v^0(x_k-)-v^0(x_k+))\\+\sum_{k={i+1}}^{N}\Big((x_I^k-x_k)\abs{\frac{\mathrm{d} -}{\mathrm{d}x}v^0(x_k)}\Big) < v^0(x_i-)+(x_I^i-x_i) \frac{\mathrm{d}-}{\mathrm{d}x}v^0(x_{i})
\end{multline}
then we define $x_R^i\coloneqq x_i$. Otherwise, we let $x_R^i$ be defined as the solution to
\begin{multline}
v^0(x_i-)+(x_I^i-x_i) \frac{\mathrm{d}-}{\mathrm{d}x}v^0(x_{i})+(x_R^i-x_i) M
=
\\
\sup_{x\in(x_i,\infty)}v^0(x)+\sum_{k=i}^{N}(v^0(x_k-)-v^0(x_k+))\\+\sum_{k={i+1}}^{N}\Big((x_I^k-x_k)\abs{\frac{\mathrm{d} -}{\mathrm{d}x}v^0(x_k)}\Big).
\end{multline}

If 
\begin{multline}
\inf_{x\in(-\infty,x_i)}v^0(x)-\sum_{k=1}^{i-1}(v^0(x_k-)-v^0(x_k+))\\-\sum_{k=3}^{i} \Big((x_J^{k-1}-x_{k-2})\abs{\frac{\mathrm{d}+}{\mathrm{d}x}v^0(x_{k-2})}\Big)>v^0(x_{i-1}+)+(x_J^i-x_{i-1})\frac{\mathrm{d}+}{\mathrm{d}x}v^0(x_{i-1})
\end{multline}
then we define $x_L^i\coloneqq x_J^i$. Otherwise, we let $x_L^i$ be defined as the solution to 

\begin{multline}
v^0(x_{i-1}+)+(x_J^i-x_{i-1})\frac{\mathrm{d}+}{\mathrm{d}x}v^0(x_{i-1})+(x_L^i-x_J^i) M
=
\\
\inf_{x\in(-\infty,x_i)}v^0(x)- \sum_{k=1}^{i-1}(v^0(x_k-)-v^0(x_k+))-\sum_{k=3}^{i} \Big((x_J^{k-1}-x_{k-2})\abs{\frac{\mathrm{d}+}{\mathrm{d}x}v^0(x_{k-2})}\Big).
\end{multline}

Finally, we define
\begin{align}
v_{N+1}^0(x)\coloneqq
\begin{cases}
\min\Big\{v^0(x_{N}+)+(x_J^N-x_{N}) \frac{\mathrm{d}+}{\mathrm{d}x}v^0(x_N),\\ \inf_{x\in(-\infty,x_N)}v^0(x)-\sum_{k=1}^{N}(v^0(x_k-)-v^0(x_k+))\\-\sum_{k=3}^{N+1} \Big((x_J^{k-1}-x_{k-2})\abs{\frac{\mathrm{d}+}{\mathrm{d}x}v^0(x_{i-2})}\Big)\Big\}, &\mbox{ if }  x < x_L^N,\\
v^0(x_{N}+)+(x_J^N-x_{N}) \frac{\mathrm{d}+}{\mathrm{d}x}v^0(x_N) +(x-x_J^N) M , &\mbox{ if } x_L^N < x < x_J^N,\\
v^0(x_{N}+)+(x-x_{N}) \frac{\mathrm{d}+}{\mathrm{d}x}v^0(x_N) , &\mbox{ if } x_J^N < x < x_N,\\
v^0(x), &\mbox{ if } x >x_N.
\end{cases}
\end{align}

The number $x_J^N$ is defined as

\begin{align}
    x_J^N\coloneqq  x_{N}-\min\Bigg\{1,\frac{v^0(x_{N}-)-v^0(x_{N}+)}{2\Big(M+\abs{\frac{\mathrm{d} +}{\mathrm{d}x} v^0(x_{N})}\Big)}\Bigg\}.
\end{align}

If 
\begin{multline}
\inf_{x\in(-\infty,x_N)}v^0(x)-\sum_{k=1}^{N}(v^0(x_k-)-v^0(x_k+))\\-\sum_{k=3}^{N+1} \Big((x_J^{k-1}-x_{k-2})\abs{\frac{\mathrm{d}+}{\mathrm{d}x}v^0(x_{i-2})}\Big)>v^0(x_{N}+)+(x_J^N-x_{N}) \frac{\mathrm{d}+}{\mathrm{d}x}v^0(x_N)
\end{multline}
then we define $x_L^N\coloneqq x_J^N$. Otherwise, we let $x_L^N$ be defined as the solution to 

\begin{multline}
v^0(x_{N}+)+(x_J^N-x_{N}) \frac{\mathrm{d}+}{\mathrm{d}x}v^0(x_N)+(x_L^N-x_{N}) M
=
\\
\inf_{x\in(-\infty,x_N)}v^0(x)- \sum_{k=1}^{N}(v^0(x_k-)-v^0(x_k+))\\-\sum_{k=3}^{N+1} \Big((x_J^{k-1}-x_{k-2})\abs{\frac{\mathrm{d}+}{\mathrm{d}x}v^0(x_{i-2})}\Big).
\end{multline}

Note that by construction, the $v_i^0$ are Lipschitz-continuous, satisfy $\inf_{x}\partial_x v_i^0 \geq \inf_{x}\partial_x v^0$, and verify 
\begin{align}\label{ordering_vs}
v_i^0(x)-v_{i+1}^0(x)\geq \frac{1}{2}(v^0(x_i -)-v^0(x_i +))
\end{align}
 for $i=1,\ldots N$ and all $x\in\mathbb{R}$. 
 
Finally, if $v^0$ contains no discontinuities, then $N=0$ and we define $v_1^0\coloneqq v^0$.


\subsection{Proof of \Cref{global_entropy_dissipation_rate_systems}}\label{sec:proof_global_entropy_dissipation_rate_systems}

We now prove \eqref{combined1}.

We perform a similar derivation as outlined in \cite[Lemma 3.2]{scalar_move_entire_solution}.

Note that $\hat{u}_1$ is Lipschitz and so we have the exact equalities,
\begin{align}
\partial_t\big(\hat{u}_1(x,t)\big)+\partial_x\big(A(\hat{u}_1(x,t))\big)&=R_1(x,t),\label{solves_equation_shift}\\
\partial_t\big(\eta(\hat{u}_1(x,t))\big)+\partial_x\big(q(\hat{u}_1(x,t))\big)&=\eta'(\hat{u}_1(x,t))R_1(x,t).\label{solves_entropy_shift}
\end{align}

Recall \eqref{Z_def}, i.e.
$
A(u|\hat{u}_1)\coloneqq A(u)-A(\hat{u}_1)-A' (\hat{u}_1)(u-\hat{u}_1).$
In particular,  $A(u|\hat{u}_1)$ is locally quadratic in $u-\hat{u}_1$. 

Fix any positive, Lipschitz-continuous test function $\phi\colon\mathbb{R}\times[0,T)\to\mathbb{R}$ with compact support. Then, we use that $\hat{u}_2$ satisfies the entropy inequality in a distributional sense:
 \begin{multline}\label{u_entropy_integral_formulation}
\int\limits_{0}^{T} \int\limits_{-\infty}^{\infty}\Bigg[\partial_t\phi\big(\eta(\hat{u}_2(x,t))\big)+\partial_x \phi \big(q(\hat{u}_2(x,t))\big)\Bigg]\,dxdt+ \int\limits_{-\infty}^{\infty}\phi(x,0)\eta(\hat{u}_2^0(x))\,dx
\\ \geq -\int\limits_0^T\int_{-\infty}^{\infty}\phi\eta'(\hat{u}_2(x,t))R_2(x,t)\,dxdt.
\end{multline}

 We also view \eqref{solves_entropy_shift} as a distributional equality:
 \begin{multline}\label{solves_entropy_shift_integral_formulation}
\int\limits_{0}^{T} \int\limits_{-\infty}^{\infty}\Bigg[\partial_t\phi\big(\eta(\hat{u}_1(x,t))\big)+\partial_x \phi \big(q(\hat{u}_1(x,t))\big)\Bigg]\,dxdt+ \int\limits_{-\infty}^{\infty}\phi(x,0)\eta(\hat{u}_1^0(x))\,dx
\\
= -\int\limits_{0}^{T} \int\limits_{-\infty}^{\infty}\phi\eta'(\hat{u}_1(x,t))R_1(x,t)\,dxdt.
\end{multline}

To get \eqref{solves_entropy_shift_integral_formulation}, we do integration by parts on \eqref{solves_entropy_shift}.

 We subtract \eqref{solves_entropy_shift_integral_formulation} from \eqref{u_entropy_integral_formulation}, to get

\begin{equation}
\begin{aligned}\label{difference_entropy_equations}
&\int\limits_{0}^{T} \int\limits_{-\infty}^{\infty} [\partial_t \phi \eta(\hat{u}_2(x,t)|\hat{u}_1(x,t))+\partial_x \phi q(\hat{u}_2(x,t);\hat{u}_1(x,t))]\,dxdt \\
&\hspace{2in}+  \int\limits_{-\infty}^{\infty}\phi(x,0)\eta(\hat{u}_2^0(x)|\hat{u}_1^0(x))\,dx \\
&\hspace{.5in}\geq -\int\limits_{0}^{T} \int\limits_{-\infty}^{\infty} \Big(\partial_t \phi\eta'(\hat{u}_1(x,t))[\hat{u}_2(x,t)-\hat{u}_1(x,t)]
\\
&\hspace{1.5in}+\partial_x \phi \eta'(\hat{u}_1(x,t))[A(\hat{u}_2(x,t))-A(\hat{u}_1(x,t))]\Big)\,dxdt 
\\
&\hspace{.5in}- \int\limits_{-\infty}^{\infty}\phi(x,0)\eta'(\hat{u}_1^0(x))[\hat{u}_2^0(x)-\hat{u}_1^0(x)]\,dx
\\
&\hspace{.5in}+
\int\limits_{0}^{T} \int\limits_{-\infty}^{\infty}\phi\Bigg[\eta'(\hat{u}_1(x,t))R_1(x,t)-\eta'(\hat{u}_2(x,t))R_2(x,t)\Bigg]\,dxdt
\end{aligned}
\end{equation}

The function $\hat{u}_2$ is an entropy solution to the  scalar  conservation law. Thus, for every Lipschitz-continuous test function $\Phi\colon\mathbb{R}\times[0,T)\to \mathbb{R}$ with compact support,
\begin{equation}
    \label{u_solves_equation_integral_formulation}
\int\limits_{0}^{T} \int\limits_{-\infty}^{\infty} \Bigg[\partial_t\Phi \hat{u}_2 + \partial_x\Phi A(\hat{u}_2) \Bigg]\,dxdt +\int\limits_{-\infty}^{\infty} \Phi(x,0)\hat{u}_2^0(x)\,dx
=-\int\limits_{0}^{T} \int\limits_{-\infty}^{\infty}\Phi R_2(x,t)\,dxdt.
\end{equation}

We also can rewrite \eqref{solves_equation_shift} in a distributional way,
\begin{multline}\label{shift_solves_equation_integral_formulation}
\int\limits_{0}^{T} \int\limits_{-\infty}^{\infty} \Bigg[\partial_t\Phi \hat{u}_1(x,t) + \partial_x\Phi A(\hat{u}_1(x,t)) \Bigg]\,dxdt +\int\limits_{-\infty}^{\infty} \Phi(x,0)\hat{u}_1^0(x)\,dx
\\
=-\int\limits_{0}^{T} \int\limits_{-\infty}^{\infty}\Phi R_1(x,t)\,dxdt.
\end{multline}
To prove \eqref{shift_solves_equation_integral_formulation}, we again do integration by parts twice.
Then, we can choose 
\begin{align}\label{our_choice_for_Phi}
\Phi = \phi\eta'(\hat{u}_1)
\end{align} 
 and subtract \eqref{shift_solves_equation_integral_formulation} from \eqref{u_solves_equation_integral_formulation}. 
This yields,
\begin{multline}\label{difference_of_solutions_equation}
\int\limits_{0}^{T} \int\limits_{-\infty}^{\infty} \Bigg[\partial_t[\phi\eta'(\hat{u}_1(x,t))][\hat{u}_2(x,t)- \hat{u}_1(x,t)] 
\\
+\partial_x[\phi\eta'(\hat{u}_1(x,t))][A(\hat{u}_2(x,t))-A(\hat{u}_1(x,t))] \Bigg]\,dxdt 
\\+\int\limits_{-\infty}^{\infty} \phi(x,0)\eta'(\hat{u}_1^0(x))[\hat{u}_2^0(x)-\hat{u}_1^0(x)]\,dx
\\
=\int\limits_{0}^{T} \int\limits_{-\infty}^{\infty}\phi\eta'(\hat{u}_1(x,t)) \Bigg[R_1(x,t)-R_2(x,t)\Bigg]\,dxdt.
\end{multline}

Recall $\hat{u}_1$ is a classical solution and verifies \eqref{solves_equation_shift}. Thus,
\begin{multline}\label{notice_this}
\partial_t\big(\eta'(\hat{u}_1(x,t))\big)=\partial_t\hat{u}_1(x,t)\eta''(\hat{u}_1(x,t))
\\
=\Bigg(-\partial_x \hat{u}_1(x,t)\big[A'(\hat{u}_1(x,t))\big]+R_1(x,t)\Bigg)\eta''(\hat{u}_1(x,t))
\\
=R_1(x,t)\eta''(\hat{u}_1(x,t))-\partial_x \hat{u}_1(x,t)\eta''(\hat{u}_1(x,t))A'(\hat{u}_1(x,t)).
\end{multline}

Thus, by \eqref{notice_this} and the definition of the relative flux in \eqref{Z_def},

\begin{multline}\label{5.2.9}
\partial_t\big(\eta'(\hat{u}_1(x,t))\big)[\hat{u}_2(x,t)-\hat{u}_1(x,t)]
+\partial_x\big(\eta'(\hat{u}_1(x,t))\big)[A(\hat{u}_2(x,t))-A(\hat{u}_1(x,t))]
\\
=
\partial_x \hat{u}_1(x,t)\eta''(\hat{u}_1(x,t)) A(\hat{u}_2(x,t)|\hat{u}_1(x,t))
+R_1(x,t)\eta''(\hat{u}_1(x,t))[\hat{u}_2(x,t)-\hat{u}_1(x,t)].
\end{multline}

We combine \eqref{difference_entropy_equations}, \eqref{difference_of_solutions_equation}, and \eqref{5.2.9} to get \eqref{combined1}.

\bibliographystyle{plain}
\bibliography{references}

\begin{thebibliography}{10}

\bibitem{MR1337114}
Alberto Bressan.
\newblock The unique limit of the {G}limm scheme.
\newblock {\em Arch. Rational Mech. Anal.}, 130(3):205--230, 1995.

\bibitem{Bressan2000}
Alberto {Bressan}.
\newblock {\em {Hyperbolic systems of conservation laws. The one-dimensional
  Cauchy problem}}, volume~20.
\newblock Oxford: Oxford University Press, 2000.

\bibitem{Bressan2020}
Alberto Bressan, Maria~Teresa Chiri, and Wen Shen.
\newblock A posteriori error estimates for numerical solutions to hyperbolic
  conservation laws.
\newblock {\em Arch. Ration. Mech. Anal.}, 241(1):357--402, 2021.

\bibitem{MR1686652}
Alberto Bressan, Graziano Crasta, and Benedetto Piccoli.
\newblock Well-posedness of the {C}auchy problem for {$n\times n$} systems of
  conservation laws.
\newblock {\em Mem. Amer. Math. Soc.}, 146(694):viii+134, 2000.

\bibitem{MR1757395}
Alberto Bressan and Marta Lewicka.
\newblock A uniqueness condition for hyperbolic systems of conservation laws.
\newblock {\em Discrete Contin. Dynam. Systems}, 6(3):673--682, 2000.

\bibitem{MR1723032}
Alberto Bressan, Tai-Ping Liu, and Tong Yang.
\newblock {$L^1$} stability estimates for {$n\times n$} conservation laws.
\newblock {\em Arch. Ration. Mech. Anal.}, 149(1):1--22, 1999.

\bibitem{chen2020uniqueness}
Geng Chen, Sam~G. Krupa, and Alexis~F. Vasseur.
\newblock Uniqueness and weak-{BV} stability for {$2\times 2$} conservation
  laws.
\newblock {\em Arch. Ration. Mech. Anal.}, 246(1):299--332, 2022.

\bibitem{MR3352460}
Elisabetta Chiodaroli, Camillo De~Lellis, and Ond\v{r}ej Kreml.
\newblock Global ill-posedness of the isentropic system of gas dynamics.
\newblock {\em Comm. Pure Appl. Math.}, 68(7):1157--1190, 2015.

\bibitem{Coc03}
Bernardo {Cockburn}.
\newblock {Continuous dependence and error estimation for viscosity methods}.
\newblock {\em {Acta Numerica}}, 12:127--180, 2003.

\bibitem{2020arXiv201103710C}
Andres~A. Contreras~Hip and Xavier Lamy.
\newblock On the {$L^2$} stability of shock waves for finite-entropy solutions
  of {B}urgers.
\newblock {\em J. Differential Equations}, 301:236--265, 2021.

\bibitem{MR546634}
Constantine~M. Dafermos.
\newblock The second law of thermodynamics and stability.
\newblock {\em Arch. Rational Mech. Anal.}, 70(2):167--179, 1979.

\bibitem{doi:10.1080/01495737908962394}
Constantine~M. Dafermos.
\newblock Stability of motions of thermoelastic fluids.
\newblock {\em Journal of Thermal Stresses}, 2(1):127--134, 1979.

\bibitem{dafermos_big_book}
Constantine~M. Dafermos.
\newblock {\em Hyperbolic conservation laws in continuum physics}, volume 325
  of {\em Grundlehren der Mathematischen Wissenschaften [Fundamental Principles
  of Mathematical Sciences]}.
\newblock Springer-Verlag, Berlin, fourth edition, 2016.

\bibitem{MR4627977}
Constantine~M. Dafermos.
\newblock Hyperbolic {C}onservation {L}aws: {P}ast, {P}resent, {F}uture.
\newblock In {\em Mathematics going forward---collected mathematical
  brushstrokes}, volume 2313 of {\em Lecture Notes in Math.}, pages 479--486.
  Springer, Cham, 2023.

\bibitem{delellis_uniquneness}
Camillo De~Lellis, Felix Otto, and Michael Westdickenberg.
\newblock Minimal entropy conditions for {B}urgers equation.
\newblock {\em Quart. Appl. Math.}, 62(4):687--700, 2004.

\bibitem{MR2564474}
Camillo De~Lellis and L\'{a}szl\'{o} Sz\'{e}kelyhidi, Jr.
\newblock On admissibility criteria for weak solutions of the {E}uler
  equations.
\newblock {\em Arch. Ration. Mech. Anal.}, 195(1):225--260, 2010.

\bibitem{DednerGiesselmann}
Andreas {Dedner} and Jan {Giesselmann}.
\newblock {A posteriori analysis of fully discrete method of lines
  discontinuous Galerkin schemes for systems of conservation laws}.
\newblock {\em {SIAM J. Numer. Anal.}}, 54(6):3523--3549, 2016.

\bibitem{DMO07}
Andreas {Dedner}, Charalambos {Makridakis}, and Mario {Ohlberger}.
\newblock {Error control for a class of Runge-Kutta discontinuous Galerkin
  methods for nonlinear conservation laws}.
\newblock {\em {SIAM J. Numer. Anal.}}, 45(2):514--538, 2007.

\bibitem{MR523630}
Ronald~J. DiPerna.
\newblock Uniqueness of solutions to hyperbolic conservation laws.
\newblock {\em Indiana Univ. Math. J.}, 28(1):137--188, 1979.

\bibitem{MR2597943}
Lawrence~C. Evans.
\newblock {\em Partial differential equations}, volume~19 of {\em Graduate
  Studies in Mathematics}.
\newblock American Mathematical Society, Providence, RI, second edition, 2010.

\bibitem{FLM20}
Eduard {Feireisl}, M\'aria {Luk\'a\v{c}ov\'a-Medvid'ov\'a}, and Hana
  {Mizerov\'a}.
\newblock {Convergence of finite volume schemes for the Euler equations via
  dissipative measure-valued solutions}.
\newblock {\em {Found. Comput. Math.}}, 20(4):923--966, 2020.

\bibitem{FLMW20}
Ulrik~Skre {Fjordholm}, Kjetil~Olsen {Lye}, Siddhartha {Mishra}, and Franziska
  {Weber}.
\newblock {Statistical solutions of hyperbolic systems of conservation laws:
  numerical approximation}.
\newblock {\em {Math. Models Methods Appl. Sci.}}, 30(3):539--609, 2020.

\bibitem{GiesselmannMakridakisPryer}
Jan {Giesselmann}, Charalambos {Makridakis}, and Tristan {Pryer}.
\newblock {A posteriori analysis of discontinuous Galerkin schemes for systems
  of hyperbolic conservation laws}.
\newblock {\em {SIAM J. Numer. Anal.}}, 53(3):1280--1303, 2015.

\bibitem{GiesselmannSikstel2023}
Jan {Giesselmann} and Aleksey {Sikstel}.
\newblock {A-posteriori error estimates for systems of hyperbolic conservation
  laws}.
\newblock {\em arXiv e-prints}, page arXiv:2305.01340, May 2023.

\bibitem{GR96}
Edwige {Godlewski} and Pierre-Arnaud {Raviart}.
\newblock {\em {Numerical approximation of hyperbolic systems of conservation
  laws}}, volume 118.
\newblock New York, NY: Springer, 1996.

\bibitem{golding2022unconditional}
William {Golding}.
\newblock {Unconditional regularity and trace results for the isentropic Euler
  equations with $\gamma = 3$}.
\newblock {\em arXiv e-prints}, page arXiv:2207.05821, July 2022.

\bibitem{2020arXiv201209261G}
William {Golding}, Sam~G. {Krupa}, and Alexis~F. {Vasseur}.
\newblock {Sharp a-contraction estimates for small extremal shocks}.
\newblock {\em arXiv e-prints}, page arXiv:2012.09261, December 2020.
\newblock To appear in Journal of Hyperbolic Differential Equations.

\bibitem{GM00}
Laurent {Gosse} and Charalambos {Makridakis}.
\newblock {Two a posteriori error estimates for one-dimensional scalar
  conservation laws}.
\newblock {\em {SIAM J. Numer. Anal.}}, 38(3):964--988, 2000.

\bibitem{HartmannHouston2003}
Ralf Hartmann and Paul Houston.
\newblock Adaptive discontinuous {G}alerkin finite element methods for the
  compressible {E}uler equations.
\newblock {\em J. Comput. Phys.}, 183(2):508--532, 2002.

\bibitem{Hesthaven18}
Jan~S. {Hesthaven}.
\newblock {\em {Numerical methods for conservation laws. From analysis to
  algorithms}}, volume~18.
\newblock Philadelphia, PA: Society for Industrial and Applied Mathematics
  (SIAM), 2018.

\bibitem{doi:10.1142/S021919972250081X}
Carl Johan~Peter Johansson and Riccardo Tione.
\newblock {$T_5$} configurations and hyperbolic systems.
\newblock {\em Communications in Contemporary Mathematics}, 0(0):2250081, 2023.

\bibitem{MR3519973}
Moon-Jin Kang and Alexis~F. Vasseur.
\newblock Criteria on contractions for entropic discontinuities of systems of
  conservation laws.
\newblock {\em Arch. Ration. Mech. Anal.}, 222(1):343--391, 2016.

\bibitem{Kroner97}
Dietmar {Kr\"oner}.
\newblock {\em {Numerical schemes for conservation laws}}.
\newblock Chichester: Wiley; Stuttgart: Teubner, 1997.

\bibitem{KO00}
Dietmar {Kr\"oner} and Mario {Ohlberger}.
\newblock {A posteriori error estimates for upwind finite volume schemes for
  nonlinear conservation laws in multi dimensions}.
\newblock {\em {Math. Comput.}}, 69(229):25--39, 2000.

\bibitem{move_entire_solution_system}
Sam~G. Krupa.
\newblock Criteria for the a-contraction and stability for the piecewise-smooth
  solutions to hyperbolic balance laws.
\newblock {\em Commun. Math. Sci.}, 18(6):1493--1537, 2020.

\bibitem{krupa2022nonexistence}
Sam~G. {Krupa} and {L{\'a}szl{\'o} Sz{\'e}kelyhidi, Jr.}
\newblock {Nonexistence of $T_4$ configurations for hyperbolic systems and the
  Liu entropy condition}.
\newblock {\em arXiv e-prints}, page arXiv:2211.14239, November 2022.

\bibitem{2017arXiv170905610K}
Sam~G. Krupa and Alexis~F. Vasseur.
\newblock On uniqueness of solutions to conservation laws verifying a single
  entropy condition.
\newblock {\em J. Hyperbolic Differ. Equ.}, 16(1):157--191, 2019.

\bibitem{scalar_move_entire_solution}
Sam~G. Krupa and Alexis~F. Vasseur.
\newblock Stability and uniqueness for piecewise smooth solutions to a nonlocal
  scalar conservation law with applications to {B}urgers-{H}ilbert equation.
\newblock {\em SIAM J. Math. Anal.}, 52(3):2491--2530, 2020.

\bibitem{MR0267257}
Stanislav~Nikolaevich Kruzhkov.
\newblock First order quasilinear equations in several independent variables.
\newblock {\em Mat. Sb. (N.S.)}, 81 (123):228--255, 1970.
\newblock English transl. in \textit{Math. USSR, Sb.}, 10:217--243, 1970.

\bibitem{Laforest04}
Marc {Laforest}.
\newblock {A posteriori error estimate for front-tracking: Systems of
  conservation laws}.
\newblock {\em {SIAM J. Math. Anal.}}, 35(5):1347--1370, 2004.

\bibitem{Laforest08}
Marc {Laforest}.
\newblock {An a posteriori error estimate for Glimm's scheme}.
\newblock In {\em {Hyperbolic problems. Theory, numerics and applications.
  Proceedings of the 11th international conference on hyperbolic problems,
  Ecole Normale Sup\'erieure, Lyon, France, July 17--21, 2006}}, pages
  643--651. Berlin: Springer, 2008.

\bibitem{lefloch_book}
Philippe~G. LeFloch.
\newblock {\em Hyperbolic systems of conservation laws}.
\newblock Lectures in Mathematics ETH Z\"{u}rich. Birkh\"{a}user Verlag, Basel,
  2002.
\newblock The theory of classical and nonclassical shock waves.

\bibitem{Leger2011}
Nicholas Leger and Alexis~F. Vasseur.
\newblock Relative entropy and the stability of shocks and contact
  discontinuities for systems of conservation laws with non-{BV} perturbations.
\newblock {\em Archive for Rational Mechanics and Analysis}, 201(1):271--302,
  2011.

\bibitem{LeVeque02}
Randall~J. {LeVeque}.
\newblock {\em {Finite volume methods for hyperbolic problems}}.
\newblock Cambridge: Cambridge University Press, 2002.

\bibitem{MR4144350}
Andrew Lorent and Guanying Peng.
\newblock On the rank-1 convex hull of a set arising from a hyperbolic system
  of {L}agrangian elasticity.
\newblock {\em Calc. Var. Partial Differential Equations}, 59(5):Paper No. 156,
  36, 2020.

\bibitem{Mak07}
Charalambos {Makridakis}.
\newblock {Space and time reconstructions in a posteriori analysis of evolution
  problems}.
\newblock {\em {ESAIM, Proc.}}, 21:31--44, 2007.

\bibitem{MR0094541}
Ol'ga~Arsen'evna Ole\u{\i}nik.
\newblock Discontinuous solutions of non-linear differential equations.
\newblock {\em Uspehi Mat. Nauk (N.S.)}, 12(3(75)):3--73, 1957.
\newblock English transl. in \textit{Amer. Math. Soc. Transl. (2)}, 26:95--172,
  1963.

\bibitem{panov_uniquness}
Evgueni~Yu. Panov.
\newblock Uniqueness of the solution of the {C}auchy problem for a first order
  quasilinear equation with one admissible strictly convex entropy.
\newblock {\em Mat. Zametki}, 55(5):116--129, 159, 1994.
\newblock English transl. in \textit{Mathematical Notes}, 55(5):517--525, 1994.

\bibitem{serre_book}
Denis Serre.
\newblock {\em Systems of conservation laws. 1}.
\newblock Cambridge University Press, Cambridge, 1999.
\newblock Hyperbolicity, entropies, shock waves. Translated from the 1996
  French original by I. N. Sneddon.

\bibitem{serre_vasseur}
Denis Serre and Alexis~F. Vasseur.
\newblock {$L^2$}-type contraction for systems of conservation laws.
\newblock {\em J. \'Ec. polytech. Math.}, 1:1--28, 2014.

\bibitem{Ulbrich2002}
Stefan Ulbrich.
\newblock A sensitivity and adjoint calculus for discontinuous solutions of
  hyperbolic conservation laws with source terms.
\newblock {\em SIAM J. Control Optim.}, 41(3):740--797, 2002.

\bibitem{MR1720782}
Alexis~F. Vasseur.
\newblock Time regularity for the system of isentropic gas dynamics with
  {$\gamma=3$}.
\newblock {\em Comm. Partial Differential Equations}, 24(11-12):1987--1997,
  1999.

\bibitem{VASSEUR2008323}
Alexis~F. Vasseur.
\newblock Recent results on hydrodynamic limits.
\newblock {\em Handbook of Differential Equations: Evolutionary Equations},
  4:323 -- 376, 2008.

\bibitem{ZhangShu}
Qiang Zhang and Chi-Wang Shu.
\newblock Error estimates to smooth solutions of {Runge}-{Kutta} discontinuous
  {Galerkin} method for symmetrizable systems of conservation laws.
\newblock {\em SIAM J. Numer. Anal.}, 44(4):1703--1720, 2006.

\end{thebibliography}
\end{document}